\documentclass[10pt,a4paper]{amsart}
\usepackage{amssymb, latexsym, amsmath, amsfonts, amsrefs}
\usepackage[mathscr]{eucal}

\newtheorem{thm}{Theorem}[section]
\newtheorem{cor}[thm]{Corollary}
\newtheorem{lem}[thm]{Lemma}
\newtheorem{prop}[thm]{Proposition}
\theoremstyle{definition}

\theoremstyle{remark}
\newtheorem{rem}[thm]{Remark}
\numberwithin{equation}{section}

\newcommand{\ra}{\rightarrow}
\newcommand{\z}{\zeta}
\newcommand{\pa}{\partial}
\newcommand{\ov}{\overline}

\newcommand{\ep}{\epsilon}
\newcommand{\no}{\noindent}
\newcommand{\Om}{\Omega}

\newcommand{\ti}{\tilde}
\newcommand{\la}{\lambda}

\newcommand{\La}{\Lambda}
\newcommand{\al}{\alpha}
\newcommand{\be}{\beta}
\newcommand{\ga}{\gamma}
\newcommand{\de}{\delta}
\newcommand{\De}{\Delta}

\newcommand{\om}{\omega}

\title{Remarks on the metric induced by the Robin function II}
\author{Diganta Borah}
\address{Indian Institute of Science Education and Research, Pune, India}
\email{dborah@iiserpune.ac.in}

\begin{document}

\begin{abstract}
Let $D$ be a smoothly bounded pseudoconvex domain in $\mathbf C^n$, $n > 1$.
Using the Robin function $\La(p)$ that arises from the Green function $G(z,
p)$ for $D$ with pole at $p \in D$ associated with
the standard sum-of-squares Laplacian, N. Levenberg and H. Yamaguchi had
constructed a K\"{a}hler metric (the so-called
$\La$-metric) on $D$. Assume that $D$ is strongly pseudoconvex and $ds^2$
denotes the
$\La$-metric on $D$. In this article, first we prove that the holomorphic
sectional curvature of $ds^2$ along normal directions converges to a negative
constant near the boundary of $D$. Then, we prove that if $D$ is not simply
connected, then any nontrivial homotopy class of $\pi_1(D)$ contains a closed
geodesic for $ds^2$. Finally, we prove that the diminesion of the space of
square integrable harmonic $(p, q)$-forms on $D$ relative to $ds^2$ is zero
except when $p+q=n$ in which case it is infinite.

\end{abstract}

\maketitle
\section{Introduction}
\no Let $D$ be a $C^{\infty}$-smoothly bounded domain in $\mathbf{C}^{n}$ ($n
\geq 2$). For $p \in D$, let $G(z, p)$ be the Green function for $D$ with pole
at $p$ associated to the standard Laplacian
\[
\Delta = 4 \sum_{i = 1}^{n} \frac{\pa^2}{\pa z_i \pa \ov z_i}
\]
on $\mathbf{C}^{n} \approx \mathbf{R}^{2n}$. Then $G(z, p)$ is the unique
function of $z \in D$ satisfying $G(z, p)$ is harmonic on $D \setminus \{p\}$,
$G(z, p) \ra 0$ as $z \ra \pa D$ and $ G(z, p) - \vert z - p \vert ^{-2n + 2}$
is harmonic near $p$. Thus
\[
\La(p) = \lim_{z \ra p} \big( G(z, p) - \vert z - p \vert ^{-2n + 2} \big)
\]
exists and is called the Robin constant for $D$ at $p$. The function
\[
\Lambda : p \ra \Lambda(p)
\]
is called the Robin function for $D$.

\medskip

The Robin function for $D$ is negative, real analytic and tends to $- \infty$
near $\pa D$ (see \cite{Y}). Further, if $D$ is pseudoconvex then by a result of
Levenberg-Yamaguchi (\cite{LY}), $\log(-\Lambda)$ is a strongly plurisubharmonic
function on $D$. Therefore
\[
ds^2 = \sum_{\al, \be =1}^{n} \frac{\pa ^2 \log(-\Lambda)}{\pa z_{\al} \pa \ov
z_{\be}} dz_{\al} \otimes d\ov z_{\be}
\]
is a K\"{a}hler metric on $D$ which is called the $\Lambda$-metric. Recall that
the holomorphic
sectional curvature of $ds^2$ at $z \in D$ along the direction $v \in
\mathbf{C}^{n}$ is given by
\[
R(z, v) = \frac{R_{\al \ov \be \ga \ov \de} v^{\al} \ov v^{\be} v^{\ga} \ov
v^{\de}}{g_{\al \ov \be} v^{\al} \ov v^{\be}}
\]
where
\[
R_{\al \ov \be \ga \ov \de} = - \frac{\pa^2 g_{\al \ov \be}}{\pa z_{\ga} \pa \ov
z_{\de}} + g^{\nu \ov \mu } \frac{\pa g_{\al \ov \mu}}{\pa z_{\ga}} \frac{\pa
g_{\nu \ov \be}}{\pa \ov z_{\de}}
\]
are the components of the curvature tensor,
\[
g_{\al \ov \be} = \frac{\pa^2 \log(-\Lambda)}{\pa z_{\al} \pa \ov z_{\be}}
 \]
are the components of $ds^2$ and $g^{\al \ov \be}$ are the entries of the matrix
$(g_{\al \ov \be})^{-1}$. In the above formulae, the standard convention of
summing
over all indices that appear once in the upper and lower position is being
followed.

\medskip

Now, let $v$ be a vector in $\mathbf{C}^{n}$. At each point $z \in \pa D$, there
is a canonical splitting $\mathbf{C}^{n} = H_z(\pa D) \oplus N_z(\pa D)$ along
the complex tangential and normal directions at $z$ and hence $v$ can uniquely
be written as $v = v_H(z) + v_N(z)$ where $v_H(z) \in H_z(\pa D)$ and $v_N(z)
\in N_z(\pa D)$. Also, the smoothness of $\pa D$ implies that if $z \in D$ is
sufficiently close to $\pa D$, then there is a unique point $\pi(z) \in \pa D$
that is closest to it, i.e., $d(z, \pa D) = \vert z - \pi(z) \vert$. Therefore,
$v$ can uniquely be written as $v = v_H(\pi(z)) + v_N(\pi(z))$. We will
abbreviate $v_H(\pi(z))$ as $v_H(z)$ and $v_N(\pi(z))$ as $v_N(z)$.
For a strongly pseudonconvex domain $D$, the boundary behaviour of $R(z,
v_N(z))$ was calculated in \cite{BV} in a special case, viz., when $z \ra z_0
\in \pa D$ along the inner normal to $\pa D$ at $z_0$. The purpose
of this article is threefold. One, we remove the restriction that $z \ra z_0$
along the inner normal in obtaining the boundary behaviour of $R(z, v_N(z))$.
More precisely, we have the following:
\begin{thm}
Let $D$ be a $C^{\infty}$-smoothly bounded strongly pseudoconvex domain in
$\mathbf{C}^{n}$. Fix $z_0 \in \pa D$ and let $v \in \mathbf{C}^{n}$. Then for
$z \in D$
\[
\lim_{z \ra z_0} R\big(z, v_{N}(z) \big) = -\frac{1}{n-1}.
\]
\end{thm}
\noindent To understand the difficulty in the computation, let us assume without
loss of generality that $z_0 = 0$ and the normal to $\pa D$ at $z_0$ is along
the $\Re{z_n}$-axis. Let $\{z_{\nu}\}$ be a sequence of points in
$D$ converging to $0$.  Without loss of generality, let us assume that the
distance between
$z_{\nu}$ and $\pa D$, denoted by $\delta_{\nu}$, is realised by a unique point
$\pi(z_{\nu}) \in \pa D$,
i.e.,
\[
\delta_{\nu} =d(z_{\nu}, \pa D) = \vert z_{\nu} - \pi(z_{\nu}) \vert
\]
for all $\nu \geq 1$. Now for each $\nu$, apply a translation $\tau_{\nu}$ to
$D$
followed by a unitary rotation $\sigma_{\nu}$ to obtain a new domain $D_{\nu}$
so
that $\pi(z_{\nu}) \in \pa D$ corresponds to $0 \in \pa D_{\nu}$ and the normal
to $\pa D_{\nu}$ at $0$ is along the $\Re{z_n}$ axis. We will denote the
composition $\sigma_{\nu} \circ \tau_{\nu}$ by $\theta_{\nu}$. Note that under
the map $\theta_{\nu}$, $z_{\nu} \in
D$ corresponds to $p_{\nu} = (0, \ldots, -\delta_{\nu}) \in D_{\nu}$ and
\[
\theta_{\nu}^{\prime}(z_{\nu}) v_N(z_{\nu}) =\big(0, \ldots, 0, \vert
v_{N}(z_{\nu}) \vert\big).
\]
Therefore, by the invariance of the $\La$-metric under translation and unitary
rotation \cite{BV}*{lemma~5.1},
\begin{multline}\label{cvr}
R_{D}\big(z_{\nu}, v_{N}(z_{\nu})\big) = R_{D_{\nu}}\Big(p_{\nu}, \big(0,
\ldots, 0, \vert
v_{N}(z_{\nu}) \vert\big)\Big)\\
=\frac{1}{\big(g_{\nu n \ov n}(p_{\nu}) \big)^{2}} \bigg(-\frac{\pa^2
g_{\nu n \ov n}}{\pa z_n \pa \ov z_n}(p_{\nu}) + \sum_{\al, \be=1}^{n}
g_{\nu}^{\beta \ov \al}(p_{\nu}) \frac{\pa g_{\nu n \ov \al}}{\pa
z_n}(p_{\nu}) \frac{\pa g_{\nu \beta \ov n}}{\pa \ov z_n}\bigg)
\end{multline}
where
\begin{equation}\label{ds_nu-cmp}
g_{\nu \al \ov \be} = \frac{\pa^2 \log(-\Lambda_{\nu})}{\pa z_{\al} \pa \ov
z_{\be}}
\end{equation}
are the components of the $\Lambda$-metric $ds^2_{\nu}$ on $D_{\nu}$ and
$g_{\nu}^{\al \ov \be}$ are the entries of the matrix $\big( g_{\nu \al \ov \be}
\big)^{-1}$. To compute the limit of the right hand side of (\ref{cvr}) we have
to find the asymptotics of the metric components $g_{\nu \al \ov \be}$ and their
derivatives along the sequence $\{p_{\nu}\}$. From (\ref{ds_nu-cmp}), it is
natural to hope that this can be achieved by computing the asymptotics of
$\Lambda_{\nu}$ and their derivatives along $\{p_{\nu}\}$. To be more precise,
let
$\psi$ be a $C^{\infty}$-smooth function on $\mathbf{C}^{n}$ that defines the
domain $D$ and $ \ov \pa \psi(0) = (0, \ldots, 1)$. Then for each $\nu \geq 1$,
$\psi_{\nu} = \psi \circ \theta_{\nu}^{-1}$ is a $C^{\infty}$-smooth
defining function for $D_{\nu}$. Also, it is evident that $\{\psi_{\nu}\}$
converges in the $C^{\infty}$-topology on compact subsets of $\mathbf{C}^{n}$ to
$\psi$. We then want to compute the rate of growth of
\[
D^{A \ov B} \Lambda_{\nu}(p_{\nu}) = \frac{\pa^{\vert A \vert + \vert B \vert}
\Lambda_{\nu}}{\pa z_1^{\al_1} \cdots \pa z_n^{\al_n} \pa \ov z_1^{\be_1} \cdots
\pa \ov z_n^{\be_n}}(p_{\nu}), \quad A = (\al_1, \ldots, \al_n), B = (\be_1,
\ldots, \be_n) \in \mathbf{N}^{n}
\]
in terms of $\psi_{\nu}(p_{\nu})$. In this regard, we prove the following:

\begin{thm}
Let $D$ be a $C^{\infty}$-smoothly bounded domain in $\mathbf{C}^{n}$ and let
$\psi$
be a $C^{\infty}$-smooth defining function for $D$ defined on all of
$\mathbf{C}^{n}$. Let $\{D_{\nu}\}$ be a sequence of $C^{\infty}$-smoothly
bounded domains in $\mathbf{C}^{n}$ with defining functions $\psi_{\nu}$ that
converge in the $C^{\infty}$-topology on compact subsets of
$\mathbf{C}^{n}$ to $\psi$. Let $p_{\nu} \in D_{\nu}$ be such that $\{p_{\nu}\}$
converges to $p_0 \in \pa D$. Define the half space
\[
\mathcal{H} = \Big\{ w \in \mathbf{C}^{n} : 2 \Re \Big( \sum_{\al=1}^{n}
\psi_{\al}(p_0) w_{\al} \Big) -1 < 0 \Big\}
\]
and let $\Lambda_{\mathcal{H}}$ denotes the Robin function for $\mathcal{H}$.
Then
\[
(-1)^{\vert A \vert + \vert B \vert} D^{A \ov B} \Lambda_{\nu} (p_{\nu}) \big(
\psi_{\nu}(p_{\nu})\big)^{2n-2+ \vert A \vert + \vert B \vert} \rightarrow D^{A
\ov B} \Lambda_{H}(p_0)
\]
as $\nu \ra \infty$.
\end{thm}

We will show in section 6 that the asymptotics obtained in the above theorem
suffice
to calculate the limit of the first term of (\ref{cvr}). However it turns out
that the second term remains
indeterminate by these asymptotics. This means that in order to calculate this
term we
need finer asymptotics of $\La_{\nu}$ and their derivatives. A similar situation
was handled in \cite{BV} by using the following result of Levenberg-Yamaguchi
\cite{LY}: The function $\la$ defined by
\begin{equation}
\la(p) =
\begin{cases}
\Lambda(p) \big(\psi(p)\big)^{2n-2} & ; \quad \text{if } p \in D\\
- \vert \pa \psi(p) \vert^{2n-2} & ; \quad \text{if } p \in \pa D
\end{cases}
\end{equation}
is $C^2$ up to $\ov D$.  We will call $\la$ the normalised Robin
function
associated to $(D, \psi)$. Thus it is expected that finer asymptotics
of $\La_{\nu}$ and their derivatives along $\{p_{\nu}\}$ could be obtained if
the functions $\la_{\nu} = \La_{\nu} \psi_{\nu}$ and their derivatives along
$\{p_{\nu}\}$ are bounded. While $\la_{\nu}(p_{\nu})$ converge to $\la(p_0)$ by
theorem~1.2,  we establish the convergence of first and second derivatives of
$\la_{\nu}$ along $\{p_{\nu}\}$ in the following:

\begin{thm}
Under the hypothesis of theorem~1.2, we have
\begin{enumerate}
\item $\displaystyle \lim_{\nu \ra \infty} \frac{\pa \la_{\nu}}{\pa
p_{\al}}(p_{\nu}) = \frac{\pa \la}{\pa p_{\al}}(p_0)$, and

\item $\displaystyle \lim_{\nu \ra \infty} \frac{\pa^2 \la_{\nu}}{\pa p_{\al}
\pa \ov
p_{\be}}(p_{\nu}) = \frac{\pa^2 \la}{\pa p_{\al} \pa \ov p_{\be}}(p_0)$.
\end{enumerate}
where $\la$ is the normalised Robin function associated to $(D, \psi)$ and
$\la_{\nu}$ is the normalised Robin function associated to $(D_{\nu},
\psi_{\nu})$.
\end{thm}

\medskip

We remark that unlike the Bergman, Carath\'{e}odory and Kobayashi metrics, the
$\La$-metric is not invariant under biholomorphisms in general. For an example
we refer to \cite{BV}. The only information in this respect that we have
is that any biholomorphism between two $C^{\infty}$-smoothly bounded strongly
pseudoconvex domains is Lipschitz with respect to the $\La$-metric. This follows
from \cite{BV}*{theorem~1.4}.
Despite this drawback, we put our effort to explore this
metric by finding its various properties analogous to those possessed by these
invariant metrics. The second theme of this article is to study the existence of
closed geodesics for the $\La$-metric of a given homotopy type. In \cite{Her},
Herbort proved that on a $C^{\infty}$-smoothly bounded strongly pseudoconvex
domain $D$ in $\mathbf{C}^{n}$ which is not simply connected, every nontrivial
homotopy class in $\pi_1(D)$ contains a closed geodesic for the Bergman metric.
Using the asymptotics of the $\La$-metric derived in \cite{BV} we prove the
following analogue for the $\La$-metric:

\begin{thm}
Let $D$ be a $C^{\infty}$-smoothly bounded strongly pseudoconvex domain in
$\mathbf{C}^{n}$ which is not simply connected. Then every nontrivial homotopy
class in $\pi_1(D)$ contains a closed geodesic for the $\La$-metric.
\end{thm}

\medskip

Let $D$ be $C^{\infty}$-smoothly bounded strongly pseudoconvex domain in
$\mathbf{C}^{n}$. H.~Donnelly and C. Fefferman \cite{DonFef} proved that $D$
does not admit any square integrable harmonic $(p, q)$-form relative to the
Bergman metric except when $p + q =n$, in which case the space of such forms is
infinite dimensional. A more transparent and elementary proof of the infinite
dimensionality of the $L^2$-cohomology of the middle dimension was given by
Ohsawa \cite{Ohs89}. In \cite{Don}, Donnelly gave an alternative proof of the
vanishing
of the $L^2$-cohomology outside the middle dimension using the following
observation of Gromov \cite{Gro}: If $M$ is a
complete K\"{a}hler manifold of complex dimension $n$ such that the
K\"{a}hler form $\om$ of $M$ can be written as $\om = d\eta$, where
$\eta$ is bounded in supremum norm, then $M$ does not admit any square
integrable harmonic $i$ form for $i \neq n$. Finally, we observe that these
ideas can be applied to the $\La$-metric to prove the following:

\begin{thm}
Let $D$ be a $C^{\infty}$-smoothly bounded strongly pseudoconvex domain in
$\mathbf{C}^{n}$. Let $\mathcal{H}_2^{p,q}(D)$ be the space of square
integrable harmonic $(p, q)$-forms relative to the $\La$-metric. Then
\[
\dim\mathcal{H}_2^{p,q}(D) =
\begin{cases}
0 & ; \quad \text{if } p+q \neq n,\\
\infty & ; \quad \text{if } p+q = n.
\end{cases}
\]
\end{thm}

\medskip

\no {\it Acknowledgements :} The author is indebted to K.~Verma for his
encouragement,
precious comments
and various helpful clarifications during the course of this work.

\section{Properties of $\la$}\label{relation-la-g}
\noindent Let $D$ be a $C^{\infty}$-smoothly bounded domain in $\mathbf{C}^{n}$
with a $C^{\infty}$-smooth defining function $\psi$ defined on all of
$\mathbf{C}^{n}$. In this section, we recall some basic properties of the
normalised Robin function $\la$ associated to $(D, \psi)$. We start by
describing the geometric meaning of $\la(p)$. Given $p \in D $, let
\[
T : D \times \mathbf{C}^{n} \ra \mathbf{C}^{n}
\]
be the map defined by
\begin{equation}
T(p,z) = \frac{z - p}{- \psi(p)}.
\end{equation}
Set
\begin{equation}\label{def-variation}
D(p) =
\begin{cases}
T(p, D) & ; \quad \text{if } p \in D,\\
\Big \{ w \in \mathbf{C}^{n} : 2 \Re \big( \sum_{\al = 1}^{n} \psi_{\al}(p)
w_{\al} \big) - 1 < 0 \Big \} & ; \quad \text{if } p \in \pa D.
\end{cases}
\end{equation}
Thus $\{D(p) : p \in \ov D\}$ is a family of domains in $\mathbf{C}^{n}$ each
containing the origin. When $p \in D$, $D(p)$ is the image of $D$ under the
affine transformation $T(p, \cdot)$ and hence by
\cite{Y}*{proposition~5.1}, we have
\[
\La_{D(p)}(0) = \La(p) \big( \psi(p) \big)^{2n -2}= \la(p).
\]
When $p \in \pa D$, $D(p)$ is a half space for which we have the explicit
formula \cite{BV}*{(1.4)}
\[
\La_{D(p)}(0) = - \big \vert \pa \psi(p) \big \vert ^{2n - 2} = \la(p).
\]
Thus for each $p \in \ov D$, $\la(p)$ is the Robin constant for $D(p)$
at the origin. We will denote the Green function for $D(p)$ with pole at $p$ by
$g(p, w)$.

To discuss the regularity of the function $\la(p)$ on $\ov D$,  we set
\[
\mathscr{D} = \cup_{p \in D} \big( p, D(p) \big) = \big\{(p, w) : p \in D, w \in
D(p) \big\} .
\]
The set $\mathscr{D}$ can be considered as a variation of domains in
$\mathbf{C}^{n}$ with parameter space $D$, i.e., as a map
\[
\mathscr{D} : p \ra D(p)
\]
which associates to each $p \in D$ a domain $D(p) \subset
\mathbf{C}^{n}$. We call $\mathscr{D}: p \ra D(p)$ the variation associated to
$(D, \psi)$. The following function
\begin{equation}\label{def-defining-fn}
f(p, w) = 2 \Re \bigg \{ \sum_{\al = 1}^{n} \int _{0} ^{1} \Big( w_{\al}
\psi_{\al} \big(p - \psi(p) t
w \big) \Big) dt \bigg \} - 1
\end{equation}
was constructed in \cite {LY} which is jointly smooth on $\mathbf{C}^{n} \times
\mathbf{C}^{n}$ and satisfies,
taking $\ti{\mathscr{D}} = D \times \mathbf{C}^{n}$,
\begin{enumerate}
\item [(i)] $\mathscr{D} = \{(p, w) \in \ti{\mathscr{D}} : f(p, w) < 0 \}$, $\pa
\mathscr{D} :=
\big\{ (p, w) : p \in D, w \in \pa D(p) \big\} = \{(p, w) \in \ti{\mathscr{D}}:
f(p, w) = 0 \}$
and $\text{Grad}_{(p, w)} f \neq 0$ on $\pa \mathscr{D}$,

\item [(ii)] For each $p \in D$, $D(p) = \{w \in \mathbf{C}^{n} : f(p, w) < 0
\}$,
$\pa D(p) = \{w \in \mathbf{C}^{n} : f(p, w) = 0 \}$ and $\text{Grad}_{w}f(p, w)
\neq 0$ on $\pa
D(p)$.
\end{enumerate}
Therefore, we say that the variation $\mathscr{D}: p \ra D(p)$ is smooth and is
defined by $f(p, w)$. It is evident that the variation
\[
\mathscr{D} \cup \pa \mathscr{D} : p \ra D(p) \cup \pa D(p) = \ov D(p)
\]
is diffeomorphically equivalent to the trivial variation $D \times \ov D$. It
follows that $g(p, w)$ has a $C^4$ extension to a neighbourhood of $\mathscr{D}
\setminus D \times \{0\}$. Now fix a point $p_0 \in D$ and let $\ov B(0, r)
\subset D(p_0)$. Then there exists a neighbourhood $U$ of $p_0$ in $D$ such that
$\ov B(0, r) \subset D(p)$ for all $p \in U$. Since $g(p, w) - \vert w \vert
^{-2n
+ 2}$ is a harmonic
function of $w \in D(p)$ and is equal to $\la(p)$ when $w = 0$, we obtain by the
mean value property of harmonic function
\begin{equation}\label{la-g}
\begin{split}
\la(p) & = \frac{1}{r^{2n - 1} \sigma_{2n}} \int_{\pa B(0, r)} \big (g(p, w) -
\vert w \vert ^{-2n + 2} \big) \, dS_{w}\\
& = - \frac{1}{r^{2n-2}} + \frac{1}{r^{2n - 1} \sigma_{2n}} \int_{\pa B(0,
r)} g(p, w) \, dS_{w}
\end{split}
\end{equation}
where by $dS$ we denote the surface area measure on a smooth surface in
$\mathbf{R}^{2n}$ and $\sigma_{2n}$ be the surface area of $\pa B(0, 1)$. It
follows that $\la(p)$ is smooth on $U$ and thus on $D$.

\medskip

Now let $1 \leq \gamma \leq n$. Observe that for each $p \in D$, the functions
\[
\frac{\pa g}{\pa p_{\gamma}}(p, w), \quad \frac{\pa^2 g}{\pa p_{\gamma} \pa \ov
p_{\gamma}}(p, w)
\]
are harmonic in all of $D(p)$ and
\[
\frac{\pa g}{\pa p_{\gamma}}(p, 0) = \frac{\pa \la}{\pa p_{\gamma}}(p), \quad
\frac{\pa^2 g}{\pa p_{\gamma} \pa \ov p_{\be}}(p, 0) = \frac{\pa^2 \la}{\pa
p_{\gamma} \pa \ov p_{\gamma}}.
\]
To find the boundary values of these functions in terms of $f$,
consider the quantities $k_1^{\gamma}$ and $k_2^{\gamma}$,
\begin{equation}\label{K_1-and-K_2}
k_1^{\gamma}(p, w) = \frac{\pa f}{\pa p_{\gamma}} (p, w) \big \lvert \pa_{w} f
(p, w) \big \rvert ^{-1} \quad \text{and} \quad k_2^{\gamma}(p, w) =
\mathcal{L}^{\gamma}f (p, w) \big \lvert \pa_{w} f (p, w) \big \rvert ^{-3}
\end{equation}
where
\begin{equation}\label{levi}
\mathcal{L}^{\gamma}f = \frac{\pa^2 f}{\pa p_{\gamma} \pa \ov{p}_{\gamma}}
\lvert \pa_{w} f \rvert ^{2} - 2 \Re \Big(\frac{\pa f}{\pa p_{\gamma}}
\sum_{\al=1}^{n} \frac{\pa f}{\pa \ov{w}_{\al}} \frac{\pa^2 f}{ \pa w_{\al} \pa
\ov{p}_{\gamma}} \Big) + \Big\lvert \frac{\pa f}{\pa p_{\gamma}} \Big\rvert^2
\De_{w} f,
\end{equation}
defined wherever $\pa_w f(p, w) > 0$, thus, in particular on
\[\pa \mathscr{D} = \cup_{p \in D} \big( p, \pa D(p) \big).\]
Note that on $\pa \mathscr{D}$, the quantities $k_1^{\gamma}$ and
$k_2^{\gamma}$ are independent of the defining function $f$ for $\mathscr{D}$.
Since $g(p, w) > 0$ on $\mathscr{D}$, $g(p, w) = 0$ on $\pa \mathscr{D}$ and
$\vert \pa_{w} g(p, w) \vert = - \frac{1}{2} \frac{\pa g}{\pa n_{w}}(p, w) > 0$
on $\pa \mathscr{D}$, we can use $-g(p, w)$ as a defining function for
$\mathscr{D}$ and hence
\[
\frac{\pa g}{\pa p_{\gamma}}(p, w) = - k_{1}^{\gamma}(p, w) \vert \pa_{w} g(p,
w) \vert
\]
and
\[
\mathcal{L}^{\gamma}g(p, w) = -k_{2}^{\gamma}(p, w) \vert \pa_{w} g(p, w) \vert
^3
\]
for all $(p, w) \in \pa \mathscr{D}$. Since $g(p, w)$ is of class $C^4$ up to
$\pa D(p)$, $\Delta_{w} g(p, w) = 0$ for $w \in \pa D(p)$ and hence from
(\ref{levi})
\begin{align*}
\frac{\pa^2 g}{\pa p_{\gamma} \pa \ov p_{\gamma}} & = - k_2^{\gamma} \vert
\pa_{w} g \vert + 2 \Re \bigg(\frac{\frac{\pa g}{\pa p_{\gamma}}}{\vert \pa_{w}
g \vert} \sum_{\al=1}^{n} \frac{\frac{\pa g}{\pa \ov w_{\al}}}{\vert \pa_{w} g
\vert} \frac{\pa^2  g}{\pa w_{\al} \pa \ov p_{\gamma}}\bigg)\\
& = - k_2^{\gamma} \vert \pa_{w} g \vert -  2 \Re \bigg(k_1^{\gamma}
\sum_{i=1}^{n} \frac{\frac{\pa g}{\pa \ov w_{\al}}}{\vert \pa_{w} g \vert}
\frac{\pa^2  g}{\pa w_{\al} \pa \ov p_{\gamma}}\bigg)
\end{align*}
for $w \in \pa D(p)$. We summarize this in the following
\begin{prop} \label{bdy-val-der-g}
The function $g(p, w)$ is smooth upto $\mathscr{D} \cup \pa \mathscr{D} =\{(p,
w) : p \in D, w \in \ov D(p)\}$. If $ 1\leq \gamma \leq n$ and $p \in D$, then
\begin{enumerate}
\item  $\frac{\pa g}{\pa p_{\gamma}}(p)$ is a harmonic function of $w \in D(p)$
with
\[
\frac{\pa g}{\pa p_{\gamma}}(p, 0) = \frac{\pa \la}{\pa p_{\gamma}}(p)
\]
and with boundary values
\[
\frac{\pa g}{\pa p_{\gamma}}(p, w) = - k_{1}(p, w) \vert \pa_{w} g(p, w) \vert,
\quad w \in \pa D(p),
\]

\item $\frac{\pa^2 g}{\pa p_{\gamma} \pa \ov p_{\gamma}}(p)$ is a harmonic
function
of $w \in D(p)$with
\[
\frac{\pa^2 g}{\pa p_{\gamma} \pa \ov p_{\be}}(p, 0) = \frac{\pa^2 \la}{\pa
p_{\gamma} \pa \ov p_{\gamma}}(p)
\]
and with boundary values
\[
\frac{\pa^2 g}{\pa p_{\gamma} \pa \ov p_{\gamma}}(p,w) = - k_2^{\gamma}(p,w)
\vert \pa_{w} g(p,w) \vert -  2 \Re \bigg(k_1^{\gamma}(p,w) \sum_{\al=1}^{n}
\frac{\frac{\pa g}{\pa \ov w_{\al}}(p,w)}{\vert \pa_{w} g(p,w) \vert}
\frac{\pa^2 g}{\pa w_{\al} \pa \ov p_{\gamma}}(p,w)\bigg), \quad w \in \pa D(p).
\]
\end{enumerate}
\end{prop}

\noindent To this end, it was proved in \cite{LY} that $g(p, w)$ is
$C^2$ up to $\{(p, w) : p \in \ov D, w \in \ov D(p)\}$ by deriving the
following estimates: there exists a constant $C$ independent of $p \in \pa D$
such that
\begin{equation}\label{estimates}
\begin{cases}
\vert k_1^{\ga}(p, w) \vert \leq C \vert w \vert^{2}\\
\vert k_2^{\ga}(p, w) \vert \leq C \vert w \vert^{3}\\
\vert \pa_w g(p, w) \vert \leq C \vert w \vert^{-2n+1}\\
\Big\vert \frac{\pa^2 g}{\pa \ov w_{\al} \pa p_{\gamma}} \Big \vert \leq C \vert
w \vert^{-2n +2}
\end{cases}
\end{equation}
for all $w \in \pa D^{\nu}$ with $\vert w \vert \geq 1$. Moreover, the
derivatives $\frac{\pa g}{\pa p_{\gamma}}$ and $\frac{\pa^2 g}{\pa p_{\gamma}
\pa \ov p_{\gamma}}$ are given by the following variation formulae:
\begin{prop}
Let $1 \leq \gamma \leq n$. Then for $p \in \ov D$ and $a \in D(p)$,
\begin{equation}\label{1st-var-g}
\frac{\pa g}{\pa p_{\gamma}}(p, a) = \frac{1}{2(n - 1) \sigma_{2n}} \int_{\pa
D(p)} k_1^{\gamma}(p, w) \big \lvert \pa_{w} g(p, w)
\rvert \frac{\pa g_a(p, w)}{\pa n_w} dS_{w}
\end{equation}
and
\begin{multline}\label{2nd-var-g}
\frac{\pa^2 g}{\pa p_{\gamma}\pa \ov{p}_{\gamma}}(p,a) = \frac{1}{2(n - 1)
\sigma_{2n}} \int_{\pa D(p)} k_2^{\gamma}(p, w) \vert
\pa_{w}g(p, w) \vert \frac{\pa g_a(p, w)}{\pa n_w} \, dS_{w}\\
+ \frac{1}{(n - 1) \sigma_{2n}} \Re \sum_{\al =1}^{n} \int_{\pa D(p)}
k_1^{\gamma}(p,w) \frac{\frac{\pa g}{\pa \ov w_{\al}}(p,w)}{\vert \pa_{w} g(p,w)
\vert} \frac{\pa^2 g}{\pa {w}_{\al} \pa \ov p_{\gamma} } (p, w) \frac{\pa
g}{\pa n_w}(w) \,dS_{w}.
\end{multline}
where $g_a(p,w)$ is the Green function for $D(p)$ with pole at $a$.
\end{prop}
\noindent We note that for $p \in D$, the above formulae are consequences of
proposition \ref{bdy-val-der-g}. For $p \in \pa D$, these formulae were
obtained in \cite{LY} by finding
\[\lim_{D \ni q \ra p} \frac{\pa g}{\pa p_{\ga}}(q, a) \quad \text{and} \quad
\lim_{D \ni q \ra p} \frac{\pa^2 g}{\pa p_{\ga} \pa \ov p_{\ga}}(q, a).
\]
A particular case of this proposition is the following:
\begin{prop}\label{varn-formula}
Let $1 \leq \gamma \leq n$ and $p \in \ov D$.
\begin{equation}\label{1st-var}
\frac{\pa \la}{\pa p_{\gamma}}(p) = \frac{1}{2(n - 1) \sigma_{2n}} \int_{\pa
D(p)} k_1^{\gamma}(p, \z) \big \lvert \pa_{w}
g(p, \z) \rvert \frac{\pa g(p, w)}{\pa n_w} dS_{w}
\end{equation}
and
\begin{multline}\label{2nd-var}
\frac{\pa^2 \la}{\pa p_{\gamma}\pa \ov{p}_{\gamma}}(p) = \frac{1}{2(n - 1)
\sigma_{2n}} \int_{\pa D(p)} k_2^{\gamma}(p, w) \vert
\pa_{w}g(p, \z) \vert^2 \, dS_{w}\\
\frac{1}{(n-1)\sigma_{2n}} \Re \sum_{\al=1}^{n} \int_{\pa D(p)} 
k_1^{\gamma}(p,w) \frac{\frac{\pa g}{\pa \ov w_{\al}}(p,w)}{\vert \pa_{w} g(p,w)
\vert} \frac{\pa^2 g}{\pa {w}_{\al} \pa \ov p_{\gamma} } (p, w) \frac{\pa g}{\pa
n_w}(p, w) \, dS_w.
\end{multline}
\end{prop}

We now consider a sequence $\{D_{\nu}\}$ of $C^{\infty}$-smoothly
bounded domains in $\mathbf{C}^{n}$ with $C^{\infty}$-smooth defining functions
$\psi_{\nu}$ such that $\{\psi_{\nu}\}$ converges in the $C^{\infty}$-topology
on compact subsets of $\mathbf{C}^{n}$ to $\psi$. In other words, $\{D_{\nu}\}$
converges in the $C^{\infty}$-topology to $D$. Another commonly used
terminology
for this is that the sequence $\{D_{\nu}\}$ is a $C^{\infty}$-perturbation of
$D$. This implies, in particular, that $D_{\nu}$ converges in the Hausdorff
sense to $D$. Now for each $\nu \geq 1$, consider the scaling map $T_{\nu} :
D_{\nu}
\times \mathbf{C}^{n} \ra \mathbf{C}^{n}$ defined by
\[
T_{\nu}(p, z) = \frac{z - p}{-\psi_{\nu}(p)}
\]
and the family of domains $\{D_{\nu}(p) : p \in \ov D_{\nu}\}$ defined by
\[
D_{\nu}(p) =
\begin{cases}
T_{\nu}(p, D_{\nu}) & ; \quad \text{if } p \in D_{\nu},\\
\Big\{w \in \mathbf{C}^{n} : 2 \Re \big( \sum_{i=1}^{n} \psi_{\nu i}(p) w_i
\big) -1 < 0 \Big\} & ; \quad \text{if } p \in \pa D.
\end{cases}
\]
The normalised Robin function $\la_{\nu}(p)$ for $(D_{\nu}, \psi_{\nu})$ is then
the Robin constant for $D_{\nu}(p)$ at $0$. We will denote the Green function
for $D_{\nu}$ with pole at $0$ by $g_{\nu}(p, w)$. Also, let
\[
\mathscr{D}_{\nu} = \cup_{p \in D_{\nu}} \big(p, D_{\nu}(p)\big) = \big\{(p, w)
: p \in D_{\nu}, w \in D_{\nu}(p) \big\}
\]
be the variation associated to $(D_{\nu}, \psi_{\nu})$ and
let
\begin{equation}\label{defn-f_nu}
f_{\nu}(p, w) = 2 \Re \bigg \{ \sum_{\al = 1}^{n} \int _{0} ^{1} \Big( w_{\al}
(\psi_{\nu})_{\al} \big(p - \psi_{\nu}(p) t
w \big) \Big) dt \bigg \} - 1.
\end{equation}
Then $f_{\nu}(p, w)$ is a smooth function on $\mathbf{C}^{n} \times
\mathbf{C}^{n}$
that defines the variation $\mathscr{D}_{\nu}$. It is evident that the functions
$f_{\nu}(p, w)$ converge in the $C^{\infty}$-topology on compact subsets of
$\mathbf{C}^{n} \times \mathbf{C}^{n}$
to the function
\[
f(p, w) =2 \Re \bigg \{ \sum_{\al = 1}^{n} \int _{0} ^{1} \Big( w_{\al}
\psi_{\al} \big(p - \psi(p) t
w \big) \Big) dt \bigg \} - 1
\]
which defines the variation $\mathscr{D}$ associated to $(D, \psi)$.

\medskip

Now let $p_{\nu} \in D_{\nu}$ be such that $\{p_{\nu}\}$ converges to $p_0 \in
\pa D$. For brevity, we let
\begin{equation}\label{defn-T^nu}
\begin{split}
& T^{\nu}(z) = T_{\nu}(p_{\nu}, z) = \frac{z-p_{\nu}}{-\psi_{\nu}(p_{\nu})},\\
& D^{\nu} = D_{\nu}(p_{\nu}) = T^{\nu}(D_{\nu}), \text{ and}\\
& g^{\nu}(w) = g_{\nu}(p_{\nu}, w).
\end{split}
\end{equation}
Thus $g^{\nu}(w)$ is the Green function for $D^{\nu}$ with pole at $0$. Let $1
\leq \ga \leq n$. By proposition \ref{bdy-val-der-g}, $\frac{\pa g_{\nu}}{\pa
p_{\ga}}(p_{\nu}, w)$ is a harmonic function of $w \in D^{\nu}$ with boundary
values
\begin{equation}\label{bd-val-1st-der-g_nu}
-k_1^{\nu \ga}(w) \vert \pa_w g^{\nu} (w) \vert
\end{equation}
where
\begin{equation}\label{def-k_1^nu}
k_1^{\nu \ga}(w) = k_{1 \nu}^{\ga}(w) = \frac{\pa f_{\nu}}{\pa p_{\ga}}(p_{\nu},
w) \vert \pa_w f_{\nu}(p_{\nu}, w) \vert^{-1}.
\end{equation}
Similarly, $\frac{\pa^2 g_{\nu}}{\pa p_{\ga} \pa \ov p_{\ga}}(p_{\nu}, w)$ is a
harmonic function of $w \in D^{\nu}$ with boundary values
\begin{equation}\label{bd-val-2nd-der-g_nu}
\frac{\pa^2 g_{\nu}}{\pa p_{\gamma} \pa \ov p_{\ga}}(p_{\nu}, w) =  - k_2^{\nu
\gamma}(w)
\vert \pa_{w} g^{\nu}(w) \vert -  2 \Re \bigg(k_1^{\nu\gamma}(w)
\sum_{\al=1}^{n}
\frac{\frac{\pa g^{\nu}}{\pa \ov w_{\al}}(w)}{\vert \pa_{w} g^{\nu}(w) \vert}
\frac{\pa^2
g_{\nu}}{\pa w_{\al} \pa \ov p_{\gamma}}(p_{\nu},w)\bigg), \quad w \in \pa
D^{\nu}
\end{equation}
where
\begin{equation}\label{def-k_2^nu}
k_2^{\nu \ga}(w) = \mathcal{L}^{\ga}f_{\nu}(p_{\nu}, w) \vert \pa_w
f_{\nu}(p_{\nu}, w) \vert^{-3}
\end{equation}
and $\mathcal{L}^{\ga}$ is defined by (\ref{levi}).
We want to conclude this section by finding uniform bounds for the functions
$k_1^{\nu \ga}(w)$ and $k_2^{\nu \ga}(w)$ near the
boundary of $\pa D^{\nu}$ which will be required to estimate the boundary values
(\ref{bd-val-1st-der-g_nu}) and (\ref{bd-val-2nd-der-g_nu}) in section 4 and 5.
For $0 < r < 1$ let $\mathcal{E}^{\nu}(r)$ be the
collar about $\pa D^{\nu}$ defined by
\[
\mathcal{E}^{\nu}(r) = \cup_{w_0 \in \pa D^{\nu}} \big\{w \in D^{\nu} : \vert w
- w_0 \vert < r \vert w_0 \vert \big\}.
\]
Note that $\mathcal{E}^{\nu}(r)$ lies in $D^{\nu}$ and
$\ov{\mathcal{E}}^{\nu}(r)$ does not contain the origin. Similarly, let
$\mathcal{E}_{\nu}(r)$ be the collar around $\pa D_{\nu}$ defined by
\[
\mathcal{E}_{\nu}(r) = \cup_{z_0 \in \pa D_{\nu}} \{z \in D_{\nu} : \vert z -
z_0 \vert < r \vert z_0 - p_{\nu} \vert \}.
\]
Note that $\mathcal{E}_{\nu}(r)$ lies in $D_{\nu}$ and does not contain the
point $p_{\nu}$. Also, note that
\begin{equation}\label{T_nu-inverse}
\mathcal{E}_{\nu}(r) = (T^{\nu})^{-1} \big(\mathcal{E}^{\nu}(r)\big).
\end{equation}

\begin{lem}\label{exist_de}
There exists a constant $m>0$, a number $0 < r < 1$,  and an integer $I$ such
that
\[
\vert \pa_{w}f_{\nu}(p_{\nu}, w) \vert > m
\]
for all $\nu \geq I$ and $w \in {\mathcal{E}}^{\nu}(r) $.
\end{lem}
\begin{proof}
Choose a $\de$-neighbourhood $U$ of $\pa D$ i.e.,
\[
U = \{ z \in \mathbf{C}^{n} : d(z, \pa D) < \de \}
\]
and a constant $m>0$ such that $\vert \pa \psi(p) \vert > 2m$ for $p \in U$.
Since $\pa \psi_{\nu}$ converges uniformly on $\ov U$ to $\pa \psi$, there
exists an integer $I$ such that
\begin{equation}\label{lbd-pa-psi_nu}
\vert \pa \psi_{\nu}(p) \vert > m
\end{equation}
for $\nu \geq I$ and $p \in U$. Modify the integer $I$ so that $\pa D_{\nu}
\subset N(\de/2)$ for all $\nu \geq
I$. Since $p_{\nu} \ra p_0 \in \pa D$, without loss of generality let us assume
that $p_{\nu} \in U$ for all $\nu \geq I$. Now define
\[
r = \frac{\de}{3 \de + 2\text{diam}(D)}.
\]
Then it is evident that
\begin{equation}\label{U-contains-collar}
\mathcal{E}_{\nu}(r) \subset U
\end{equation}
for $\nu \geq I$. Now fix $\nu \geq I$ and $w \in \mathcal{E}^{\nu}(r)$. If we
define
$z=T_{\nu}^{-1}w = p_{\nu} - \psi_{\nu}(p_{\nu})w$
then, by (\ref{T_nu-inverse})
\[
z \in \mathcal{E}_{\nu}(r) \subset U.
\]
From (\ref{defn-f_nu}),
\[
\vert \pa _w f_{\nu} (p_{\nu}, w) \vert = \vert \pa \psi_{\nu} (z) \vert > m
\]
by (\ref{lbd-pa-psi_nu}).
\end{proof}

\medskip

We now modify step 4 of chapter 4 \cite{LY} to obtain the following
estimates:
\begin{lem}\label{upper-bound-der-f-and-f_j}
Let $r$ and $I$ be as in lemma \ref{exist_de}. Then there exists a constant $M >
0$
such that
\begin{itemize}
\item [(i)] $\vert (\pa f_{\nu} / \pa w_{\al})(p_{\nu}, w) \vert < M$,

\item [(ii)] $\vert (\pa f_{\nu} / \pa p_{\gamma}) (p_{\nu}, w) \vert < M
\big(1+ \vert w \vert ^{-1} \big)\vert w \vert ^ 2$,

\item[(iii)] $\vert (\pa^2 f_{\nu} / \pa w_{\al} \pa w_{\be})
(p_{\nu}, w) \vert < M \vert w \vert ^{-1}$,

\item [(iv)] $\vert (\pa^2 f_{\nu} / \pa p_{\gamma} \pa w_{\al}) (p_{\nu}, w)
\vert < M \big(1 + \vert w \vert^{-1} \big) \vert w \vert$,

\item [(v)] $\vert (\pa^2 f_{\nu} / \pa p_{\gamma} \pa p_{\mu}) (p_{\nu}, w)
\vert < M \big(1 + \vert w \vert^{-1} + \vert w \vert^{-2}\big) \vert w \vert
^ 3$.
\end{itemize}
for all $\nu \geq I$ and $w \in {\mathcal{E}}^{\nu}(r)$.
\end{lem}

\begin{proof}
Let $U$ be as in the proof of lemma \ref{exist_de} and choose $R>0$ such that $U
\subset B(0, R)$. Since $\{\psi_{\nu}\}$ converges in the
$C^{\infty}$-topology on compact subsets of $\mathbf{C}^{n}$ to $\psi$, we can
find a constant $M_1 > 0$ such that $\psi$, $\psi_{\nu}$, $\nu \geq 1$, and
their
derivatives of order up to two are bounded in absolute value by $M_1$ on $\ov
B(0, R)$.

\medskip

Now let $\nu \geq I$ and let $w \in \mathcal{E}^{\nu}(r)$. Then we have
\begin{equation}\label{p-psi(p)tw}
p_{\nu} - \psi_{\nu}(p_{\nu}) t w \in B(0, R), \quad 0 \leq t \leq 1.
\end{equation}
Before proving this, note that this implies in particular that $\psi_{\nu}$ and
its derivatives of order up to $2$ are bounded in absolute value by $M_1$ at the
points $p_{\nu} - \psi_{\nu}(p_{\nu}) t w$ for all $0 \leq t \leq 1$. Now to
prove (\ref{p-psi(p)tw}), let $0 \leq t \leq 1$. Let
\[
z = T_{\nu}^{-1}w = p_{\nu} - \psi_{\nu}(p_{\nu}) w.
\]
Then by (\ref{T_nu-inverse}) $z \in \mathcal{E}_{\nu}(r)$ and hence by
(\ref{U-contains-collar}), $z \in U$. Now
\[
p_{\nu} - \psi_{\nu}(p_{\nu}) t w = p_{\nu} + t(z - p_{\nu}) = (1-t)p_{\nu} + t
z \in B(0, R)
\]
as $p_{\nu}, z \in U \subset B(0, R)$.

\medskip

\par (i) Differentiating (\ref{def-defining-fn}) with respect to $w_{\al}$ under
the integral sign, we have
\[
\frac{\pa f}{\pa w_{\al}}(p, w)  = \psi_{\al}\big( p - \psi(p) w \big), \quad p,
w \in \mathbf{C}^{n}.
\]
Hence for $\nu \geq I$ and $w \in \mathcal{E}^{\nu}(r)$,
\[
\Big \vert \frac{\pa f_{\nu}}{\pa w_{\al}}(p_{\nu}, w) \Big \vert = \big\vert
\psi_{\nu \alpha}\big(p_{\nu} - \psi_{\nu}(p_{\nu}) w \big) \big\vert \leq M_1.
\]

\medskip

\par (ii) Differentiating (\ref{def-defining-fn}) with respect to $p_{\gamma}$
under the integral sign, we have
\[
\frac{\pa f}{\pa p_{\gamma}}(p, w) = \sum_{\al = 1}^{n} \int_{0}^{1}
\frac{\pa}{\pa
p_{\gamma}} \Big( w_{\al}  \psi_{\al}\big(
p - \psi(p) t w \big)\Big) + \frac{\pa}{\pa p_{\gamma}} \Big(\ov{w}_{\al}
\psi_{\ov{\al}} \big( p - \psi(p) t w \big)\Big)
\, dt, \quad p, w \in \mathbf{C}^{n}.
\]
Observe that
\[
\frac{\pa}{\pa p_{\gamma}} \Big( w_{\al} \psi_{\al} \big(p -\psi(p)tw \big)
\Big) =
w_{\al} \psi_{\gamma\al} \big(p -\psi(p) t
w \big)- 2 t\psi_{\gamma}(p)  \Re \sum_{i = 1}^{n}  w_{i} w_{\al}\psi_{i\al}
\big(p
-\psi(p) t w \big).
\]
Therefore,
\begin{multline}\label{der-f-p}
\frac{\pa f}{\pa p_{\gamma}}(p, w) = \sum_{\al = 1}^{n} \int_{0}^{1} \Big(
w_{\al}
\psi_{\gamma\al} \big(p -\psi(p) t w \big)
+
\ov{w}_{\al} \psi_{\gamma \ov \al} \big(p -\psi(p) t w \big) \Big) \, dt\\
- 2 \psi_{\gamma}(p) \Re \sum_{i, \al = 1}^{n} \int_{0}^{1} \Big( w_{i}
w_{\al}\psi_{i\al} \big(p -\psi(p) t w \big)
 + w_{i} \ov{w}_{\al} \psi_{i \ov \al} \big(p -\psi(p) t w \big) \Big)t \, dt.
\end{multline}
Hence, for $\nu \geq I$ and $w \in \mathcal{E}^{\nu}(r)$,
\begin{align*}
\Big \vert \frac{\pa f_{\nu}}{\pa p_{\gamma}}(p_{\nu}, w) \Big \vert & \leq \sum
_{\al
= 1}^{n} \int_{0}^{1} \vert w_{\al} \vert
\big\vert \psi_{\nu \gamma \al}\big( p_{\nu} - \psi_{\nu}(p_{\nu}) t w \big)
\big\vert +
\vert \ov{w}_{\al} \vert \big\vert \psi_{\nu \gamma \ov
\al}\big( p_{\nu} - \psi_{\nu}(p_{\nu}) t w \big) \big\vert \, dt\\
& + 2 \vert \psi_{\nu \gamma}(p_{\nu}) \vert \sum_{i, \al = 1}^{n} \int_{0}^{1}
\vert
w_i \vert \vert w_{\al} \vert \psi_{\nu i \al}(p_{\nu} -
\psi_{\nu}(p_{\nu})tw) \vert + \vert w_i \vert \vert \ov{w}_{\al} \vert
\psi_{\nu i
\ov \al}(p_{\nu} - \psi_{\nu}(p_{\nu})tw) \vert t \, dt\\
& \leq \int_{0}^{1} 2 \vert w \vert \sqrt{n} M_1 \, dt + 2 M_1 \sum_{i = 1}^{n}
\int_{0}^{1} 2 \vert w_i
\vert \vert w \vert
\sqrt{n} M_1 t \, dt\\
&\leq 2 \sqrt{n} M_1 \vert w \vert + 2 n^{3/2} (M_1)^{2} \vert w \vert ^ 2\\
& \leq M_2 \big(1 + \vert w \vert ^{-1} \big) \vert w \vert^2
\end{align*}
where $M_2 = 2 n^{3/2} (M_1)^{2} $.

\medskip

\par (iii) Differentiating (\ref{def-defining-fn}) with respect to $w_{\al}$
under the integral sign, we have
\[
\frac{\pa f}{\pa w_{\al}}(p, w)  = \psi_{\al}\big( p - \psi(p) w \big), \quad p,
w \in \mathbf{C}^{n}.
\]
Differentiating this equation with respect to $w_{\be}$,
\[
\frac{\pa ^ 2 f}{\pa w_{\be} \pa w_{\al}} (p, w) = \big(-\psi(p) \big)\psi_{\al
\be}(p - \psi(p) w), \quad p, w \in
\mathbf{C}^{n}.
\]
Let $\nu \geq I$ and $w \in \mathcal{E}^{\nu}(r)$. Let
\[
z = T_{\nu}^{-1}w = p_{\nu} - \psi_{\nu}(p_{\nu})w.
\]
Then by (\ref{p-psi(p)tw}), $z \in B(0, R)$. Now we have
\[
\Big \vert \frac{\pa ^ 2 f_{\nu}}{\pa w_{\be} \pa w_{\al}} (p_{\nu}, w) \Big
\vert
\leq \frac{\vert z - p_{\nu}\vert}{\vert w \vert}\vert \psi_{\nu \al \beta}(z)
\vert
\leq 2 R M_1 \vert w \vert^{-1} = M_3 \vert w \vert^{-1}
\]
where $M_3 = 2 R M_1$. Finally differentiating (\ref{der-f-p}), we obtain (iv)
and (v).
\end{proof}

\begin{prop}\label{upper_boun_k,k_j}
There exist $0<r<1$, a constant $C$ and an integer $I$ such that
\begin{enumerate}
\item [(1)] $\vert k_1^{\nu \ga} (w) \vert \leq C (1 + \vert w \vert^{-1}) \vert
w \vert ^ 2$,
and

\item [(2)] $\vert k_2^{\nu \ga}(w) \vert \leq C (1 + \vert w \vert^{-1} + \vert
w
\vert^{-2})\vert w \vert ^ 3$
\end{enumerate}
for all $\nu \geq I$ and $w \in \ov{\mathcal{E}}^{\nu}(r)$.
\end{prop}

\begin{proof}
Let $0<r<1$, $m>0$ and $I$ be as in lemma \ref{exist_de}. Choose $M$ as in lemma
\ref{upper-bound-der-f-and-f_j}. Then from (\ref{def-k_1^nu})
\[
\vert k_1^{\nu} (w) \vert = \Big \vert \frac{\pa f_{\nu}}{\pa
p_{\gamma}}(p_{\nu}, w) \Big \vert \vert \pa_{w}f_{\nu}(p_{\nu}, w) \vert^{-1}
<\frac{M}{m}\big(1+ \vert w \vert ^{-1} \big) \vert w \vert ^ 2
\]
for $\nu \geq I$ and $w \in \mathcal{E}^{\nu}(r)$. Also, since $0 \not \in
\ov{\mathcal{E}}^{\nu}(r)$, the function
\[
\vert k_1^{\nu} (w) \vert (1 + \vert w \vert^{-1})^{-1} \vert w \vert^{-2}
\]
is continuous up to $\ov{\mathcal{E}}^{\nu}(r)$ and hence (1) follows.

\medskip

\no Similarly, from (\ref{def-k_2^nu})
\begin{multline*}
\vert k_2^{\nu}(w) \vert < \frac{1}{m^3} \Big( M (1 + \vert w \vert^{-1} +
\vert w \vert^{-2}) \vert w \vert
^3 M^2 + 2 n M (1 + \vert w \vert^{-1}) \vert w \vert ^2 M M (1 + \vert w
\vert^{-1}) \vert w \vert \\
 + (M (1 + \vert w \vert ^{-1}) \vert w \vert ^2)^2 n M \vert w
\vert^{-1}\Big) \leq C  (1 + \vert w \vert^{-1} + \vert w \vert^{-2}) \vert w
\vert ^3
\end{multline*}
for some constant $C$ whenever $\nu \geq I$ and $w \in \mathcal{E}^{\nu}(r)$.
Again the function
\[
\vert k_2^{\nu}(w) \vert  (1 + \vert w \vert^{-1} + \vert w \vert^{-2})^{-1}
\vert w \vert ^{-3}
\]
is continuous upto $\ov{\mathcal{E}}^{\nu}(r)$ and hence (2) follows.
\end{proof}

\section{Asymptotics of $\Lambda_{\nu}$}
\noindent In this section we prove theorem 1.2. First, we recall the following
stability result from \cite{BV}. 

\begin{prop}\label{stability}
Let $D$ be a domain in $\mathbf{C}^{n}$ with $C^2$-smooth boundary and let
$\{D_j\}$
be a $C^2$- perturbation of $D$. Let $G(z, p)$ be the Green function for $D$
with pole at $p$ and let
$\Lambda(p)$ be the Robin function for $D$. Similarly, let $G_j(z, p)$ be the
Green function for $D_j$ with pole at $p$ and $\Lambda_j(p)$ the Robin function
for $D_j$. Then
\[
\lim_{j \ra \infty} G_j(z, p) =G(z, p)
\]
uniformly on compact subsets of $D \setminus \{p\}$ and
\[
\lim_{j \ra \infty} D^{A \ov B} \Lambda_j (p) = D^{A, \ov B} \Lambda(p)
\]
uniformly on compact subsets of $D$.
\end{prop}
\noindent For a proof see \cite{BV}*{proposition 7.1, propostion 7.2}. This
proposition, together with \cite{LY}*{proposition~5.1} yields the following
boundary behaviour of the functions $G_j(z, p)$.

\begin{cor}\label{bdy-conv-Green-fn}
Let $D$ be a domain in $\mathbf{C}^{n}$ with $C^{\infty}$-smooth boundary and
let
$\{D_j\}$ be a $C^{\infty}$-perturbation of $D$. Let $z_j \in \ov D_j$ be such
that $\{z_j\}$
converges
to a point $z_0 \in \pa D$. Then for any $p \in D$,
\[
\lim_{j \ra \infty} G_j(z_j, p) = G(z_0, p)
\]
and identifying $z=(z_1, \ldots, z_n) \in \mathbf{C}^{n}$ with $x=(x_1, \ldots,
x_{2n}) \in \mathbf{R}^{2n}$,
\[
\lim_{j \ra \infty} \frac{\pa G_j}{\pa x_k}(z_j, p) = \frac{\pa G}{\pa x_k}(z_0,
p)
\]
for $1 \leq k \leq 2n$.
\end{cor}
\begin{proof}
Since the Green function is invariant under translation and rotation, without
loss of generality, we assume that $z_0 = 0$ and the normal to $\pa D$ at $z_0$
is along $x_{2n}$ axis. By the implicit function theorem, we can find a ball
$B(0,
r)$, a $C^{\infty}$-smooth function $\phi$ defined on $B(0^{\prime}, r) \subset
\mathbf{R}^{2n-1}$,  a sequence $\{\phi_j\}$ of $C^{\infty}$-smooth functions
defined on $B(0^{\prime}, r)$ that converges in $C^{\infty}$-topology on compact
subsets of $B(0^{\prime}, r)$ to $\phi$ such that
\begin{equation}\label{graph}
\begin{cases}
B(0, r) \cap \pa D = \{\big(x^{\prime}, \phi(x^{\prime}) \big) : x^{\prime} \in
B(0^{\prime}, r) \},\\
B(0, r) \cap \pa D_j = \{\big(x^{\prime}, \phi_j(x^{\prime})\big) : x^{\prime}
\in B(0^{\prime}, r)\}.
\end{cases}
\end{equation}
Now let $p \in D$. Shrinking $r$ if necessary, let us assume that $2r < \vert p
\vert$. Then for $z \in B(0, r) \cap D_j$,
\begin{equation}\label{bd-G_j}
G_j(z, p) < \vert z - p\vert^{-2n+2} < r^{-2n+2}.
\end{equation}
Consider the dilation
\[
Z = Sz = \frac{z}{r}
\]
and set
\[
\Om = S \big(B(0, r) \cap D \big), \quad \Om_j = S \big(B(0, r) \cap D_j \big).
\]
Define
\[
u(Z) = r^{2n-2} G(z, p), \quad Z \in \Om, 
\]
and
\[
u_j(Z) = r^{2n - 2} G_j(z, p), \quad Z \in \Om_j
\]
Then by (\ref{graph}) , (\ref{bd-G_j}) and in view of proposition
\ref{stability}, the sequence $\{u_j\}$ on $\{\Om_j\}$ satisfies the hypothesis
of \cite{LY}*{proposition~5.1} and therefore
\[
\begin{cases}
\lim_{j \ra \infty} u_j(Z_j) = u(0),\\
\lim_{j \ra \infty} \frac{\pa u_j}{\pa \ti x_k}(Z_j) = \frac{\pa u}{\pa \ti
x_k}(0).
\end{cases}
\]
where $Z_j = Sz_j$. This implies that
\[
\begin{cases}
\lim_{j \ra \infty} G_j(z_j, p) = G(0, p),\\
\lim_{j \ra \infty} \frac{\pa G_j}{\pa x_k}(z_j, p) = \frac{\pa G}{\pa x_k}(0,
p).
\end{cases}
\]
\end{proof}

\noindent \textit{Proof of theorem 1.2.}
Consider the affine maps $T^{\nu} : \mathbf{C}^{n} \ra \mathbf{C}^{n}$ defined
by
\[
T^{\nu}(z) = \frac{z - p_{\nu}}{-\psi_{\nu}(p_{\nu})}
\] 
and the scalled domains $D^{\nu} = T^{\nu}(D_{\nu})$. Recall from the previous
section that a defining function for $D^{\nu}$ is given by
\[
f_{\nu}(p_{\nu}, w) = 2 \Re \bigg \{ \sum_{\al = 1}^{n} \int _{0}
^{1} \Big( w_{\al} \psi_{\nu\al} \big(p_{\nu} - \psi_{\nu}(p_{\nu}) t w
\big) \Big) dt \bigg \} - 1.
\]
It is evident that $\{f_{\nu}(p_{\nu}, \cdot)\}$ converges in the
$C^{\infty}$-topology on compact subsets of $\mathbf{C}^{n}$ to
\[
f(p_0, w) = 2 \Re \Big(\sum_{\al=1}^{n} \psi_{\al}(p_0) w_{\al} \Big)
-1.
\]
This implies that $\{D^{\nu}\}$ is a $C^{\infty}$-perturbation of the half space
\[
\mathcal{H} = \Big\{ w : 2 \Re \Big(\sum_{\al=1}^{n} \psi_{\al}(p_0)
w_{\al} \Big) -1 < 0 \Big\}.
\]
Therefore, by proposition \ref{stability}
\begin{equation}\label{lim-der-La_nu}
\lim_{\nu \ra \infty} D^{A \ov B} \Lambda_{D^{\nu}} (0) = D^{A \ov
B}\Lambda_{\mathcal{H}}(0).
\end{equation}
Now by \cite{BV}*{(1.1)},
\[
\Lambda_{D^{\nu}}(p) = \Lambda_{\nu}\big(p_{\nu} - p \psi_{\nu}(p_{\nu})\big)
\big(\psi_{\nu}(p)\big)^{2n-2}
\]
Differentiating this we obtain
\[
D^{A \ov B} \Lambda_{D^{\nu}}(0) = (-1)^{\vert A \vert + \vert B \vert} D^{A \ov
B} \Lambda_{\nu} (p_{\nu}) \big( \psi_{\nu}(p_{\nu})\big)^{2n-2+ \vert A \vert +
\vert B \vert}.
\]
Hence from (\ref{lim-der-La_nu}),
\[
\lim_{\nu \ra \infty} D^{A \ov B} (-1)^{\vert A \vert + \vert B \vert} D^{A \ov
B} \Lambda_{\nu} (p_{\nu}) \big( \psi_{\nu}(p_{\nu})\big)^{2n-2+ \vert A \vert +
\vert B \vert}  = D^{A \ov B}\Lambda_{\mathcal{H}}(0).
\]
which completes the proof. \hfill $\square$

\section{Estimates on the first derivatives}
\noindent Let $1 \leq \ga \leq n$. By proposition \ref{bdy-val-der-g},
$\frac{\pa g_{\nu}}{\pa p_{\gamma}}(p_{\nu},
w)$ is a harmonic function of $w \in D^{\nu}$,
\[
\frac{\pa g_{\nu}}{\pa p_{\gamma}}(p_{\nu},0) = \frac{\pa \la_{\nu}}{\pa
p_{\gamma}}(p_{\nu})
\]
and
\begin{equation}\label{bdy-val-der-g_nu}
\frac{\pa g_{\nu}}{\pa p_{\gamma}}(p_{\nu}, w) = - k_{1}^{\nu \gamma}( w) \vert
\pa_w g^{\nu}(w) \vert, \quad w \in \pa D^{\nu}.
\end{equation}
Therefore,
\begin{equation}\label{1st-der-la_nu(p_nu)}
\frac{\pa \la_{\nu}}{\pa p_{\gamma}}(p_{\nu})
=  \frac{1}{2(n-1) \sigma_{2n}}\int_{\pa D^{\nu}} k_1^{\nu \gamma}(w) \lvert
\pa_{w} g^{\nu}(w) \rvert \frac{\pa g^{\nu}}{\pa n_w}(w) dS_{w}.
\end{equation}
Thus to find the limit of the above integrals, we need to estimate the boundary
values (\ref{bdy-val-der-g_nu}). For this we modify Step 3 of
chapter 4 \cite{LY}.

\begin{lem}\label{external_ball}
There exists a number $0 < \rho < 1$ and an integer $I$ such that for $\nu \geq
I$ and $w_0 \in \pa D^{\nu}$, we can find a ball of radius $\rho \vert w_0
\vert$ that is externally tangent to $\pa D^{\nu}$ at $w_0$.
\end{lem}

\begin{proof}
Since $D$ is bounded, we can find a ball $B(0, R)$ which contains $D$. 
Since $\{D_{\nu}\}$ converges in $C^2$-topology to $D$, there exists an integer
$I$ such that  $D_{\nu} \subset B(0, R)$ for all $\nu \geq I$. By implicit
function theorem, there exists a number $\ti \rho$ such that modifying $I$ we
can
find for each $\nu \geq I$ and $z_0 \in \pa D_{\nu}$, a ball of radius $\ti
\rho$ that is externally tangent to $\pa D_{\nu}$ at $z_0$. Now let $\nu \geq I$
and $w_0 \in \pa D^{\nu}$. Since $D^{\nu}$ is obtained from $D_{\nu}$ by means
of a translation followed by dialation of factor $-\psi_{\nu}(p_{\nu})$, it
follows that we can find a ball of radius $\ti{\rho}/\big(-\psi_{\nu}(p_{\nu})
\big)$
that is externally tangent to $\pa D^{\nu}$ at $w_0$.
Also there exists $z_0 \in \pa D_{\nu}$ such that
\[
w_0 = \frac{z_0 - p_{\nu}}{-\psi_{\nu}(p_{\nu})}
\]
which implies that
\[
\frac{\ti \rho}{-\psi_{\nu}(p_{\nu})} = \frac{\ti \rho \vert w_0 \vert}{\vert
z_0 - p_{\nu} \vert} \geq \frac{\ti \rho}{2 R} \vert w_0 \vert.
\]
Thus taking $\rho = \ti \rho / 2R$, it follows that the we can find a ball of
radius $\rho \vert w_0 \vert$ that is tangent to $\pa D^{\nu}$ at $w_0$.
\end{proof}

\begin{prop}\label{bd-grad-g^nu}
There exists an an integer $I$ and a constanct $C>0$ such that
\[
\big \vert \pa_w g^{\nu}(w) \big \vert \leq C \vert w \vert^{-2n +1}
\]
for all $\nu \geq I$ and $w \in \pa D^{\nu}$.
\end{prop}

\begin{proof}
Choose $0 < \rho <1$, an integer $I$ and a constant $C$ as in lemma
\ref{external_ball}. Let $\nu \geq I$ and $w_0 \in \pa D^{\nu}$. Let $B$ be the
ball
of radius $\rho \vert w_0 \vert$ that is externally tangent to $\pa D^{\nu}$
at $w_0$. Let $E$ be the ball centred at $w_0$ and
of radius $\rho \vert w_0 \vert$. Then $w \in E$ implies that
\[
\vert w \vert > \vert w_0 \vert - \rho \vert w_0 \vert = (1 - \rho) \vert w_0
\vert.
\]
Therefore, for $w \in E \cap D^{\nu}$,
\[
0 < g^{\nu}(w) \leq \vert w \vert ^{-2n + 2} < \big( (1 - \rho) \vert w_0
\vert
\big)^{-2n + 2}.
\]
By step 2 of chapter 4 \cite{LY}, we have
\[
\vert \pa_{w} g^{\nu} (w_0) \vert \leq c \big ( (1 - \rho)  \lvert w_0 \rvert
\big)
^ {-2n + 2}
(\rho \lvert w_0 \rvert)^{-1}
\]
where $c$ does not depend on $g^{\nu}(w)$ or $D^{\nu}$. Thus
\[
\vert \pa_{w} g^{\nu}(w_0) \vert \leq C \lvert w_0 \rvert ^{-2n +1},
\]
where $C = c \rho^{-1} (1 - \rho)^ {-2n + 2}$ is independent of $\nu$
and $w_0 \in \pa D^{\nu}$.
\end{proof}

\begin{prop}\label{bdy-est-1-der-g}
There exists a constant $C>0$ and an integer $I$ such that
\[
\Big \vert \frac{\pa g_{\nu}}{\pa p_{\gamma}}(p_{\nu}, w) \Big \vert = \big
\vert k_{1}^{\nu \gamma}(w)\big \vert \big \vert \pa_w g^{\nu}(w) \big \vert
\leq C  \big(1 + \vert w \vert^{-1} \big) \vert w \vert ^{-2n+3}, \quad w \in
\pa D^{\nu}
\]
for all $\nu \geq I$.
\end{prop}
\begin{proof}
By proposition \ref{upper_boun_k,k_j}, there exists a constant $C$ and an
integer $I$ such that
\[
\vert k_{1}^{\nu \ga}(w) \vert \leq C  \big( 1 + \vert w \vert ^{-1}\big)
\vert w \vert^2, \quad w \in \pa D^{\nu}
\]
for all $\nu \geq I$. In view of proposition \ref{bd-grad-g^nu}, we can modify
the constant $C$ and the integer $I$ so that
\[
\vert \pa_{w}g^{\nu}(w) \vert \leq C \vert w \vert ^{-2n+1}, \quad w \in \pa
D^{\nu}
\]
for all $\nu \geq I$. Hence, from (\ref{bdy-val-der-g_nu}),
\[
\Big \vert \frac{\pa g_{\nu}}{\pa p_{\gamma}}(p_{\nu}, w) \Big \vert = \big
\vert k_{1}^{\nu \gamma}(w)\big \vert \big \vert \pa_w g^{\nu}(w) \big \vert
\leq C^2  \big(1 + \vert w \vert^{-1} \big) \vert w \vert ^{-2n+3}, \quad w \in
\pa D^{\nu}
\]
for all $\nu \geq I$.
\end{proof}

\begin{prop}\label{estimate2}
$
\displaystyle \lim_{\nu \ra \infty} \frac{\pa \la_{\nu}}{\pa
p_{\gamma}}(p_{\nu}) = \frac{\pa
\la}{\pa p_{\gamma}}(p_0).
$
\end{prop}

\begin{proof}
In view of proposition \ref{varn-formula} we
have to prove that
\begin{multline}\label{conv-1stder-int}
\lim_{\nu \ra \infty} \frac{1}{2(n-1) \sigma_{2n}}\int_{\pa D^{\nu}} k_{1}^{\nu
\gamma}(w) \vert \pa_w
g^{\nu}(w)\vert \frac{\pa g^{\nu}}{\pa n_w}(w) dS_w \\
= \frac{1}{2(n-1) \sigma_{2n}} \int_{\pa \mathcal{H}} k_1^{\gamma}(p_0, w) \vert
\pa
g(p_0, w) \vert \frac{\pa g}{\pa n_w}(p_0, w)
dS_w.
\end{multline}
where $\mathcal{H} = D(p_0)$. Let $R>1$. Then the boundary surfaces $B(0, R)
\cap \pa D^{\nu}$ converge to
$B(0, R) \cap \mathcal{H}$ continuously in the sense that the unit normal
vectors
\[
\frac{\pa_w g^{\nu}(w)}{\vert \pa_w g^{\nu}(w) \vert} \ra \frac{\pa g(p_0,
w)}{\vert \pa_w g(p_0, w) \vert}
\]
uniformly on compact sets, except at the corners $B(0, R) \cap \pa D^{\nu}$.
Also, if $w^{\nu} \in \pa D^{\nu}$ and $\{w^{\nu}\}$ converges to $w^0 \in \pa
\mathcal{H}$, then by definition
\begin{equation}\label{conv-k_nu}
\lim_{\nu \ra \infty} k_1^{\nu \gamma}(w^{\nu}) = k_1^{\gamma}(p_0, w^0)
\end{equation}
and by corollary \ref{bdy-conv-Green-fn}
\begin{equation}\label{bdy-conv-1st-der}
\lim_{\nu \ra \infty} \frac{\pa g^{\nu}}{\pa w_{\al}}(w^{\nu}) = \frac{\pa
g}{\pa w_{\al}}(p_0, w^0)
\end{equation}
for $1 \leq \al \leq n$. Hence,
\begin{multline}\label{conv-1stder-int-cpt}
\lim_{\nu \ra \infty} \frac{1}{2(n-1) \sigma_{2n}} \int_{B(0, R) \cap \pa
D^{\nu}} k_{1}^{\nu \gamma}(w)
\vert \pa_w g^{\nu}(w)\vert \frac{\pa g^{\nu}}{\pa n_w}(w) dS_w\\
= \frac{1}{2(n-1) \sigma_{2n}} \int_{B(0, R) \cap \pa \mathcal{H}}
k_1^{\gamma}(p_0, w)
\vert \pa g(p_0, w) \vert
\frac{\pa g}{\pa n_w}(p_0, w) dS_w.
\end{multline}
To esitmate these integrals outside the ball $B(0, R)$, note that by proposition
\ref{bdy-est-1-der-g}, there exists a constant $C$ and an integer $I$ such that
\[
\big \vert k_1^{\nu \gamma}(w) \big \vert \big\vert \pa_{w} g^{\nu}(w) \big
\vert \leq C \vert w \vert^{-2n+3}, \quad w \in \pa D^{\nu}, \vert w \vert > 1
\]
for all $\nu \geq I$. Therefore,
\begin{multline}\label{bound-1st-int}
\left\vert \frac{1}{2(n-1) \sigma_{2n}} \int_{B^c(0, R) \cap \pa D^{\nu}}
k_1^{\nu \gamma}(w) \lvert
\pa_{w} g^{\nu}(w) \rvert \frac{\pa g^{\nu}}{\pa n_w}(w) dS_{w} \right\vert \\
\leq C R^{-2n+3} \frac{1}{2(n-1) \sigma_{2n}} \int_{\pa B^c(0, R) \cap \pa
D^{\nu}}
\Big( - \frac{\pa g^{\nu}}{\pa n_{\z}}(w)\Big) \, dS_{w}
\end{multline}
for all $\nu \geq I$. Since
\[
\int_{\pa B^c(0, R) \cap \pa D^{\nu}}
\Big( - \frac{\pa g^{\nu}}{\pa n_{\z}}(w)\Big) \, dS_{w} \leq \int_{\pa D^{\nu}}
\Big( - \frac{\pa g^{\nu}}{\pa n_{w}}(w)\Big) \, dS_{w}
= (2n - 2) \sigma_{2n},
\]
we have from (\ref{bound-1st-int})
\begin{equation}\label{1st-int-small}
\left\vert \frac{1}{2(n-1) \sigma_{2n}} \int_{B^c(0, R) \cap \pa D^{\nu}}
k_1^{\nu \gamma}(w) \lvert
\pa_{w} g^{\nu}(w) \rvert \frac{\pa g^{\nu}}{\pa n_w}(w) dS_{w} \right\vert =
O(R^{-2n+3})
\end{equation}
uinformly for all $\nu \geq I$.
By (\ref{estimates}), we can modify the constant $C$ so that
\[
\big \vert k_1^{\gamma}(p_0, w) \big \vert \big \vert \pa_w g(p_0, w) \big
\vert \leq C \vert w \vert^{-2n+3}, \quad w \in \pa \mathcal{H}, \vert w \vert
>1
\]
and as above we obtain
\begin{equation}\label{1st-int-small2}
\left \vert \frac{1}{2(n-1) \sigma_{2n}} \int_{B^c(0, R) \cap \pa \mathcal{H}}
k_1^{\gamma}(p_0, w) \big\vert
\pa_w
g(p_0, w) \big \vert \frac{\pa g}{\pa n_w}(w) \, dS_w \right \vert =
O(R^{-2n+3}).
\end{equation}
Now (\ref{conv-1stder-int}) follows from (\ref{conv-1stder-int-cpt}),
(\ref{1st-int-small}) and (\ref{1st-int-small2}).
\end{proof}

\begin{rem}\label{unif-conv-g_nu}
Note that the arguments of this section also imply that
for any $a \in \mathcal{H}$, 
\begin{multline*}
\lim_{\nu \ra \infty} \frac{\pa g_{\nu}}{\pa p_{\ga}}(p_{\nu}, a) =\lim_{\nu \ra
\infty} \frac{1}{2(n-1) \sigma_{2n}}\int_{\pa D^{\nu}} k_{1}^{\nu \gamma}(w)
\vert \pa_w
g^{\nu}(w)\vert \frac{\pa g_{\nu a}}{\pa n_w}(p_{\nu},w) dS_w\\
=\frac{1}{2(n-1) \sigma_{2n}}\int_{\pa \mathcal{H}} k_{1}^{\gamma}(w) \vert
\pa_w
g^{0}(w)\vert \frac{\pa g_{a}}{\pa n_w}(p_{\nu},w) dS_w = \frac{\pa g}{\pa
p_{\ga}}(p_0, a).
\end{multline*}
Moreover, by proposition \ref{bdy-est-1-der-g}, the functions $\frac{\pa
g_{\nu}}{\pa p_{\ga}}(p_{\nu}, w)$ are uniformly bounded on compact subsets of
$\mathcal{H}$ for all large $\nu$. Indeed, let $\ov B(0, r) \subset
\mathcal{H}$. Then $\ov B(0, r) \subset D^{\nu}$ for all large $\nu$. It follows
that
\[
\Big \vert \frac{\pa g_{\nu}}{\pa p_{\ga}}(p_{\nu}, w) \Big \vert \leq
C r^{-2n+3} (1 + r^{-1})
\]
for $w \in \pa D^{\nu}$ and hence for $w \in D^{\nu}$ by the maximum priciple.
Therefore, $\big\{\frac{\pa g_{\nu}}{\pa p_{\ga}}(p_{\nu}, a)\big\}$
converges uniformly on compact subsets of $\mathcal{H}$
to $\frac{\pa g}{\pa p_{\ga}}(p_0, a)$.
\end{rem}

\section{Estimates on the second derivatives}
\noindent By proposition \ref{bdy-val-der-g}, $\frac{\pa^2 g_{\nu}}{\pa
p_{\gamma} \pa \ov p_{\ga}}(p_{\nu}, w)$ is a harmonic function of $w \in
D^{\nu}$,
\[
\frac{\pa^2 g_{\nu}}{\pa p_{\gamma} \pa \ov p_{\ga}}(p_{\nu},0) = \frac{\pa^2
\la_{\nu}}{\pa p_{\gamma} \pa \ov p_{\ga}}(p_{\nu}),
\]
and
\begin{equation}\label{bdy-val-der-g_nu11}
\frac{\pa^2 g_{\nu}}{\pa p_{\gamma} \pa \ov p_{\ga}}(p_{\nu}, w) =  - k_2^{\nu
\gamma}(w)
\vert \pa_{w} g^{\nu}(w) \vert -  2 \Re \bigg(k_1^{\nu\gamma}(w)
\sum_{\al=1}^{n}
\frac{\frac{\pa g^{\nu}}{\pa \ov w_{\al}}(w)}{\vert \pa_{w} g^{\nu}(w) \vert}
\frac{\pa^2
g_{\nu}}{\pa w_{\al} \pa \ov p_{\gamma}}(p_{\nu},w)\bigg), \quad w \in \pa
D^{\nu}.
\end{equation}
Therefore,
\begin{multline}\label{exp-2nd-der-la_nu}
\frac{\pa^2 \la_{\nu}}{\pa p_{\gamma}\pa \ov{p}_{\gamma}}(p_{\nu}) =
\frac{1}{2(n-1) \sigma_{2n}} \int_{\pa D^{\nu}} k_2^{\nu}(w)
\vert \pa_{w}g^{\nu}(\z) \vert \frac{\pa g^{\nu}}{\pa n_w}(w) \, dS_{w}\\
+ \frac{1}{(n-1) \sigma_{2n}} \Re \sum_{\al=1}^{n} \int_{\pa D^{\nu}} k_1^{\nu
\gamma}(w) \frac{\frac{\pa g^{\nu}}{\pa \ov w_{\al}}(w)}{\vert \pa_{w}
g^{\nu}(w) \vert} \frac{\pa^2 g_{\nu}}{\pa w_{\al} \pa \ov
p_{\gamma}}(p_{\nu},w) \frac{\pa g^{\nu}}{\pa n_w}(w) \, dS_{w}.
\end{multline}
By similar arguments as in the previous section
\begin{multline}\label{conv-1st-int}
\lim_{\nu \ra \infty} \frac{1}{2(n-1) \sigma_{2n}} \int_{\pa D^{\nu}}
k_2^{\nu}(w)
\vert \pa_{w}g^{\nu}(\z) \vert \frac{\pa g^{\nu}}{\pa n_w}(w) \, dS_{w}\\
= \frac{1}{2(n-1) \sigma_{2n}} \int_{\pa \mathcal{H}} k_2(p_0, w) \vert
\pa_{w}g(p_0,
w) \vert \frac{\pa g}{\pa n_w}(w) \, dS_{w}
\end{multline}
where $\mathcal{H} = D(p_0)$. Thus we only need to find the limit of the second
integrals. This requires to
estimate the functions
\begin{equation}\label{mix-der}
\frac{\pa^2 g_{\nu}}{\pa w_{\al} \pa \ov p_{\gamma}}(p_{\nu}, w)
\end{equation}
on $\pa D^{\nu}$. Since $\frac{\pa g_{\nu}}{\pa p_{\ga}}(p_{\nu}, w)$ is a
harmonic function of $w \in D^{\nu}$ with boundary values
\begin{equation}\label{defn-F^nu}
F^{\nu}(w) = -k_1^{\nu \ga}(w) \vert \pa_w g^{\nu}(w) \vert = - \frac{ \frac{\pa
f_{\nu}}{\pa p_{\ga}}(p_{\nu}, w)}{\vert \pa_w f_{\nu}(p_{\nu}, w) \vert}  \vert
\pa_w g^{\nu}(w) \vert,
\end{equation}
to estimate (\ref{mix-der}), we need to estimate the derivatives of
$F^{\nu}(w)$. This will be done by modifying Steps 2 and 3 of chapter
5 \cite{LY}.

\medskip

In what follows we will identify the point $z = (z_1, \ldots, z_n)$ in
$\mathbf{C}^{n}$ with the
point $x = (x_1, \ldots x_{2n})$ in $\mathbf{R}^{2n}$. Similarly $w = (w_1,
\ldots, w_n)$ and $W = (W_1, \ldots, W_{n})$ in $\mathbf{C}^{n}$ will be
identified with $y = (y_1, \ldots y_{2n})$ and $Y = (Y_1, \ldots, Y_{2n})$ in
$\mathbf{R}^{2n}$ respectively. First, we note the following version of a
tubular neighbourhood theorem:
\begin{prop}\label{implicit-fn}
There exist $0<r<1$ and $M>1$ and an integer $I$ such that for $\nu \geq I$
and any $z_0 = (x^{\prime}_{0}, x_{02n})$ in the neighbourhood
\[
\bigcup _{z \in \pa D_{\nu}} \{ z+tn_{z} : -r < t < r \}
\]
of $\pa D_{\nu}$, $B(z_0, r) \cap \pa D_{\nu}$ can be represented, after a
rotation and translation of coordinates, in the form $x_{2n} = \phi
(x^{\prime})$ where
\begin{itemize}
\item [(a)] $\phi(x^{\prime})$ is smooth in $B(x_0^{\prime}, r) \subset
\mathbf{R}^{2n-1}$ with $\phi(x^{\prime}_{0}) = x_{02n} - t$, where $t$ is such
that $z_0 = z_0^{*}+t n_{z_0^{*}}$ for some $z_0^{*} \in \pa D^{\nu}$, and 

\item [(b)] all partial derivatives of $\phi$ of order upto $6$ are bounded in
absolute value on $B(x_0^{\prime}, r)$ by $M$.
\end{itemize}
\end{prop}
\noindent Now fix $r$, $M$ and $I$ as in proposition \ref{implicit-fn}.
Modifying the
integer $I$, if necessary, we may assume that
\[
d(p_{\nu}, \pa D) < r
\]
and
\[
\pa D_{\nu} \subset \big\{z : d(z, \pa D) < r  \big\}
\]
for all $\nu \geq I$. This would imply that
\begin{equation}\label{ti-z_nu-p_nu}
\vert \ti{z}_{\nu} - p_{\nu} \vert < \text{diam}(D) + 2r
\end{equation}
for $\nu \geq I$ and $\ti{z}_{\nu} \in \pa D_{\nu}$. Now, choose $0 < \eta < 1$
such that
\begin{equation}\label{defn-eta}
\frac{\eta}{1-\eta}\big(\text{diam}(D) + 2r \big) < r.
\end{equation}

\begin{lem} \label{implicit-graph}
Let $\nu \geq I$ and $w^{\nu} \in D^{\nu} \setminus \{0\}$ be such that
\[
\{ w \in \mathbf{C}^{n} : \vert w - w^{\nu} \vert < \eta \vert w^{\nu} \vert \}
\cap \pa D^{\nu} \neq \emptyset .
\]
Let $S^{\nu} : \mathbf{C}^{n} \ra \mathbf{C}^{n}$ be the affine map
defined by
\[
W = S^{\nu}(w) = \frac{w - w^{\nu}}{\eta \vert w^{\nu} \vert}
\]
and set
\[
\Om^{\nu} = S^{\nu} \Big( \big\{ w \in \mathbf{C}^{n} : \vert w - w^{\nu} \vert
< \eta \vert w^{\nu} \vert \big\}
\cap  D^{\nu} \Big) = \{ \vert W \vert < 1 \} \cap S^{\nu}(D^{\nu}).
\]
Then we can find $\Phi^{\nu} \in C^{\infty}(\{Y^{\prime} :
\vert Y^{\prime} \vert < 1 \})$ with
\begin{enumerate}
\item $\{ \vert W \vert < 1\} \cap \pa \Om^{\nu}  = \{ Y_{2n} =
\Phi^{\nu}(Y^{\prime}) \}$, and

\item $\Big \vert \frac{\pa^{\al} \Phi^{\nu}}{\pa Y^{\al}} \Big\vert \leq M$ for
$\al = (\al_1, \ldots, \al_n)$ with $\vert \al\vert \leq 6$ if $\vert Y^{\prime}
\vert < 1$.
\end{enumerate}
\end{lem}

\begin{proof}
Let
\[
z_{\nu} = (T^{\nu})^{-1}(w^{\nu}) = p_{\nu} - \psi_{\nu}(p_{\nu}) w_{\nu}
\]
and let
\[
b_{\nu} = (T^{\nu})^{-1}\Big(\big\{w : \vert w - w^{\nu} \vert < \eta \vert
w_{\nu}
\vert \big\}\Big) = \big\{ z \in \mathbf{C}^{n} : \vert z - z_{\nu} \vert <
\eta \vert z_{\nu} - p_{\nu} \vert \big\}.
\]
Then $b_{\nu} \cap \pa D_{\nu} \neq \emptyset$ and hence there is a point
$\tilde z_{\nu} \in \pa D_{\nu}$ such that
\[
\vert \tilde z_{\nu} - z_{\nu} \vert < \eta \vert z_{\nu} - p_{\nu} \vert \leq
\eta (\vert z_{\nu} -
\tilde z_{\nu} \vert + \vert \tilde z_{\nu} - p_{\nu} \vert).
\]
Therefore,
\begin{equation}\label{z-z_0}
\vert \tilde z_{\nu} - z_{\nu} \vert < \frac{\eta}{1-\eta} \vert \tilde z_{\nu}
- p_{\nu} \vert \leq
\frac{\eta}{1-\eta} \big(\text{diam}(D) + 2r \big) < r
\end{equation}
by (\ref{ti-z_nu-p_nu}) and (\ref{defn-eta}) and hence
\[
z_{\nu} \in \bigcup_{z \in \pa D_{\nu}} \{ z+tn_{z} : -r < t < r \}.
\]
By proposition \ref{implicit-fn}, $B(z_{\nu}, r) \cap \pa D_{\nu}$
can be represented after a rotation and translation of coordinates in the form
$x_{2n} = \phi_{\nu}(x^{\prime})$
where $\phi_{\nu}(x^{\prime})$ is $C^{\infty}$ on $B(x^{\prime}_{\nu}, r)$,
\begin{equation}\label{x_nu}
\phi_{\nu}(x_{\nu}^{\prime}) = x_{\nu 0} -t_{\nu}
\end{equation}
where
\begin{equation}\label{t_nu}
-t_{\nu} = d(z_{\nu}, \pa D_{\nu}) < \eta \vert z_{\nu} - p_{\nu}\vert
\end{equation}
and all partial
derivatives of $\phi_{\nu}$ of order up to
$6$ are bounded in absolute value by $M$.  The surface
\[
\{(x^{\prime}, x_{2n}) : x_{2n} = \phi_{\nu}(x^{\prime}), \vert x^{\prime} -
x^{\prime}_{\nu} \vert < r \}
\]
is mapped by $S^{\nu} \circ T^{\nu}$ onto the surface
\[
\{(Y^{\prime}, Y_{2n}) : Y_{2n} = \Phi^{\nu}(Y^{\prime}), \vert Y^{\prime} \vert
< R^{\nu}\}
\]
where, letting $w^{\nu} = (y^{\nu \prime}, y^{\nu}_{2n})$, $p_{\nu} =
(p_{\nu}^{\prime}, p_{\nu 2n})$
\[
\Phi^{\nu}(Y^{\prime}) = \frac{\phi_{\nu}(p_{\nu}^{\prime} -
\psi_{\nu}(p_{\nu})y^{\nu \prime} - \psi_{\nu}(p_{\nu}) \eta
\vert w^{\nu} \vert Y^{\prime})}{-\psi_{\nu}(p_{\nu}) \eta
\vert w^{\nu} \vert} + \frac{\psi_{\nu}(p_{\nu})y^{\nu}_{2n} - p_{\nu
2n}}{-\psi_{\nu}(p_{\nu}) \eta \vert w^{\nu}
\vert}
\]
and
\[
R^{\nu} = \frac{r}{- \psi_{\nu}(p_{\nu}) \eta \vert w^{\nu} \vert} =
\frac{r}{\eta
\vert z_{\nu} - p_{\nu} \vert}.
\]
But from (\ref{z-z_0})
\[
\eta \vert z_{\nu} - p_{\nu} \vert \leq \eta(\vert z_{\nu} - \tilde z_{\nu}
\vert + \vert \tilde z_{\nu}
- p_{\nu} \vert) \leq \eta(\frac{\eta}{1 - \eta} \vert \tilde z_{\nu} - p_{\nu}
\vert + \vert
\tilde z_{\nu} - p_{\nu} \vert) = \frac{\eta}{1-\eta} \vert \tilde z_{\nu} -
p_{\nu} \vert < r
\]
so that $R^{\nu} > 1$. This implies that
\[
\{ \vert W \vert < 1 \} \cap \pa \Om^{\nu}  \subset \{(Y^{\prime}, Y_{2n}) :
Y_{2n} = \Phi^{\nu}(Y^{\prime}), \vert Y^{\prime} \vert < R^{\nu}\}.
\]
By using the properties of $\phi_{\nu}$ and the explicit formula for
$\Phi^{\nu}$ above, it follows that
\[
\{ \vert W \vert < 1 \} \cap \pa \Om^{\nu} = \{Y_{2n} =
\Phi^{\nu}(Y^{\prime})\}
\]
where $\Phi^{\nu} \in C^{\infty}(\{ Y^{\prime} : \vert Y^{\prime} \vert < 1 \})$
and satisfies
\begin{itemize}
\item [(a)] $0 < \Phi^{\nu}(0) < 1$ by (\ref{x_nu}) and (\ref{t_nu}), and
\item [(b)] $\vert \frac{\pa^{\al} \Phi^{\nu}}{\pa Y^{\al}} \vert < M $ for all
$\al = (\al_1, \cdots, \al_n)$ with $\vert \al \vert \leq 6$ if $\vert
Y^{\prime} \vert < 1$.
\end{itemize}
\end{proof}
\noindent Now we modify Step 2 of chapter 5 \cite{LY}, to obtain the following
uniform estimates:
\begin{prop}\label{estimate-der-g}
There exists a constant $C > 0$ and an integer $I $ such that for $1 \leq i, j,
k \leq 2n$

\begin{enumerate}
\item $\vert (\pa g^{\nu} / \pa y_i)( w) \vert \leq C \vert w \vert^{-2n+1}$,

\item $\vert (\pa^2 g^{\nu} / \pa y_i \pa y_j)(w) \vert \leq C \vert w
\vert^{-2n}$,

\item $\vert (\pa^3 g^{\nu} / \pa y_i \pa y_j \pa y_k)(w) \vert \leq C \vert
w\vert ^{-2n - 1}$
\end{enumerate}
for all $\nu \geq I$ and $w \in \ov D^{\nu} \setminus \{0\}$.
\end{prop}

\begin{proof}
The proofs for (1), (2) and (3) are similar and so we prove only (1). Fix $1
\leq i \leq 2n$. Suppose that (1) is not true. Then there exists a sequence
$\{w^{\nu}\}$ such that $w^{\nu} \in D^{\nu} \setminus \{0\}$ and
\begin{equation}\label{g-diverges}
\lim_{\nu \ra \infty} \Big \vert \frac{\pa g^{\nu}}{\pa y_i}(w^{\nu})
\Big \vert \vert w^{\nu} \vert ^{ 2n - 1} = \infty.
\end{equation}
We claim that for all but finitely many $\nu$,
\[
B(w^{\nu}) = \big \{ w \in \mathbf{C}^{n} : \vert w - w^{\nu} \vert < \eta \vert
w^{\nu} \vert \big \}
\]
intersects $\pa D^{\nu}$. Indeed, suppose that $B(w^{\nu}) \cap \pa D^{\nu} =
\emptyset$
for some $\nu$. Then $B(w^{\nu}) \subset D^{\nu}$ and therefore,
\[
g^{\nu}(w) \leq \vert w \vert ^{-2n + 2} \leq (1 - \eta)^{-2n+2} \vert w^{\nu}
\vert ^{-2n+2}, \quad w \in \pa B(w^{\nu}).
\]
Now, by Poisson integral formula, there exists a constant $c_n>0$ independent of
$\nu$ such
that
\[
\Big \vert \frac{\pa g^{\nu} }{\pa y_i} (w^{\nu}) \Big \vert \leq
\frac{c_n}{(1-\eta)^{2n-2}\eta} \vert w^{\nu} \vert ^{-2n + 1}.
\]
But this can be true only for finitely many $\nu$ by (\ref{g-diverges}) and
hence the claim. Therefore, if we let
\[
\Om^{\nu} = S^{\nu} \big(B(w^{\nu}) \cap D^{\nu}\big) = \{\vert W \vert < 1 \}
\cap S^{\nu}(D^{\nu})
\]
then by lemma \ref{implicit-graph}, we can find for all large $\nu$,
functions $\Phi^{\nu}\in C^{\infty}\big(\{Y^{\prime} : \vert Y^{\prime}
\vert < 1\}\big)$ such that
\[
\Om^{\nu} = \{\vert W \vert < 1\} \cap \{ Y = (Y^{\prime}, Y_{2n}) : \vert
Y^{\prime} \vert < 1, Y_{2n} < \Phi^{\nu}(Y^{\prime}) \}
\]
and
\[
\Big \vert \frac{\pa^{\al} \Phi^{\nu}}{\pa Y^{\al}} \Big \vert < M \text{ for
all }
\vert \al \vert \leq N, \text{ if }
\vert
Y^{\prime} \vert < 1.
\]
Since $M$ is independent of $\nu$, by the Arzela-Ascoli theorem, after passing
to
a subsequence, $\{\Phi^{\nu}\}$ together with all
partial derivatives of order up to $6$ converge uniformly on
compact subsets of $\{Y^{\prime} : \vert Y^{\prime} \vert < 1\}$ to a function
$\Phi \in C^6\big(\{Y^{\prime} : \vert Y^{\prime} \vert < 1\}\big)$. Set
\[
\Om = \{ \vert W \vert < 1 \} \cap \big\{ Y = (Y^{\prime}, Y_{2n}) : \vert
Y^{\prime} \vert <1, Y_{2n} < \Phi(Y^{\prime}) \big\}.
\]
Now define the function $u^{\nu}$ on $\Om^{\nu}$ by
\[
u^{\nu}(W) = \vert w_{\nu} \vert ^{2n - 2} (1 - \eta)^{2n-2} g^{\nu}(w)
\]
for $W = (w - w_{\nu}) / (\eta \vert w_{\nu} \vert)$. Then $u^{\nu}$ is harmonic
on $\Om^{\nu}$, continuous up to $\pa \Om^{\nu}$,
and $u_{\nu}(W) = 0$ on $\{ \vert W \vert < 1 \} \cap \pa \Om^{\nu}$. Since
\[
0 < g^{\nu}(w) < \vert w \vert ^{-2n + 2} < (1 - \eta)^{-2n + 2} \vert w^{\nu}
\vert ^{-2n + 2}, \quad w \in B(w^{\nu}) \cap D^{\nu}
\]
we have
\[
0 < u^{\nu}(W) < 1, \quad W \in \Om^{\nu}
\]
By Harnack's theorem, passing
to a subsequence, $\{u^{\nu}\}$ converges uniformly on compact subsets of $\Om$
to a
harmonic function $u$ on $\Om$. From
\cite{LY}*{proposition~5.1}, it follows that
\[
\lim_{\nu \ra \infty} \Big\vert \frac{\pa u^{\nu}}{\pa y_i}(0) \Big\vert =
\frac{\pa u}{\pa y_i}(0)
\]
which is finite. Hence from the definition of $u^{\nu}$,
\[
\lim_{\nu \ra \infty} \Big\vert \frac{\pa g^{\nu}}{\pa y_i}(w^{\nu})
\Big\vert \vert w^{\nu} \vert ^{2n - 1} < \infty
\]
which is a contradiction. Hence (1) must hold.
\end{proof}

\medskip

We now want to modify Step 3 of chapter 4 \cite{LY}. Recall that
\[
\mathcal{E}^{\nu}(r) = \bigcup_{w_0 \in \pa D^{\nu}}\big\{w \in D^{\nu} : \vert
w -
w_0 \vert < r \vert w_0 \vert \big\}
\]
is a collar about $\pa D^{\nu}$ lying in $D^{\nu}$ whose closure does not
contain the origin. Similarly
\[
\mathcal{E}_{\nu}(r) = (T^{\nu})^{-1}\big(\mathcal{E}^{\nu}(r)\big) =
\bigcup_{z_0\in\pa D_{\nu}}\{z \in D_{\nu} : \vert z - z_0 \vert <
r_0 \vert z_0 - p_{\nu} \vert\}
\]
is a collar about $\pa D_{\nu}$ lying in $D_{\nu}$ whose closure does noth
contain the point $p_{\nu}$.

\begin{lem} \label{der_g_nu/grad_g_nu-bdd}
There exist $0 < r_0 < 1$, a constant $C>0$ and an integer $I$ such that
\begin{equation}\label{der_g/grad_g}
\bigg \vert \frac{\pa^2 g^{\nu}}{\pa y_i \pa y_j}(w) \bigg\vert \big\vert
\pa_w g^{\nu}(w) \big\vert^{-1} \leq C \vert w \vert ^{-1}, \quad w \in
\mathcal{E}^{\nu}(r_0)
\end{equation}
for all $\nu \geq I$.
\end{lem}

\begin{proof}
By the relation
\[
g^{\nu}(w) = \psi_{\nu}(p_{\nu})^{2n-2} G_{\nu}(z,p_{\nu}), \quad z = p_{\nu} -
\psi_{\nu}(p_{\nu})
w,
\]
we observe that (\ref{der_g/grad_g}) is equivalent to
\begin{equation}\label{der_G_nu/grad_G_nu}
\bigg \vert \frac{\pa^2 G_{\nu}}{\pa x_i \pa x_j}(z, p_{\nu}) \bigg
\vert \big\vert
\pa_z G_{\nu}(z, p_{\nu}) \big\vert ^{-1} \leq C \vert z - p_{\nu} \vert ^{-1},
\quad
z \in \mathcal{E}_{\nu}(r_0).
\end{equation}
We prove (\ref{der_G_nu/grad_G_nu}) by contradiction. So, suppose that there do
not exist $0<r_0<1$, $C>0$ and integer $I$ such that (\ref{der_G_nu/grad_G_nu})
holds for all $\nu \geq I$.
Then there exist a sequence
$\{z_{0\nu}\}$ with $z_{0\nu} \in \pa D_{\nu}$, and a sequence $\{z_{\nu}\}$
with
\begin{equation}\label{z_nu1}
z_{\nu} \in D_{\nu} \text{ and } \vert z_{\nu} - z_{0\nu} \vert < \frac{1}{\nu}
\vert z_{0\nu} - p_{\nu} \vert, \quad \nu \geq 1
\end{equation}
such that
\begin{equation}\label{der_G_nu/grad_G_nu-ubd}
\bigg\vert \frac{\pa ^2 G_{\nu}}{\pa x_i \pa x_j} (z_{\nu}, p_{\nu})
\bigg\vert \big\vert \pa_z G_{\nu}(z_{\nu}, p_{\nu}) \big\vert^{-1} \geq
\nu \vert z_{\nu} - p_{\nu} \vert^{-1}, \quad \nu \geq 1.
\end{equation}
By passing to a subsequence if necessary, we may assume that
\[
\lim_{\nu \ra \infty} z_{0\nu} = z_0 \in \pa D.
\]
Then, from (\ref{z_nu1}),
\[
\lim_{\nu \ra \infty} z_{\nu} = z_0.
\]
\textit{Claim}: $p_0 = z_0$. Suppose that this is not
true. Then we can find an $\ep > 0$
such that $B(p_0, 2\ep) \cap B(z_0, \ep) = \emptyset$.
Taking $\ep$ sufficiently small and $\nu$ sufficiently large, we can find by
the implicit function theorem a $C^{\infty}$-smooth function $\phi$ on
$B(x_0^{\prime}, \ep)$ and a sequence
$\{\phi_{\nu}\}$ of $C^{\infty}$-smooth functions on $B(x_0^{\prime}, \ep)$ that
converges in $C^{\infty}$-topology on compact subsets of $B(x_0^{\prime}, \ep)$
to $\phi$ such that
\begin{equation}\label{graphs-D_nu}
\begin{cases}
B(z_0, \ep) \cap \pa D = \Big\{\big(x^{\prime}, \phi(x^{\prime})\big):
x^{\prime} \in
B(x_0^{\prime}, \ep) \Big\},\\
B(z_0, \ep) \cap \pa D_{\nu}  =\Big\{\big(x^{\prime},
\phi_{\nu}(x^{\prime})\big):
x^{\prime} \in B(x_0^{\prime}, \ep) \Big\}.
\end{cases}
\end{equation}
Without loss of generality let us assume that all $p_{\nu}$ lie in $B(p_0,
\ep)$. Then
\begin{equation}\label{bd-G_nu}
G_{\nu}(z, p_{\nu}) \leq \vert z - p_{\nu} \vert^{-2n+2} < \ep^{-2n+2}, \quad z
\in B(z_0, \ep) \cap D_{\nu} 
\end{equation}
Now consider the affine map
\[
Z = Sz = \frac{z - z_0}{\ep}
\]
and set
\[
\Om =  S\big(B(z_0, \ep/2) \cap D \Big), \quad \Om_{\nu} = S \big(B(z_0, \ep/2)
\cap D_{\nu}\Big).
\]
Define
\[
h_{\nu}(Z) = \ep^{2n-2} G(z, p_{\nu}), \quad Z \in \Om_{\nu}.
\]
Then $h_{\nu}$ is harmonic on $\Om_{\nu}$, $h_{\nu} = 0$ on $B(0, 1) \cap
\pa \Om_{\nu}$  and by (\ref{bd-G_nu})
\[
0 < h_{\nu}(Z) \leq 1, \quad Z \in \Om_{\nu}.
\]
Therefore, by Harnack's principle, after passing to a subsequence if necessary,
$\{h_{\nu}\}$ converges uniformly on compact subsets of $\Om$ to a positive
harmonic function $h$.  In view of (\ref{graphs-D_nu}), the sequence
$\{h_{\nu}\}$ on $\{\Om_{\nu}\}$ satisfies the hypothesis of
\cite{LY}*{proposition~5.1} and hence
\begin{equation}\label{lim-der-h_nu_1}
\begin{cases}
\lim_{\nu \ra \infty} \vert \pa_{Z} h_{\nu} ( Z_{\nu}) \vert = \vert
\pa_{Z} h(0)
\vert,\\
\lim_{\nu \ra \infty} \Big \vert \frac{\pa^2 h_{\nu}}{\pa \ti X_i \pa \ti
X_j}(Z_{\nu})
\Big \vert = \Big \vert \frac{\pa^2 h}{\pa \ti X_i \pa \ti X_j}(0) \Big
\vert < \infty
\end{cases}
\end{equation}
where $Z_{\nu} = Sz_{\nu}$. By the Hopf lemma,
\[
\vert \pa_{Z} h(0) \vert > 0.
\]
Hence
\[
\lim_{\nu \ra \infty} \frac{\Big \vert \frac{\pa ^2 G_{\nu}}{\pa x_i \pa x_j}
(z_{\nu}, p_{\nu}) \Big \vert}{\vert \pa_z G_{\nu}(z_{\nu}, p_{\nu}) \vert}
\vert z_{\nu} - p_{\nu} \vert = \ep \lim_{\nu \ra \infty} \frac{\Big \vert
\frac{\pa^2 h_{\nu}}{\pa X_i \pa X_j}(Z_{\nu}) \Big \vert}{\vert
\pa_{Z} h_{\nu}
(Z_{\nu}) \vert} \vert z_{\nu} - p_{\nu} \vert = \frac{\Big \vert
\frac{\pa^2
h}{\pa X_i \pa X_j}(0) \Big \vert}{\vert \pa_{Z} h (0) \vert} \vert z_0 - p_0
\vert < \infty
\]
which contradicts (\ref{der_G_nu/grad_G_nu-ubd}). Therefore, we must have $p_0 =
z_0$ and hence the claim.

\medskip

Now define
\[
k_{\nu} = \vert p_{\nu} - z_{0\nu} \vert.
\]
Consider the affine maps $S_{\nu} : \mathbf{C}^{n} \ra \mathbf{C}^{n}$ defined
by
\[
\ti z = S_{\nu}(z) = \frac{z - p_{\nu}}{k_{\nu}}
\]
and let $\ti D _{\nu} = S_{\nu}(D_{\nu})$. A defining function for $\ti D_{\nu}$
is given by
\begin{align*}
\psi_{\nu} \circ S_{\nu}^{-1}(\ti z) & =  \psi_{\nu}(p_{\nu} + k_{\nu} \ti z)\\
& = \psi_{\nu}(p_{\nu}) + 2 k_{\nu} \Re \big( \sum_{\al = 1}^{n}
(\psi_{\nu})_{\al}(p_{\nu}) \ti z_{\al} \big) + k_{\nu}^2 O(1)
\end{align*}
for $\ti z$ on a compact subset of $\mathbf{C}^{n}$. Since $\{\psi_{\nu}\}$
converges in the $C^{\infty}$-topology on compact subsets of $\mathbf{C}^{n}$ to
$\psi$, we note that $O(1)$ is independent of $\nu$. Now
\[
\ti \psi_{\nu} (\ti z) = \frac{\psi_{\nu} \circ S_{\nu}^{-1} (\ti z) }{k_{\nu}}
=
\frac{\psi_{\nu}(p_{\nu})}{k_{\nu}} + 2 \Re \big( \sum_{\al =
1}^{n}(\psi_{\nu})_{\al}(p_{\nu}) \ti z_{\al} \big) +  k_{\nu} O(1)
\]
is again a defining function for $\ti D_{\nu}$. Note that we can find a ball
$B$ centered at $p_0$, positive smooth functions $\phi_{\nu}$ on $B$ such
that
\[
- \psi_{\nu}(p) = \phi_{\nu}(p) d(p, \pa D_{\nu}), \quad p \in B.
\]
By differentiating the above relation, it can be seen that the functions
$\phi_{\nu}$, for all large $\nu$, are unifomly bounded above by a
constant $c > 0$ on possibly a smaller ball $B^{\prime}$ centered
at $p_0$. This implies that
for all large $\nu$,
\[
\Big \vert \frac{\psi_{\nu}(p_{\nu})}{k_{\nu}} \Big \vert \leq \frac{c
d_{\nu}(p_{\nu}, \pa D_{\nu})}{\vert p_{\nu} - z_{0\nu} \vert} \leq c
\]
and hence after passing to a subsequence, $\{\psi_{\nu}(p_{\nu}) /k_{\nu}\}$
converges to a number $\ti c \leq 0$. Thus the functions $\ti \psi_{\nu}$
converge in
the $C^{\infty}$-topology on compact subsets of $\mathbf{C}^{n}$ to the
function
\[
\ti \psi(\ti z) = \ti c + 2 \Re \big( \sum_{\al = 1}^{n} \psi_{\al}(p_0) \ti
z_{\al}
\big).
\]
This implies that the domains $\ti D_{\nu}$ are $C^{\infty}$-perturbation of
the half space
\[
\ti H = \{ \ti z \in \mathbf{C}^{n} : \ti c + 2 \Re \big( \sum_{\al = 1}^{n}
\psi_{\al}(p_0) \ti z_{\al} \big) < 0 \}.
\]
Since $\ti c \leq 0$, it is evident that
\begin{equation}\label{contrad}
0 \in \ov{\ti{\mathcal{H}}}.
\end{equation}

\medskip

\no We will now derive a contradiction by proving that (\ref{contrad}) is
false. 
First, observe that $0 = S_{\nu}(p_{\nu}) \in \ti
D_{\nu}$. Let $\ti g_{\nu}(\ti z)$ be the Green function for $\ti D_{\nu}$ with
pole
at $0$. Then
\begin{equation}\label{ti_g-G}
\ti g_{\nu}(\ti z) = G(z, p_{\nu}) k_{\nu}^{2n - 2}.
\end{equation}
Now let $\ti z_{0 \nu} = S_{\nu}(z_{0 \nu})$. Then $\ti z_{0 \nu} \in \pa \ti
D_{\nu}$
and
\[
\vert \ti z_{0\nu} \vert = \Big\vert \frac{z_{0\nu} - p_{\nu}}{k_{\nu}} \Big
\vert =
1.
\]
Therefore, after passing to a subsequence, $\{\ti z_{0\nu}\}$ converges
to a point $\ti z_0$ with
\[
\vert \ti z_0 \vert = 1.
\]
Evidently, $\ti z_0 \in \pa \ti H$. Also, let $\ti z_{\nu} = S_{\nu}(z_{\nu})$.
Then
\[
\vert \ti z_{\nu} - \ti z_{0 \nu} \vert = \Big \vert \frac{z_{\nu} - z_{0 \nu
}}{k_{\nu}} \Big \vert < \frac{1}{\nu}
\]
by (\ref{z_nu1}). Therefore,
\[
\lim_{\nu \ra \infty} \ti z_{\nu} = \ti z_0.
\]
Now we derive the contradiction by considering the following two cases:

\medskip

\textit{Case I}. $0 \in \ti H$. Let $\ti g(\ti z)$ be the Green
function for $\ti H$ with pole at $0$. Then by
corollary~\ref{bdy-conv-Green-fn},
\[
\begin{cases}
\lim_{\nu \ra \infty} \vert \pa_{\ti z}\ti g_{\nu}(\ti z_{\nu}) \vert = \vert
\pa_{\ti z}
\ti g(\ti z_0) \vert > 0,\\
\lim_{\nu \ra \infty} \frac{\pa^2 \ti g_{\nu}}{\pa \ti x_k \pa \ti x_l} (\ti
z_{\nu}) =
\frac{\pa^2 \ti g}{\pa \ti x_k \pa \ti x_l}(\ti z_0) \neq \infty.
\end{cases}\]
Now from (\ref{ti_g-G}),
\[
\lim_{\nu \ra \infty} \bigg \vert \frac{\pa^2 G}{\pa x_i \pa x_j} (z_{\nu},
p_{\nu}) \bigg \vert \big\vert \pa_{z} G (z_{\nu}, p_{\nu}) \big\vert^{-1} \vert
z_{\nu} - p_{\nu}
\vert = \lim_{\nu \ra \infty} \bigg \vert \frac{\pa^2 \ti g_{\nu}}{\pa \ti x_i
\pa \ti x_j}(\ti z_{\nu}) \bigg \vert \big\vert \pa_{\ti z} \ti g_{\nu} (\ti
z_{\nu}) \big\vert^{-1} \vert \ti z_{\nu}
\vert < \infty
\]
which contradicts (\ref{der_G_nu/grad_G_nu-ubd}) and hence $0 \not \in
\ti{\mathcal{H}}$.

\medskip

\textit{Case II}. $0 \in \pa \ti H$. By the implicit function
theorem we can find a ball $B(\ti z_0, \ep)$, a $C^{\infty}$-smooth function
$\phi$
on $B(\ti x_0^{\prime}, \ep)$ and a sequence $\{\phi_{\nu}\}$ of
$C^{\infty}$-smooth
functions on $B(\ti x_0^{\prime}, \ep)$ that converges in the
$C^{\infty}$-topology on
compact subsets of $B(\ti x_0^{\prime}, \ep)$ to $\phi$ such that
\begin{equation}\label{graphs-ti-D_nu}
\begin{cases}
B(\ti x_0, \ep) \cap \pa \ti{\mathcal{H}} = \Big\{ \big(\ti x^{\prime},
\phi(\ti x^{\prime})\big) : \ti x^{\prime} \in B(\ti x_0^{\prime}, \ep) \Big\}\\
B(\ti x_0, \ep) \cap \pa \ti D_{\nu}= \Big\{ \big(\ti x^{\prime},
\phi_{\nu}(\ti x^{\prime})\big) : \ti x^{\prime} \in B(\ti x_0^{\prime}, \ep)
\Big\}
\end{cases}
\end{equation}
Without loss of generality let us assume that $\ep < 1/2$. Then, since $\vert
\ti z_0 \vert
= 1$,
\begin{equation}\label{bd-g_nu(w)}
g_{\nu}(\ti z) < \vert \ti z \vert^{-2n+2} < 2^{2n-2}, \quad \ti z \in B(\ti
z_0, \ep) \cap \ti
D_{\nu}.
\end{equation}
Now, consider the affine map
\[
\ti Z = S\ti z = \frac{\ti z - \ti z_0}{\ep}
\]
and set
\[
\Om = S \big(B(\ti z_0, \ep) \cap \ti H\big), \quad \Om_{\nu} = S\big(B(\ti z_0,
\ep)
\cap \ti D_{\nu} \big).
\]
Define
\begin{equation}\label{defn-h-ti-Z}
h(\ti Z) = 2^{-2n+2}g(\ti z), \quad \ti Z \in \Om
\end{equation}
and
\begin{equation}\label{defn-h_nu-ti-Z}
h_{\nu}(\ti Z) =2^{-2n+2} g_{\nu}(\ti z), \quad \ti Z \in \Om_{\nu}.
\end{equation}
Then $h_{\nu}$ is a positive harmonic funtion on $\Om_{\nu}$
and satisfies $h_{\nu} = 0$ on $B(0, 1) \cap \pa \Om_{\nu}$. Moreover, by
(\ref{bd-g_nu(w)}),
\[
0 < h_{\nu}(\ti Z) <1, \quad \ti Z \in \Om_{\nu}.
\]
By passing to a subsequence if necessary, it follows from Harnack's principle
that
$\{h_{\nu}\}$ converges uniformly on compact subsets of $\Om$ to a positive
harmonic function $h$ which satisfies $h = 0$ on $B(0, 1) \cap \pa \Om$. In view
of (\ref{graphs-ti-D_nu}), the sequence $\{h_{\nu}\}$ satisfies the hypothesis
of \cite{LY}*{proposition~5.1} and hence from (\ref{ti_g-G}) and
(\ref{defn-h_nu-ti-Z}),
\begin{multline*}
\lim_{\nu \ra \infty} \bigg \vert \frac{\pa^2 G}{\pa x_i \pa x_j}(z_{\nu},
p_{\nu})\bigg \vert \big\vert \pa_z G(z_{\nu}, p_{\nu}) \big\vert^{-1} \vert
z_{\nu} -
p_{\nu} \vert \\
 = \ep \lim_{\nu \ra \infty} \bigg\vert \frac{\pa^2
h_{\nu}}{\pa \ti X_i \pa \ti X_j}(\ti Z_{\nu})\bigg\vert \big\vert \pa_{\ti Z}
h_{\nu}(\ti Z_{\nu}) \big\vert^{-1}\vert \ti z_{\nu}
\vert = \ep \bigg\vert \frac{\pa^2 h}{\pa \ti X_i \pa \ti X_j} (0) \bigg
\vert \big \vert \pa_{\ti Z} h(0) \big\vert^{-1}
\end{multline*}
where $\ti Z_{\nu} = S \ti z_{\nu}$. Now by the reflection principle, $h$
extends as a harmonic function to a
neighbourhood of $0$ and hence the quantity on the extreme right of the
above equation is finite. This contradicts (\ref{der_G_nu/grad_G_nu-ubd}) and
hence $0 \not \in \pa \ti{\mathcal{H}}$.

By Case I and Case II,  $0 \not \in \ov{\ti{\mathcal{H}}}$ which contradicts
(\ref{contrad}). Therefore (\ref{der_G_nu/grad_G_nu}) holds and the lemma is
proved.
\end{proof}

Recall that if $r>0$  and $I$ are as in lemma \ref{exist_de}, then the
function $F^{\nu}(w)$ is defined and smooth on the collar
$\mathcal{E}^{\nu}(r)$.

\begin{prop}\label{est-F_nu}
There exists $0 < r < 1$, a constant $C > 0$ and an integer $I$ such that
\begin{enumerate}
\item $\vert F^{\nu}(w) \vert < C (1 + \vert w \vert^{-1}) \vert w \vert ^{-2n +
3}$,

\item $\vert (\pa F^{\nu} / \pa y_i) (w) \vert < C (1+ \vert w \vert^{-1}) \vert
w \vert^{-2n + 2}$,

\item $\vert (\pa^2 F^{\nu} / \pa y_i \pa y_j) (w) \vert < C (1+ \vert w
\vert^{-1})\vert w \vert
^{-2n + 1}$
\end{enumerate}
for all $\nu \geq I$ and $w \in \mathcal{E}^{\nu}(r)$.
\end{prop}

\begin{proof}
Choose $m>0$, $0<r<1$ and $I$ as in lemma \ref{exist_de}. Choose $M>0$ as in
lemma \ref{upper-bound-der-f-and-f_j}. Modify $I$ and choose a constant $C$ so
that  proposition \ref{estimate-der-g} holds. Modify $r$ and $I$ so that
lemma \ref{der_g_nu/grad_g_nu-bdd} holds. Now fix $\nu \geq I$.

\medskip

(1) Let $w \in \mathcal{E}^{\nu}(r)$, $\vert w \vert > 1$. Then
by lemma \ref{exist_de}, lemma \ref{upper-bound-der-f-and-f_j} and proposition
\ref{estimate-der-g},
\[
\vert F^{\nu}(w) \vert = \frac{\vert \frac{\pa f_{\nu}}{\pa p_{\gamma}}(p_{\nu},
w)\vert}{\vert \pa_w f_{\nu}(p_{\nu}, w) \vert} \vert \pa_w g^{\nu}(w) \vert
\leq
\frac{M (1+ \vert w \vert^{-1})\vert w \vert^2}{m} C \vert w \vert^{-2n+1} = C_2
(1 + \vert w \vert ^{-1})\vert w \vert^{-2n +
3}.
\]
where $C_1 = M C/m$ is independent of $\nu$ and $w$.

\medskip

(2) Differentiating $F^{\nu}(w)$ with respect to $y_i$,
\begin{equation}\label{1st-der-F}
\frac{\pa F^{\nu}}{\pa y_i} = \frac{-\frac{\pa^2 f_{\nu}}{\pa p_{\gamma} \pa
y_i}}{\vert \pa_w f_{\nu} \vert} \vert \pa_w g^{\nu} \vert + \frac{1}{4}
\frac{\pa f_{\nu}}{\pa p_{\gamma}} \frac{\sum_{k=1}^{2n} \frac{\pa f_{\nu}}{\pa
y_k} \frac{\pa^2 f_{\nu}}{\pa y_k \pa y_i}}{\vert \pa_w f_{\nu} \vert^3} \vert
\pa_w g^{\nu} \vert - \frac{1}{4} \frac{\pa f_{\nu}}{\pa p_{\gamma}}
\frac{1}{\vert \pa_w f_{\nu} \vert}
\frac{\sum_{k=1}^{2n} \frac{\pa g_{\nu}}{\pa y_k} \frac{\pa^2 g_{\nu}}{\pa y_k
\pa y_i}}{\vert \pa_w g^{\nu} \vert}
\end{equation}
Thus for $w \in \mathcal{E}^{\nu}(r)$ with $\vert w \vert > 1$, by lemma
\ref{exist_de}, lemma \ref{upper-bound-der-f-and-f_j} and proposition
\ref{estimate-der-g}, and the fact that
\[
\frac{\frac{\pa g^{\nu}}{\pa y_k}}{\vert\pa_w g^{\nu}\vert} \leq 2,
\]
we have
\begin{multline*}
\Big \vert \frac{\pa F^{\nu}}{\pa y_i} (w) \Big\vert \leq \frac{M (1+ \vert w
\vert^{-1})
\vert w \vert}{m} C
\vert w \vert ^{-2n + 1} + \frac{1}{4} M (1+ \vert w \vert^{-1})
\vert w \vert^2 \frac{2n M M \vert w \vert^{-1}}{m^3} C \vert w
\vert^{-2n+1}\\ + \frac{1}{4} M (1+ \vert w \vert^{-1}) \vert w \vert^2
\frac{1}{m} 2n 2 C \vert w \vert^{-2n} \leq C_2 (1+ \vert w \vert^{-1})\vert w
\vert^{-2n + 2}.
\end{multline*}

\medskip

(3) In order to prove this estimate, we differentiate (\ref{1st-der-F}) with
respect to $y_j$ and estimate as before. All terms except those of the form
\[
\frac{\frac{\pa f_{\nu}}{\pa p_{\gamma}}}{\vert \pa_{w}f_{\nu} \vert}
\frac{\frac{\pa^2 g_{\nu}}{\pa y_k \pa y_i} \frac{\pa^2 g_{\nu}}{\pa y_l \pa
y_i}}{\vert \pa _w g_{\nu} \vert}  \quad \text{or}
\quad \frac{\frac{\pa f_{\nu}}{\pa p_{\nu}}}{\vert \pa_{w}f_{\nu} \vert}
\frac{\frac{\pa g_{\nu}}{\pa y_k} \frac{\pa g_{\nu}}{\pa y_l} \frac{\pa^2
g_{\nu}}{\pa y_k \pa y_i} \frac{\pa^2 g_{\nu}}{\pa y_l \pa y_j}}{\vert \pa _w
g_{\nu} \vert^3}
\]
can be estimated from the above by $\text{const.} (1+ \vert w \vert^{-1}) \vert
w \vert^{-2n+1}$ for
$w \in \mathcal{E}^{\nu}(r)$. Also by lemma \ref{der_g_nu/grad_g_nu-bdd}, the
above terms can be esitmated from the above by $\text{const.} (1+ \vert w
\vert^{-1}) \vert w \vert^{-2n+1}$ for $w \in \mathcal{E}^{\nu}(r_0)$.
\end{proof}

We now modify the Steps 4 and 5 of chapter 5 \cite{LY} to find an upper bound
for $\frac{\pa^2 g_{\nu}}{\pa \ov w_{\al}\pa p_{\ga}}(p_{\nu}, w)$.

\begin{prop}\label{F*}
There exist $0 < r < 1$ and an integer $I$ such that for $\nu \geq I$ and $w_0
\in \pa D^{\nu}$, we can find a function $F^{*}(w)$ (depending on the parameters
$\nu$ and $w_0$) of class $C^2$ on
\[
E = \{w \in D^{\nu} : \vert w - w_0 \vert < r \vert w_0
\vert\}
\]
such that
\[
H_E F^{*}(w) = \frac{\pa g_{\nu}}{\pa p_{\gamma}}(p_{\nu}, w), \quad w \in E.
\]
Moreover, there exists a constant $C>0$ independent of $\nu$
and $w_0 \in \pa D^{\nu}$ such that
\begin{enumerate}
\item $\vert F^{*}(w) \vert < C (1 + \vert w_0 \vert^{-1}) \vert w_0 \vert ^{-2n
+ 3}$ in $E$.

\item $\vert (\pa F^{*}/\pa y_i)(w_0) \vert < C(1 + \vert w_0 \vert^{-1})  \vert
w_0 \vert ^{-2n + 2}$,
$i=1, \ldots, n$.

\item $\vert \Delta _w F^{*}(w) \vert < C (1 + \vert w_0 \vert^{-1}) \vert w_0
\vert^{-2n+1}$ in $E$.
\end{enumerate}
\end{prop}

\begin{proof}
Choose $0 < r < 1$,  a constant $C$ and an integer $I$ as in proposition
\ref{est-F_nu}. Now fix $\nu \geq I$
and $w_0 \in \pa
D^{\nu}$ and let
\[
B = \{ w : \vert w - w_0 \vert < r \vert w_0 \vert\}.
\]
Then $E =B \cap D^{\nu}$. Since
\[
\frac{\pa g_{\nu}}{\pa p_{\gamma}}(p_{\nu}, w) = H_{D^{\nu}} F^{\nu} (w)
\]
on $D^{\nu}$, the function $(\pa g_{\nu}/\pa p_{\gamma})(p_{\nu}, w)$ is
harmonic
on $E$ with boundary values
\begin{equation}\label{bdy-value}
\begin{cases}
F^{\nu} & ; \quad \text{if } w \in B \cap \pa D^{\nu},\\
H_{D^{\nu}}F^{\nu} & ; \quad \text{if } w \in \pa B \cap  D^{\nu}.
\end{cases}
\end{equation}
Let $u$ be the harmonic function on $E$ with boundary values
\[
u(w) =
\begin{cases}
0 & ; \quad \text{if } w \in  B \cap \pa D^{\nu},\\
H_{D^{\nu}}F^{\nu} - F^{\nu}  & ; \quad \text{if } w \in \pa B \cap D^{\nu}
\end{cases}
\]
and set
\[
F^{*}(w) = F^{\nu}(w) + u(w), \quad w \in E.
\]
Then
\[
H_{E} F^{*} = H_{E} F^{\nu} + u
\]
is a harmonic function on $E$ with boundary values (\ref{bdy-value}) and hence
\[
H_{E} F^{*} = \frac{\pa g_{\nu}}{\pa p_{\gamma}}(p_{\nu}, w)
\]
on $E$ and this proves the first part of the proposition.

\medskip

To prove the second part, note that by proposition \ref{est-F_nu} and continuity
of the function
\[
\vert F^{\nu}(w)\vert (1 + \vert w \vert^{-1})^{-1} \vert w \vert ^{2n-3}
\]
up to $\ov
{\mathcal{E}}^{\nu}(r)$,
we have
\[
\vert F^{\nu}(w) \vert \leq C (1 + \vert w \vert ^{-1} )\vert w \vert^{-2n+3},
\quad w \in \ov{\mathcal{E}}^{\nu}(r).
\]
In particular, the above hods for $w \in \ov E$. Also, since $(1 + \vert w \vert
^{-1} ) \vert w \vert^{-2n+3}$ is superharmonic on $\mathbf{C}^{n}$, this also
implies that
\begin{equation}
\vert H_{D^{\nu}} F^{\nu}(w) \vert \leq C (1 + \vert w \vert ^{-1} ) \vert w
\vert^{-2n+3}, \quad w \in \ov D^{\nu}.
\end{equation}
Therefore,
\[
\vert H_{D^{\nu}} F^{\nu} (w) - F^{\nu}(w) \vert \leq 2C(1 + \vert w \vert ^{-1}
)\vert w \vert^{-2n+3},
\quad w \in \pa B \cap D^{\nu}
\]
which implies that
\[
\vert u(w) \vert \leq 2C(1 + \vert w \vert ^{-1} ) \vert w \vert^{-2n+3}, \quad
w
\in \ov E.
\]
Since $E \subset \{w : \vert w - w_0 \vert < r \vert w_0 \vert \}$,
\[
\vert F^{*}(w) \vert \leq \vert F_{\nu}(w) \vert + \vert u(w) \vert \leq 3 C (1
+ \vert w \vert ^{-1} ) \vert w \vert^{-2n+3} \leq 3 C (1-r)^{-2n+2}( 1 + \vert
w_0 \vert^{-1}) \vert w_0 \vert^{-2n+3}, \quad w \in E
\]
which proves (1).

\medskip

To prove (2), note that from the above calculation
\[
\vert u(w) \vert \leq 2 C (1-r)^{-2n+2}(1 + \vert w_0 \vert^{-1}) \vert w_0
\vert^{-2n+3}, \quad w \in E.
\]
Also $u(w) = 0$ for $w \in B \cap \pa D^{\nu}$. Moreover by
lemma \ref{external_ball}, we can modify the integer $I$ if necessary, to find a
$\rho > 0$ which is
independent of $\nu$ and $w_0$ such that there exists a ball of radius $\rho
\vert
w_0 \vert$ which is externally tangent to $\pa D^{\nu}$ at $w_0$. Hence taking
$R=\min(\rho \vert w_0 \vert, r \vert w_0 \vert )$ in Step 2 of chapter 4
\cite{LY}, we can find a constant $c$ independent of $D^{\nu}$ and $u$ such
that
\[
\vert \pa_w u (w_0) \vert < \frac{ 2 c C (1-r)^{-2n+2} (1 + \vert w_0
\vert^{-1})
\vert
w_0 \vert^{-2n + 3} }{
\min(r \vert w_0 \vert, \rho \vert w_0 \vert)} = \ti C (1 + \vert w_0
\vert^{-1})  \vert w_0 \vert^{-2n + 2}
\]
where $\ti C$ is independent of $\nu$ and $w_0 \in \pa D^{\nu}$. This together
with proposition \ref{est-F_nu} implies
\[
\Big \vert \frac{\pa F^{*} }{\pa y_i} (w_0) \Big \vert \leq \Big \vert
\frac{\pa F_{\nu}}{\pa y_i}(w_0) \Big \vert + \Big \vert \frac{\pa u}{\pa
y_i}(w_0)
\Big \vert \leq (C + \ti C) (1 + \vert w_0 \vert^{-1}) \vert w_0 \vert ^{-2n +
2}
\]
which proves (2).

\medskip

Finally using the fact that $u$ is harmonic we obtain from proposition
\ref{est-F_nu} that
\[
\vert \Delta_w F^{*} (w) \vert = \vert \Delta_w F^{\nu}(w) \vert \leq n C (1 +
\vert w\vert^{-1}) 
\vert w
\vert ^{-2n + 1} \leq n C (1 - r)^{-2n }(1 + \vert w_0 \vert^{-1})  \vert w_0
\vert ^{-2n + 1} , \quad
w \in E
\]
and this proves (3).
\end{proof}

\begin{prop}\label{estimate-mixed-2nd-der-g}
There exist a constant $C>0$ and an integer $I$ such that
\begin{equation}\label{mixed-2nd-der-g}
\Big \vert \frac{\pa^2 g_{\nu}}{\pa \ov w_{\al} \pa p_{\gamma}} (p_{\nu}, w)
\Big \vert < C (1 + \vert w \vert ^{-1}) \vert w \vert ^{-2n + 2}
\end{equation}
for all $\nu \geq I$ and $ w \in \ov D^{\nu}$.
\end{prop}
\bigskip

\begin{proof} \ref{estimate-mixed-2nd-der-g}.
Let $0< r < 1$, $C> 0$ and $I$ be as in proposition \ref{F*} and fix $\nu \geq
I$. By the maximum
principle, it suffices to
prove (\ref{mixed-2nd-der-g}) for $w_0 \in \pa D^{\nu}$. Given
such $w_0$, we let $F^{*}$ be a $C^2$-smooth function on
\[E = \{ w \in D^{\nu} : \vert w - w_0 \vert < r \vert w_0 \vert \}\]
satisfying the estimates of proposition \ref{F*}. Now consider the affine map
\[
W = S(w) = \frac{w - w_0}{r \vert w_0 \vert}
\]
and let $\Om = S(E)$. Define the functions $u$ and $h$ on $\Om$ by setting
\[
u(W) = \frac{\pa g_{\nu}}{\pa p_{\gamma}}(p_{\nu}, w) \quad \text{and}  \quad
h(W) = F^{*}(w).
\]
Then $u = H_{\Om}h$ on $\Om$ and by proposition \ref{F*},
\begin{align*}
\text{(1)} & \vert h(W) \vert < C (1 + \vert w_0 \vert^{-1})\vert w_0
\vert ^{-2n + 3} \quad
\text{in $\Om$},\\
\text{(2)} & \Big \vert \frac{\pa h}{\pa Y_i}(0) \Big \vert = \Big \vert
\frac{\pa F^{*}}{\pa y_i}(w_0) \Big \vert r \vert w_0 \vert < C r (1 + \vert w_0
\vert^{-1})\vert w_0
\vert ^{-2n + 3}, \text{ and}\\
\text{(3)} & \vert \Delta_{W} h (W)\vert = \vert \Delta_{w}
F^{*}(w) \vert r^2 \vert w_0 \vert ^2 \leq C r^2 (1 + \vert w_0 \vert^{-1})
\vert w_0 \vert ^{-2n + 3} \leq C r  (1 + \vert w_0 \vert^{-1})
\vert w_0 \vert ^{-2n + 3}
\quad \text{in $\Om$}.
\end{align*}
By lemma \ref{external_ball}, we can modify the integer $I$ to find a $\rho >
0$ which is independent of $\nu $ and $w_0$ such that there exists a ball $B$ of
radius
$\rho \vert  w_0 \vert$ which is externally tangent to $\pa D^{\nu}$ at $w_0$.
Setting
$T(B) = \tilde{B}$, we see that the ball $\tilde{B} \subset \mathbf{C}^{n}
\setminus \Om$ has radius $\rho / r$ and is tangent to $\pa \Om$ at
$0$. Let $\tilde{B}_2$ be the ball with centre same as $\tilde{B}$ and
radius $\rho/r+2$. Hence by \cite{LY}*{pp 60, lemma 5.1$^{\prime}$, }, there
exists a constant $M$ depending only on $\rho/r$ such that
\[
\vert \pa_{\ov W} u(0) \vert \leq M C (1 + \vert w_0 \vert^{-1}) \vert w_0 \vert
^{-2n+3}
\]
Since
\[
\frac{\pa u}{\pa \ov W_{\al}}(0) = \frac{\pa ^ 2 g_{\nu}}{\pa \ov w_{\al} \pa
p_{\gamma}}(p, w_0) r \vert w_0 \vert,
\]
we have
\[
\Big \vert \frac{\pa ^ 2 g_{\nu}}{\pa \ov w_{\al} \pa p_{\gamma}}(p_{\nu}, w_0)
\Big
\vert
\leq  \frac{M C}{r} (1 + \vert w_0 \vert^{-1})\vert w_0 \vert ^{-2n+2}
\]
which proves the proposition.
\end{proof}

\medskip

\begin{prop}\label{conv-cpt-2nd-der-g_nu}
Let $w^{\nu} \in \pa D^{\nu}$ be such that $\{w^{\nu}\}$ converges to $w^0 \in
\pa
\mathcal{H} = \pa D(p_0)$. Then
\[
\lim_{\nu \ra \infty} \frac{\pa ^ 2 g_{\nu}}{\pa \ov w_{\al} \pa
p_{\gamma}}(p_{\nu}, w^{\nu}) = \frac{\pa ^ 2 g}{\pa \ov w_{\al} \pa
p_{\gamma}}(p_0, w^{0}).
\]
\end{prop}
\begin{proof}
This follows from standard boundary elliptic regularity arguments and the fact
that $D^{\nu}$ is $C^{\infty}$-close to $D$.
\end{proof}

\medskip

\begin{prop}\label{conv-2nd-der-sy}
$
\displaystyle \lim_{\nu \ra \infty} \frac{\pa^2 \la_{\nu}}{\pa p_{\ga} \pa \ov
p_{\ga}}(p_{\nu}) =
\frac{\pa^2 \la}{\pa p_{\ga} \pa \ov p_{\ga}}(p_0).
$
\end{prop}

\begin{proof}
By proposition \ref{varn-formula} and (\ref{conv-1st-int}), we only need to
prove that
\begin{multline}\label{conv-2nd-int}
\lim_{\nu \ra \infty} \int_{\pa D^{\nu}} k_1^{\nu \gamma}(w) \frac{\frac{\pa
g^{\nu}}{\pa \ov w_{\al}}(w)}{\vert \pa_{w} g^{\nu}(w) \vert} \frac{\pa^2
g_{\nu}}{\pa w_{\al} \pa \ov p_{\gamma}}(p_{\nu},w) \frac{\pa g^{\nu}}{\pa
n_w}(w) \, dS_{w}\\ = \int_{\pa \mathcal{H}} k_1^{\ga}(p_0, w) \frac{\frac{\pa
g}{\pa
\ov w_{\al}}(p_0, w)}{\vert \pa_w g (p_0, w)\vert} \frac{\pa^2 g}{\pa w_{\al}
\pa \ov
p_{\ga}}(p_0, w) \frac{\pa g}{\pa n_w}(p_0, w) \, dS_w.
\end{multline}
Let $R>1$. Then by the arguments of the proof of proposition \ref{estimate2}
together with proposition
\ref{conv-cpt-2nd-der-g_nu}, we have
\begin{multline}\label{conv-2nd-int-cpt}
\lim_{\nu \ra \infty} \int_{B(0, R) \cap \pa D^{\nu}} k_1^{\nu \gamma}(w)
\frac{\frac{\pa g^{\nu}}{\pa \ov w_{\al}}(w)}{\vert \pa_{w} g^{\nu}(w) \vert}
\frac{\pa^2 g_{\nu}}{\pa w_{\al} \pa \ov p_{\gamma}}(p_{\nu},w) \frac{\pa
g^{\nu}}{\pa n_w}(w) \, dS_{w}\\
 = \int_{B(0, R) \cap \pa \mathcal{H}} k_1^{\ga}(p_0,w) \frac{\frac{\pa g}{\pa
\ov
w_{\al}}(p_0,w)}{\vert \pa_w g (p_0,w)\vert} \frac{\pa^2 g}{\pa w_{\al} \pa \ov
p_{\ga}}(p_0, w) \frac{\pa g}{\pa n_w}(p_0,w) \, dS_w.
\end{multline}
To estimate the above integrals outside $B(0, R)$, note that by corollary
\ref{upper_boun_k,k_j}, there exist a constant $C$ and an integer $I$ such that
\[
\vert k_1^{\nu \ga}(w)\vert \leq C \vert w \vert^2, \quad w \in \pa D^{\nu},
\vert w \vert > 1
\]
for $\nu \geq I$. In view of proposition \ref{estimate-mixed-2nd-der-g}, we
can modify $C$ and $I$ so that
\[
\Big \vert \frac{\pa^2 g_{\nu}}{\pa \ov w_{\al} \pa p_{\gamma}} (p_{\nu}, w)
\Big \vert \leq C \vert w \vert ^{-2n+2}, \quad w \in \pa D^{\nu}, \vert w \vert
> 1
\]
for $\nu \geq I$. Therefore,
\begin{multline}\label{bound-2nd-int}
\left \vert  \int_{B^c(0, R) \cap \pa D^{\nu}} k_1^{\nu}(w) \frac{\pa^2
g_{\nu}}{\pa w_{\al} \pa \ov p_{\gamma}}(p_{\nu}, w)
\frac{\frac{\pa g^{\nu}}{\pa \ov w_{\al}}(w)}{\vert \pa_w g^{\nu}(w) \vert}
\frac{\pa
g^{\nu}}{\pa n_w}(w) dS_w\right \vert\\
\leq  C^2 R^{-2n+4} \int_{B^c(0, R) \cap \pa D^{\nu}}\Big(-\frac{\pa
g^{\nu}}{\pa n_w}(w) \Big ) dS_w.
\end{multline}
for $\nu \geq I$. Again
\[
\int_{B^c(0, R) \cap \pa D^{\nu}}\Big(-\frac{\pa g^{\nu}}{\pa n_w}(w) \Big )
dS_w
\leq \int_{\pa D^{\nu}}\Big(-\frac{\pa g^{\nu}}{\pa n_w}(w) \Big ) dS_w
= 2(n-1) \sigma_{2n}
\]
and hence from (\ref{bound-2nd-int})
\begin{equation}\label{2nd-int-small}
\left \vert  \int_{B^c(0, R) \cap \pa D^{\nu}} k_1^{\nu}(w) \frac{\pa^2
g_{\nu}}{\pa w_{\al} \pa \ov p_{\gamma}}(p_{\nu}, w)
\frac{\frac{\pa g^{\nu}}{\pa \ov w_{\al}}(w)}{\vert \pa_w g^{\nu}(w) \vert}
\frac{\pa
g^{\nu}}{\pa n_w}(w) dS_w\right \vert = O(R^{-2n+4})
\end{equation}
uniformly for all $\nu \geq I$. Also by (\ref{estimates}), we can modify the
above constant $C$ so that
\[
\vert k_1^{\ga}(p_0, w) \vert \leq C \vert w \vert^2 \quad \text{and} \quad \Big
\vert \frac{\pa^2 g}{\pa \ov w_{\al} \pa p_{\gamma}} (p_{0}, w)
\Big \vert \leq C \vert w \vert ^{-2n+2}
\]
for $w \in \pa \mathcal{H}$ with $\vert w \vert > 1$. As above we obtain
\begin{equation}\label{2nd-int-small2}
\left \vert \int_{B^{c}(0, R) \cap \pa \mathcal{H}}   k_1^{\ga}(p_0, w)
\frac{\frac{\pa g}{\pa \ov w_{\al}}(p_0, w)}{\vert \pa_w g (p_0, w)\vert}
\frac{\pa^2 g}{\pa w_{\al} \pa \ov p_{\ga}}(p_0, w) \frac{\pa g}{\pa n_w}(p_0,
w) \, dS_w \right
\vert = O(R^{-2n+4}).
\end{equation}
From (\ref{conv-2nd-int-cpt}), (\ref{2nd-int-small}) and (\ref{2nd-int-small2})
it follows that (\ref{conv-2nd-int}) holds.
\end{proof}

\medskip

\noindent{\textit{Proof of Theorem 1.3}.} In view of proposition
\ref{estimate2}, we only need to prove that
\[
\lim_{\nu \ra \infty} \frac{\pa^2 \la_{\nu}}{\pa p_{\al} \pa \ov
p_{\be}}(p_{\nu}) = \frac{\pa^2 \la}{\pa p_{\al} \pa \ov p_{\be}}(p_0).
\]
But this follows from proposition \ref{conv-2nd-der-sy} by a unitary change of
coordinates. \hfill $\square$

\section{Holomorphic sectional curvature}
\noindent In this section we prove theorem 1.1 by deriving the asymptotics of
the terms
in (\ref{cvr}).

\begin{lem} \label{asympt-g_nu}
We have
\begin{enumerate}
\item $\lim_{\nu \ra \infty} (g_{\nu})_{\al \ov \be}(p_{\nu})
\big(\psi_{\nu}(p_{\nu})\big)^2 = (2n-2)\psi_{\al}(0) \psi_{\ov \be}(0)$,

\item $\lim_{\nu \ra \infty} \frac{\pa (g_{\nu})_{\al \ov \be}}{\pa
z_{\gamma}}(p_{\nu}) \big(\psi_{\nu}(p_{\nu})\big)^3 = -2 (2n -2 )
\psi_{\al}(0) \psi_{\ov \be}(0) \psi_{\gamma}(0)$,

\item $\lim_{\nu \ra \infty} \frac{\pa ^2 (g_{\nu})_{\al \ov \be}}{\pa
z_{\gamma} \pa z_{\ov \delta}}(p_{\nu})\big(\psi_{\nu}(p_{\nu})\big)^4  =
6(2n-2)\psi_{\al}(0) \psi_{\ov \be}(0) \psi_{\delta}(0)$.
\end{enumerate}
\end{lem}
\begin{proof}
Let $\mathcal{H}$ be the half space
\[
\mathcal{H} = \Big\{ z \in \mathbf{C}^{n} : 2 \Re \Big( \sum_{i=1}^{n}
\psi_{i}(0) z_i\Big) - 1 < 0\Big\} = \{ z \in \mathbf{C}^{n} : 2 \Re z_n - 1 < 0
\}.
\]
From \cite{BV}*{(1.4)}, the Robin function for $\mathcal{H}$ is given by
\[ \La_{\mathcal{H}}(z) = - \bigg( \frac{ \big| \pa \psi(0) \big|} {2 \Re
\big(\sum_{i=1}^{n} \psi_{i}(0) z_{i} \big)
-1} \bigg)^{2n-2} = - \bigg( 2 \Re \Big(\sum_{i=1}^{n} \psi_{i}(0) z_{i} \Big)
-1 \bigg)^{-2n+2}\]
so that
\begin{itemize}
\item $\La_{\mathcal{H}}(0) = -1$,

\item $(\La_{\mathcal{H}})_{a}(0) = -(2n-2)\psi_{a}(0)$,

\item $(\La_{\mathcal{H}})_{ab}(0) = -(2n-2)(2n-1) \psi_{a}(0) \psi_{b}(0)$,

\item $(\La_{\mathcal{H}})_{abc}(0) = -(2n-2)(2n-1)(2n) \psi_{a}(0) \psi_{b}(0)
\psi_{c}(0)$ and

\item $(\La_{\mathcal{H}})_{abcd}(0) = -(2n-2)(2n-1)(2n)(2n+1) \psi_{a}(0)
\psi_{b}(0) \psi_{c}(0) \psi_{d}(0)$
\end{itemize}
where the indices $a$, $b$, $c$, $d$ refer to either holomorphic or conjugate
holomorphic derivatives. Hence by theorem 1.2, we get
\begin{itemize}
\item $\La_{\nu}(p_{\nu}) \big(\psi_{\nu}(p_{\nu}))^{2n-2} \ra -1$,

\item $\La_{\nu a}(p_{\nu}) \big(\psi_{\nu}(p_{\nu})\big)^{2n-1} \ra (2n-2)
\psi_{a}(0)$,

\item $\La_{\nu ab}(p_{\nu}) \big(\psi_{\nu}(p_{\nu}) \big)^{2n} \ra
-(2n-2)(2n-1) \psi_{a}(0) \psi_{b}(0)$,

\item $\La_{\nu abc}(p_{\nu}) \big(\psi_{\nu}(p_{\nu}) \big)^{2n+1} \ra
(2n-2)(2n-1)(2n) \psi_{a}(0) \psi_{b}(0) \psi_{c}(0)$ and

\item $\La_{\nu abcd}(p_{\nu}) \big(\psi_{\nu}(p_{\nu}) \big)^{2n+2} \ra
-(2n-2)(2n-1)(2n)(2n+1) \psi_{a}(0) \psi_{b}(0) \psi_{c}(0)
\psi_{d}(0)$.
\end{itemize}

\medskip

Now
\begin{equation}\label{g-alpha-beta}
g_{\al \ov{\be}} = \frac {\pa^{2} \log(-\La)} {\pa z_{\al} \pa \ov{z}_{\be}} =
\frac {\La_{\al \ov{\be}}} {\La} - \frac
{\La_{\al} \La_{\ov{\be}}} {\La^{2}}.
\end{equation}
Multiplying both sides of this equation by $\psi^{2}$, we get
\begin{equation*}
g_{\al \ov{\be}} \psi^{2} = \frac {\La_{\al \ov{\be}} \psi^{2n}} {\La
\psi^{2n-2}} - \frac {(\La_{\al} \psi^{2n-1})
(\La_{\ov{\be}} \psi^{2n-1})} {(\La \psi^{2n-2})^{2}}.
\end{equation*}
It follows that 
\[
\lim_{\nu \ra \infty} g_{\nu \al \ov{\be}} (p_{\nu})
\left(\psi_{\nu}(p_{\nu})\right)^{2} = (2n-2) \psi_{\al}(0) \psi_{\ov{\be}}(0)
\]
which is (i).

\medskip

Differentiating (\ref{g-alpha-beta}) with respect to $z_{\ga}$, we obtain
\begin{equation}\label{der-g-alpha-beta}
\frac {\pa g_{\al \ov{\be}}} {\pa z_{\ga}} = \frac {\La_{\al \ov{\be} \ga}}
{\La} - \left(\frac {\La_{\al \ov{\be}}
\La_{\ga}} {\La^{2}} + \frac{\La_{\al \ga} \La_{\ov{\be}}} {\La^{2}} +
\frac{\La_{\ov{\be} \ga} \La_{\al}} {\La^{2}}
\right) + \frac{2 \La_{\al} \La_{\ov{\be}} \La_{\ga}} {\La^{3}}.
\end{equation}
Multiplying both sides of this equation by $\psi^{3}$, we get
\begin{multline*}
\frac {\pa g_{\al \ov{\be}}} {\pa z_{\ga}} \psi^{3} =   \frac {\La_{\al \ov{\be}
\ga} \psi^{2n+1}} {\La \psi^{2n-2}} -
\left(\frac {(\La_{\al \ov{\be}} \psi^{2n}) (\La_{\ga} \psi^{2n-1})} {(\La
\psi^{2n-2})^{2}} + \frac{(\La_{\al \ga}
\psi^{2n}) (\La_{\ov{\be}}\psi^{2n-1})} {(\La \psi^{2n-2})^{2}} +
\frac{(\La_{\ov{\be} \ga} \psi^{2n}) (\La_{\al}
\psi^{2n-1})} {(\La \psi^{2n-2})^{2}} \right)\\
 + \frac{2 (\La_{\al} \psi^{2n-1}) (\La_{\ov{\be}} \psi^{2n-1}) (\La_{\ga}
\psi^{2n-1})} {(\La \psi^{2n-2})^{3}}.
\end{multline*}
It follows that 
\[
\lim_{\nu \ra \infty} \frac {\pa g_{\nu \al \ov{\be}}} {\pa z_{\gamma}}(p_{\nu})
\psi_{\nu}(p_{\nu})^{3} = -2(2n-2) \psi_{\al}(0) \psi_{\ov{\be}}(p)
\psi_{\gamma}(0)
\]
which is (ii).

\medskip

\no Differentiating (\ref{der-g-alpha-beta}) with respect to $\ov{z}_{\delta}$,
we
obtain
\begin{multline*}
\frac{\pa^{2}g_{\al \ov{\be} }}{\pa z_{\ga}\pa\ov{z}_{\ov{\delta} }}=
\frac{\La_{\al \ov{\be} \ga\ov{\delta} }}{\La}
-\left(\frac{\La_{\al \ov{\be} \ga}\La_{\ov{\delta} }}{\La^{2}}
+\frac{\La_{\al \ov{\be} \ov{\delta} }\La_{\ga}}{\La^{2}}
+\frac{\La_{\al \ga\ov{\delta} }\La_{\ov{\be} }}{\La^{2}}
+\frac{\La_{\ov{\be} \ga\ov{\delta} }\La_{\al }}{\La^{2}}\right)
-\left(\frac{\La_{\al \ov{\be} }\La_{\ga\ov{\delta} }}{\La^{2}}
+\frac{\La_{\al \ga}\La_{\ov{\be} \ov{\delta} }}{\La^{2}}
+\frac{\La_{\al \ov{\delta} }\La_{\ov{\be} \ga}}{\La^{2}}\right)\\
+2\left(\frac{\La_{\al \ov{\be} }\La_{\ga}\La_{\ov{\delta} }}{\La^{3}}
+\frac{\La_{\al \ga}\La_{\ov{\be} }\La_{\ov{\delta} }}{\La^{3}}
+\frac{\La_{\ov{\be} \ga}\La_{\al }\La_{\ov{\delta} }}{\La^{3}}
+\frac{\La_{\al \ov{\delta} }\La_{\ov{\be} }\La_{\ga}}{\La^{3}}
+\frac{\La_{\ov{\be} \ov{\delta} }\La_{\al }\La_{\ga}}{\La^{3}}
+\frac{\La_{\ga\ov{\delta} }\La_{\al }\La_{\ov{\be} }}{\La^{3}}\right)-
\frac{6\La_{\al }\La_{\ov{\be} }\La_{\ga}\La_{\ov{\delta} }}{\La^{4}}.
\end{multline*}
Multiplying both sides by $\psi^{4}$, this equation can be written in a form
where $\La$ is multiplied by $\psi^{2n-2}$
and first, second, third and fourth order derivatives of $\La$ are multiplied by
$\psi^{2n-1}$, $\psi^{2n}$,
$\psi^{2n+1}$ and $\psi^{2n+2}$ respectively. It follows that
\[
\lim_{\nu \ra \infty} \frac{\pa^{2} g_{\nu \al \ov{\be}}} {\pa z_{\ga} \pa
\ov{z}_{\ov{\delta}}}(p_{\nu}) \big( \psi_{\nu}(p_{\nu}) \big)^{4} =
6(2n-2) \psi_{\al}(0) \psi_{\ov{\be}}(0) \psi_{\ga}(0) \psi_{\ov{\delta}}(0)
\]
which is (iii).
\end{proof}

\medskip

To obtain finer asymptotics of the derivatives of $\La_{\nu}$ along
$\{p_{\nu}\}$, we need the following:
\begin{lem} \label{der-psi_nu-by-psi_nu}
Let $1 \leq \al \leq n-1$. Then
\[
\lim_{\nu \ra \infty} \frac{(\psi_{\nu})_{\al}(p_{\nu})}{\psi_{\nu}(p_{\nu})} =
\frac{1}{2}\big(\psi_{\al n}(0) + \psi_{\al \ov n}(0) \big).
\]
\end{lem}
\begin{proof}
Fix a $\nu$ and define the function $f$ on $[0,1]$ by
\begin{equation}\label{f}
f(t) = \psi_{\nu}(tp_{\nu}) = \psi_{\nu}(0, \ldots, 0, - \delta_{\nu} t).
\end{equation}
By Taylor's theorem
\[
f(1) = f(0) + f^{\prime}(0) + \frac{1}{2}f^{\prime \prime}(s)
\]
for some $s \in (0, 1)$. Therefore, by successive application of the Chain rule
to (\ref{f}), we obtain
\begin{equation}\label{psi_nu-p_nu}
\psi_{\nu}(p_{\nu}) = -\delta_{\nu}\big((\psi_{\nu})_{n}(0) + (\psi_{\nu})_{\ov
n}(0) \big) + \frac{\delta_{\nu}^2}{2}\big( (\psi_{\nu})_{n n}(\zeta_{\nu}) + 2
(\psi_{\nu})_{n \ov n}(\zeta_{\nu}) + (\psi_{\nu})_{\ov n \ov n}(\zeta_{\nu})
\big)
\end{equation}
where $\zeta_{\nu} = s p_{\nu}$.

\medskip

Now fix $1 \leq \al \leq n-1$ and define the
function $g$ on $[0, 1]$ by
\begin{equation}\label{g}
g(t) = (\psi_{\nu})_{\al}(tp_{\nu}) = (\psi_{\nu})_{\al}(0, \ldots, 0, -
\delta_{\nu} t).
\end{equation}
By Taylor's theorem
\[
g(1) = g(0) + g^{\prime}(0) + \frac{1}{2} g^{\prime \prime}(s)
\]
for some $s^{\prime} \in (0, 1)$. Therefore, by successive application of the
chain rule to (\ref{g}), we obtain
\begin{equation}\label{der-psi_nu-p_nu}
(\psi_{\nu})_{\al}(p_{\nu}) = -\delta_{\nu} \big((\psi_{\nu})_{\al n}(0) +
(\psi_{\nu})_{\al \ov n}(0)\big) + \frac{\delta_{\nu}^2}{2} \big(
(\psi_{\nu})_{\al
n n}(\eta_{\nu}) + 2 (\psi_{\nu})_{\al n \ov n}(\eta_{\nu}) + (\psi_{\nu})_{\al
\ov
n \ov n}(\eta_{\nu})\big)
\end{equation}
where $\eta_{\nu} = s^{\prime} p_{\nu}$.
It is now evident from (\ref{psi_nu-p_nu}) and (\ref{der-psi_nu-p_nu}), that
\[
\lim_{\nu \ra \infty} \frac{(\psi_{\nu})_{\al}(p_{\nu})}{\psi_{\nu}(p_{\nu})} =
\frac{1}{2}\big( \psi_{\al n}(0) + \psi_{\al \ov n}(0) \big)
\]
and the lemma is proved.
\end{proof}
\no Using this lemma and theorem 1.3, we obtain the following finer asymptotics
of the first and second order derivatives of $\La_{\nu}$ along $\{p_{\nu}\}$.
\begin{lem}\label{finer-2nd-der-La}
Let $1 \leq \al \leq n-1 $ and $1 \leq \be \leq n $. Then
\begin{itemize}
\item [(i)] $\lim_{\nu \ra \infty} \Lambda_{\nu \al}(p_{\nu})
\big(\psi_{\nu}(p_{\nu})\big)^{2n-2} = \la_{\al}(0) + (2n-2) C_{\al}$,

\item [(ii)] $\lim_{\nu \ra \infty} \Lambda_{\al \ov \be}(p_{\nu})
\big(\psi_{\nu}(p_{\nu})\big)^{2n-1} = -(2n-2) \la_{\al}(0) \psi_{\ov \be}(0) -
(2n-2)(2n-1) \psi_{\ov \be}(0) C_{\al} + (2n-2)\psi_{\al \ov \be}(0)$
\end{itemize}
where $C_{\al} = \frac{1}{2} \big(\psi_{\al n}(0) + \psi_{\al \ov n}(0) \big)$.
\end{lem}

\begin{proof}
The normalised robin function
\begin{equation}\label{la_nu-La_nu}
\la(z) =
\begin{cases}
\La(z) \big(\psi(z)\big)^{2n-2} & \text{if } z \in D\\
- \vert \pa \psi(z) \vert^{2n-2} & \text{if } z \in \pa D
\end{cases}
\end{equation}
associated to $(D, \psi)$ is $C^2$ on $\ov D$. In particular, $\la(0) = -1$.
Differentiating $\la$
with respect to $z_{\al}$, we obtain
\[
\Lambda_{\al} \psi^{2n-2} = \lambda_{\al} - (2n-2)
\lambda \psi^{-1} \psi_{\al}.
\]
Hence by theorems 1.2, 1.3 and lemma \ref{der-psi_nu-by-psi_nu},
\[
\lim_{\nu \ra \infty} \Lambda_{\nu \al}(p_{\nu})
\big(\psi_{\nu}(p_{\nu})\big)^{2n-2} = \lambda_{\al}(0) + (2n-2) C_{\al}
\]
which is (i). Similarly differentiating (\ref{la_nu-La_nu}) with respect to
$z_{\al}$ followed by $\ov z_{\be}$ we obtain
\[
\Lambda_{\al \ov \be} \psi^{2n -1} = \lambda_{\al \ov \be} \psi - (2n-2)
(\lambda_{\al} \psi_{\ov \be} + \lambda_{\ov \be} \psi_{\al}) +(2n-2)(2n-1)
\la \psi^{-1} \psi_{\al} \psi_{\ov \be} - (2n-2) \la \psi_{\al \ov \be}
\]
Again by theorems 1.2, 1.3 and lemma \ref{der-psi_nu-by-psi_nu},
\[
\lim_{\nu \ra \infty} \Lambda_{\al \ov \be}(p_{\nu})
\big(\psi_{\nu}(p_{\nu})\big)^{2n-1} = -(2n-2) \la_{\al}(0) \psi_{\ov \be}(0) -
(2n-2)(2n-1) \psi_{\ov \be}(0) C_{\al} + (2n-2)\psi_{\al \ov \be}(0)
\]
which is (ii).
\end{proof}

\begin{lem} \label{finer-2nd-der-g_nu}
Let $1 \leq \al \leq n-1$ and $1 \leq \be \leq n$. Then
\[
\lim_{\nu \ra \infty} g_{\nu \al \ov{\be}}(p_{\nu}) \big(\psi_{\nu}(p_{\nu})
\big) = (2n-2) \left( \frac{1}{2} \big\{ \psi_{\al n}(0) + \psi_{\al \ov{n}}(0)
\big\} \psi_{\ov{\be}}(0)-\psi_{\al\ov{\be}}(0)\right).
\]
\end{lem}
\begin{proof}
We have
\[
g_{\al \ov{\be}} = \frac{\pa ^{2} \log(-\La)}{\pa z_{\al} \pa \ov{z}_{\be}}
= \frac{\La_{\al \ov{\be}}}{\La} - \frac{\La_{\al} \La_{\ov{\be}}}{\La ^{2}}.
\]
Multiplying both sides of this equation by $\psi$, we get
\begin{equation}\label{g_ij-times-psi}
g_{\al \ov{\be}}\psi^{}  = \frac{\La_{\al
\ov{\be}}\psi^{2n-1}}{\La\psi^{2n-2}}-\frac{(\La_{\al}\psi^{2n-2})(\La_{\ov{\be}
}\psi^{2n-1})}{(\La\psi^{2n-2})^{2}}.
\end{equation}
By the proof of lemma \ref{asympt-g_nu}
\[
\La_{\nu}(p_{\nu}) \big( \psi_{\nu}(p_{\nu}) \big)^{2n-2} \ra -1
\]
and
\[
\La_{\nu \ov{\be}}(p_{\nu}) \big( \psi_{\nu}(p_{\nu}) \big)^{2n-1} \ra (2n-2)
\psi_{\ov{\be}}(0).
\]
Therefore using lemma \ref{finer-2nd-der-La} we obtain from
(\ref{g_ij-times-psi}),
\begin{multline*}
\lim_{\nu \ra \infty} g_{\nu \al \ov{\be}}(p_{\nu}) \psi_{\nu}(p_{\nu})
= (2n-2) \la_{\al}(0) \psi_{\ov{\be}}(0) + (2n-2)(2n-1) \psi_{\ov{\be}}(0)C -
(2n-2) \psi_{\al\ov{\be}}(0)\\
\hspace{1cm} - \big \{\la_{\al}(0) + (2n-2) C_{\al} \big\} \big \{(2n-2)
\psi_{\ov{\be}}(0) \big\}
\end{multline*}
Simplifying the right hand side we obtain
\begin{equation*}
\begin{split}
\lim_{\nu \ra \infty} g_{\nu \al \ov{\be}}(p_{\nu}) \psi_{\nu}(p_{\nu})
& = (2n-2) \big( \psi_{\ov{\be}}(0) C_{\al} - \psi_{\al\ov{\be}}(0) \big)\\
& = (2n-2) \Big(\frac{1}{2} \big\{\psi_{\al n}(0) + \psi_{\al \ov{n}}(0) \big\}
\psi_{\ov{\be}}(0) - \psi_{\al
\ov{\be}}(0) \Big).
\end{split}
\end{equation*}
\end{proof}

Since we do not have any information about the third order derivatives of
$\la(p) = \psi^{2n-2} \Lambda(p)$ near the boundary of $D$, the above method
fails to give finer asymptotics of $\Lambda_{\nu \al \ov \be \gamma}$. However
by
proposition \ref{bdy-val-der-g}, the function
\begin{equation}
g(p, w) = \psi(p)^{2n-2} G(p, z)
\end{equation}
where $w = (z-p)/(-\psi(p))$, is $C^2$ up to $\mathscr{D} \cup \pa \mathscr{D}$
and for each $p \in D$, $\frac{\pa g}{\pa p_{\al}}(p)$ and $\frac{\pa^2 g}{\pa
p_{\al} \pa \ov p_{\be}}(p)$ are harmonic functions of $w \in \ov D(p)$ and
hence can be differentiated infinitely often with respect to $w$. Moreover
\begin{equation}
\frac{\pa g}{\pa p_{\al}}(p, 0) = \frac{\pa \la}{\pa p_{\al}}(p) \quad
\text{and} \quad
\frac{\pa^2 g}{\pa p_{\al} \pa \ov p_{\be}}(p,0) = \frac{\pa^2 \la}{\pa p_{\al}
\pa \ov p_{\be}}.
\end{equation}
In the following, we exploit these properties to calculate finer asymptotics of
$\La_{\nu \al \ov \be \gamma}$ by expressing it in terms of mixed derivatives of
$g_{\nu}$.

\medskip

\no By \cite{LY}*{Proposition~6.1}, the functions
\begin{equation}\label{def-G_al}
\begin{cases}
G_{\al}(p, z) & = \quad \Big(\frac{\pa G}{\pa p_{\al}} + \frac{\pa G}{\pa
z_{\al}} \Big) (p, z),\\
G_{\al \ov \be}(p, z) & = \quad \Big(\frac{\pa G_{\al}}{\pa p_{\ov \be}} +
\frac{\pa G_{\al}}{\pa \ov z_{\be}} \Big) (p, z)
\end{cases}
\end{equation}
are real analytic, symmetric function in $D \times D$ and are harmonic in $z$
and in $p$. By \cite{LY}*{6.14}
\begin{equation}\label{La-G}
\La_{\al \ov \be \gamma}(p) = 2 \frac{\pa G_{\al \ov \be}}{\pa z_{\gamma}}(p, p)
\end{equation}

By \cite{LY}*{Proposition~6.2}, the functions
\begin{equation}\label{def-g_0-g_al}
\begin{cases}
g_{0}(p, w) & = \quad g(p, w) + \frac{1}{n-1} \sum_{i=1}^{n} w_i \frac{\pa
g}{\pa w_i},\\
g_{\al}(p, w) & = \quad \psi(p) \frac{\pa g}{\pa p_{\al}}(p, w) -(n-1)
\psi_{\al}(p) \big( g_{0}(p, w) + \ov{g_{0}(p, w)} \big)
\end{cases}
\end{equation}
are harmonic functions of $w \in D(p)$ for each $p \in \ov D$. From
\cite{LY}*{page~83},
\begin{equation}\label{g-G}
\frac{\pa G_{\al \ov \be}}{\pa z_{\gamma}}(p, p) = -
{\big(\psi(p)\big)^{-2n-1}} \Big\{ -2n \psi_{\ov \be}(p) \frac{\pa
g_{\al}}{\pa w_{\gamma}}(p, 0) + \psi(p) \frac{\pa^2 g_{\al}}{\pa
w_{\gamma} \pa \ov p_{\be}}(p, 0)\Big\}
\end{equation}
Combining (\ref{La-G}) and (\ref{g-G}),
\begin{equation}\label{La-mix-der-g}
\Lambda_{\al \ov \be \gamma}(p) \big(\psi(p)\big)^{2n} = 4 n \frac{\psi_{\ov
\be}(p)}{\psi(p)} \frac{\pa g_{\al}}{\pa w_{\gamma}}(p, 0)
-\frac{\pa^2 g_{\al}}{\pa w_{\gamma} \pa \ov p_{\be}}(p, 0)
\end{equation}

\begin{lem}\label{finer-3rd-der-La}
Let $1 \leq \alpha, \gamma \leq n$ and $1 \leq \beta \leq n-1$. Then
\[
\lim_{\nu \ra \infty} \Lambda_{\nu \alpha \ov \beta \gamma} (p_{\nu})
\big( \psi_{\nu}(p_{\nu}) \big)^{2n}
\]
exists and is finite.
\end{lem}
\begin{proof}
By (\ref{La-mix-der-g}) and lemma \ref{der-psi_nu-by-psi_nu}, we only need to
prove that
\[
\lim_{\nu \ra \infty} \frac{\pa g_{\nu \al}}{\pa w_{\gamma}}(p_{\nu}, 0) \quad
\text{and} \quad \lim_{\nu \ra \infty} \frac{\pa^2 g_{\nu \al}}{\pa w_{\gamma}
\pa \ov p_{\be}}(p_{\nu}, 0)
\]
exist and are finite.

\medskip

Now $g_{\nu \al}(p_{\nu}, w)$ is a harmonic function of $w \in D^{\nu}$. To
estimate the boundary values of these functions, note that the first term of
$g_{\nu 0}(p_{\nu}, w)$, i.e., $g_{\nu}(p_{\nu}, w)$ is bounded by $\vert w
\vert^{-2n+2}$ for all $\nu$ and by proposition \ref{estimate-der-g}, the second
term is bounded by $C \vert w \vert ^{-2n+2}$ for all large $\nu$. Therefore,
from (\ref{def-g_0-g_al})
\begin{equation}\label{bound-g_0}
\vert g_{\nu 0}(p_{\nu}, w) \vert \leq C \vert w \vert^{-2n+2}, \quad w \in \pa
D^{\nu}
\end{equation}
for all large $\nu$. Again, by proposition \ref{bdy-est-1-der-g}, $\vert
\frac{\pa g_{\nu}}{\pa p_{\al}} (p_{\nu}, w) \vert$ is bounded by $C (1+\vert w
\vert^{-1})\vert w \vert^{-2n+3}$ for
all lare $\nu$. Also $\psi_{\nu}(p_{\nu})$ and $\psi_{\nu \al}(p_{\nu})$
are bounded by a constant $C$ for all large $\nu$. Hence from
(\ref{def-g_0-g_al}) and (\ref{bound-g_0}),
\begin{equation}\label{bound-g_al}
\vert g_{\nu \al}(p_{\nu}, w)\vert \leq C (1+ \vert w \vert^{-1})\vert w
\vert ^{-2n + 3}, \quad w \in \pa D^{\nu}
\end{equation}
for all large $\nu$.

\medskip

Choose
$r>0$ such that $\ov B(0, r) \subset \mathcal{H}$. Since $D^{\nu}$ converges in
the Hausdorff sense to $\mathcal{H}$, there exists an integer $I$ such that $\ov
B(0, r)
\subset D^{\nu}$ for all $\nu \geq I$. Therefore
\begin{equation}\label{|w|}
\vert w \vert > r
\end{equation}
for all $\nu \geq I$ and $w \in \pa D^{\nu}$. Hence from (\ref{bound-g_al}),
\[
\vert g_{\nu \al}(p_{\nu}, w) \vert \leq C  r^{-2n+3} (1 +
r^{-1}), \quad w \in \pa D^{\nu}
\]
for all large $\nu$. Therefore, $g_{\nu \al}(p_{\nu}, w)$ is uniformly bounded
on $B(0, r)$ for all large $\nu$. Moreover, by \cite{LY}*{Proposition~6.2}
and the fact that $\frac{\pa g_{\nu}}{\pa p_{\al}}(p_{\nu},0) = \frac{\pa
\la_{\nu}}{\pa p_{\al}}(p_{\nu})$,
\begin{equation}\label{g_al(p,0)}
g_{\nu\al}(p_{\nu}, 0)  = \psi_{\nu}(p_{\nu}) \frac{\pa \la_{\nu}}{\pa
p_{\al}}(p_{\nu}) - (2n-2) \psi_{\nu\al}(p_{\nu}) \la(p_{\nu})
\end{equation}
which converges. It follows from Harnack's priciple that
\[
\lim_{\nu \ra \infty} \frac{\pa g_{\nu \al}}{\pa w_{\gamma}}(p_{\nu}, 0)
\]
exists.

\medskip

\no Now differentiating(\ref{def-g_0-g_al}) with respect to $\ov p_{\be}$, we
obtain
\begin{equation}
\frac{\pa g_0}{\pa \ov p_{\be}}(p,w) = \frac{\pa g}{\pa \ov p_{\be}} +
\frac{1}{n-1} \sum_{i=1}^{n} w_i \frac{\pa^2 g}{\pa \ov p_{\be} \pa w_i}
\end{equation}
and
\begin{multline}
\frac{\pa g_{\al}}{\pa \ov p_{\be}}(p,w) = \psi(p) \frac{\pa^2 g}{\pa
p_{\al} \pa \ov p_{\be}}(p, w) + \psi_{\ov \be}(p) \frac{\pa g}{\pa p_{\al}}(p,
w)
-(n-1) \psi_{\al}(p) \Big( \frac{\pa g_0}{\pa \ov p_{\be}}(p,w) + \ov{\frac{\pa
g_0}{\pa p_{\be}}(p,w)}\Big)\\
-(n-1)  \psi_{\al \ov \be}(p) \big(g_0(p, w) + \ov{g_0(p,w)}\big)
\end{multline}
which are harmonic functions of $w \in D$. As above $\big\vert \frac{\pa
g_{\nu}}{\pa \ov p_{\be}} \big \vert$ is bounded by $C (1+\vert w
\vert^{-1})\vert w \vert^{-2n+3}$ for all large $\nu$. Also, by proposition
\ref{estimate-mixed-2nd-der-g}, $\big \vert \frac{\pa^2 g_{\nu}}{\pa \ov p_{\be}
\pa w_i}
\big \vert$ is bounded by $C (1 + \vert w \vert^{-1})\vert w \vert ^{-2n+2}$ for
all large $\nu$. It follows that
\begin{equation}
\Big \vert \frac{\pa g_{\nu 0}}{\pa \ov p_{\be}}(p_{\nu}, w) \Big \vert \leq C
\vert w \vert ^{-2n+3}, w \in \pa D^{\nu}
\end{equation}
for all large $\nu$. From proposition \ref{bdy-val-der-g}, for $1 \leq \gamma
\leq n$, $p \in D$
\[
\Big \vert \frac{\pa^2 g}{\pa p_{\gamma} \pa \ov p_{\gamma}}(p, w) \Big \vert
\leq \vert k_2^{\gamma}(p,w) \vert \vert \pa_w g (p, w) \vert +  2 \vert
k_1^{\gamma} \vert \sum_{i=1}^{n} \Big \vert \frac{\pa^2 g}{\pa w_i \pa \ov
p_{\gamma}} \Big \vert, \quad w \in \pa D(p).
\]
It follows that
\[
\Big \vert \frac{\pa^2 g_{\nu}}{\pa p_{\gamma} \pa \ov p_{\gamma}}(p_{\nu}, w)
\Big \vert\leq  C (1 + \vert w \vert^{-1} + \vert w \vert^{-2}) \vert w \vert
^{-2n + 4}, w \in \pa D^{\nu}
\]
and hence by a unitary change of coordinates
\[
\Big \vert \frac{\pa^2 g_{\nu}}{\pa p_{\al} \pa \ov p_{\be}}(p_{\nu}, w) \Big
\vert \leq C (1 + \vert w \vert^{-1} + \vert w \vert^{-2})\vert w \vert ^{-2n +
4}, w \in \pa D^{\nu}
\]
for all large $\nu$. Thus
\[
\Big \vert \frac{\pa g_{\nu \al}}{\pa \ov p_{\be}}(p_{\nu}, w) \Big \vert \leq C
(1 + \vert w \vert^{-1} + \vert w \vert^{-2})\vert w \vert ^{-2n + 4}\leq C
r^{-2n+4}(1+ r^{-1} + r^{-2}), \quad w \in \pa D^{\nu}
\]
for all large $\nu$. Therefore, the sequence $\{\frac{\pa g_{\nu \al}}{\pa \ov
p_{\be}}(p_{\nu}, w)\}$ is uniformly bounded on $B(0, r)$. Moreover,
\[
\frac{\pa g_{\nu \al}}{\pa \ov p_{\be}}(p_{\nu}, 0) = \psi_{\nu}(p_{\nu})
\frac{\pa^2 \la_{\nu}}{\pa p_{\al} \pa \ov p_{\be}}(p_{\nu}) + \psi_{\nu \ov
\be}(p_{\nu}) \frac{\pa \la_{\nu}}{\pa p_{\al}}(p_{\nu}) - (2n-2) \psi_{\nu
\al}(p_{\nu}) \frac{\pa \la_{\nu}}{\pa \ov p_{\be}}(p_{\nu}) - (2n-2)
\psi_{\nu\al \ov \be}(p_{\nu}) \la_{\nu}(p_{\nu})
\]
which converges. It follows from Harnack's principle that
\[
\lim_{\nu \ra \infty} \frac{\pa^2 g_{\nu \al}}{\pa w_{\gamma}
\pa \ov p_{\be}}(p_{\nu}, 0)
\]
exists.
\end{proof}

\begin{lem}\label{finer-3rd-der-g_nu}
 Let $1\leq \al, \ga \leq n$ and $1\leq \be \leq n-1$. Then
\[
\lim_{\nu \ra \infty}\frac{\pa g_{\nu \al \ov{\be}}}{\pa z_{\ga}}(p_{\nu}) \big(
\psi(p_{\nu}) \big)^{2}
\]
exists and is finite.
\end{lem}
\begin{proof}
From (\ref{der-g-alpha-beta}), we obtain
\begin{multline*}
\frac {\pa g_{\nu \al \ov{\be}}} {\pa z_{\ga}} \psi_{\nu}^{2} = \frac {\La_{\nu
\al \ov{\be}
\ga} \psi_{\nu}^{2n}} {\La_{\nu} \psi_{\nu}^{2n-2}} -
\left(\frac {(\La_{\nu \al \ov{\be}} \psi_{\nu}^{2n-1}) (\La_{\nu \ga}
\psi_{\nu}^{2n-1})} {(\La_{\nu}
\psi_{\nu}^{2n-2})^{2}} + \frac{(\La_{\nu \al \ga}
\psi_{\nu}^{2n}) (\La_{\nu \ov{\be}}\psi_{\nu}^{2n-2})} {(\La_{\nu}
\psi_{\nu}^{2n-2})^{2}} +
\frac{(\La_{\nu \ov{\be} \ga} \psi_{\nu}^{2n-1}) (\La_{\nu \al}
\psi_{\nu}^{2n-1})} {(\La_{\nu} \psi_{\nu}^{2n-2})^{2}} \right)\\
+ \frac{2 (\La_{\nu \al} \psi_{\nu}^{2n-1}) (\La_{\nu \ov{\be}}
\psi_{\nu}^{2n-2}) (\La_{\nu \ga}
\psi_{\nu}^{2n-1})} {(\La_{\nu} \psi_{\nu}^{2n-2})^{3}}.
\end{multline*}
In view of theorem 1.2 and lemma \ref{finer-2nd-der-La} it is seen that the
second and third terms have finite
limits along $\{p_{\nu}\}$ and by lemma \ref{finer-3rd-der-La} the first term
has finite limit along $\{p_{\nu}\}$.
\end{proof}

\begin{lem}\label{finer-det}
The limit
\[
\lim_{\nu \ra \infty} \det \big( g_{\nu \al \ov{\be}}(p_{\nu}) \big) \big(
\psi_{\nu}(p_{\nu}) \big)^{n+1}
\]
exists and is nonzero.
\end{lem}
\begin{proof}
Let $(\Delta_{\al \ov{\be}})$ be the cofactor matrix of $(g_{\al \ov{\be}})$.
Then expanding by the $n$-th row,
\[
\det( g_{\al \ov{\be}} ) = g_{n \ov{1}} \Delta_{n \ov{1}} + \ldots + g_{n
\ov{n}} \Delta_{n \ov{n}}.
\]
Therefore,
\begin{equation}\label{det-times-psi}
\det(g_{\al \ov{\be}}) \psi^{n+1} = (g_{n \ov{1}} \psi^{2}) (\Delta_{n \ov{1}}
\psi^{n-1}) + \ldots + (g_{n \ov{n}}
\psi^{2}) (\Delta_{n \ov{n}} \psi^{n-1}).
\end{equation}
Note that
\begin{align*}
\Delta_{n \ov{\al}} \psi^{n-1}
& = \psi^{n-1} (-1)^{n + \al} \det
\begin{pmatrix}
g_{1 \ov{1}} & \dots & g_{1 \ov{\al-1}} & g_{1 \ov{\al + 1}} & \dots & g_{1
\ov{n}}\\
\hdotsfor[1.5]{6}\\
g_{n-1 \ov{1}} & \dots & g_{n-1 \ov{\al -1}} & g_{n-1\ov{\al+1}} & \dots &
g_{n-1\ov{n}}
\end{pmatrix}\\
& = (-1)^{n+\al} \det
\begin{pmatrix}
g_{1\ov{1}}\psi & \dots & g_{1\ov{\al-1}}\psi & g_{1\ov{\al+1}}\psi & \dots &
g_{1\ov{n}}\psi\\
\hdotsfor[1.5]{6}\\
g_{n-1 \ov{1}}\psi & \dots & g_{n-1\ov{\al-1}}\psi & g_{n-1\ov{\al+1}}\psi &
\dots & g_{n-1\ov{n}}\psi
\end{pmatrix}
\end{align*}
By lemma \ref{finer-2nd-der-g_nu} , if $1 \leq \al \leq n-1$ and $1 \leq \be
\leq n$, then the term $g_{\nu \al \ov{\be}}(p_{\nu})\psi_{\nu}(p_{\nu})$
converges to a finite quantity . It follows that if $1 \leq \al \leq n-1$ then
\[
\lim_{\nu \ra \infty} \Delta_{\nu n \ov{\al}}(p_{\nu}) \big(
\psi_{\nu}(p_{\nu})\big)^{n-1}
\]
exists and is finite. Also if $1 \leq \al, \be \leq n-1$, then $g_{\nu \al
\ov{\be}}(p_{\nu})\psi_{\nu}(p_{\nu})$ converges to $-(2n-2)\psi_{\al
\ov{\be}}(0)$. Therefore
\[
\lim_{\nu \ra \infty} \Delta_{\nu n \ov{n}}(p_{\nu}) \big(\psi_{\nu}(p_{\nu})
\big)^{n-1} = (-1)^{n} (2n-2)^{n} \det \big(\psi_{\al \ov{\be}}(0) \big)_{1\leq
\al, \be \leq n-1}.
\]
Finally by lemma \ref{asympt-g_nu}, if $1 \leq \al, \be \leq n$, then $g_{\nu
\al \ov{\be}}(p_{\nu})\big(\psi_{\nu}(p_{\nu})\big)^{2}$ converges to $(2n-2)
\psi_{\al}(0)
\psi_{\ov{\be}}(0)$. Now it follows from (\ref{det-times-psi}) that
\[
\lim_{\nu \ra \infty} \det \big( g_{\nu \al \ov{\be}}(p_{\nu}) \big) \big(
\psi_{\nu}(p_{\nu}) \big)^{n+1}
= (-1)^{n}(2n-2)^{n+1} \det \big( \psi_{\al \ov{\be}}(0) \big)_{1\leq \al,\be
\leq n-1}\neq 0
\]
as $D$ is strongly pseudoconvex at $0$.
\end{proof}

\medskip

\no {\it Proof of Theorem 1.1:} We have
\[
-\frac {1} {\big( g_{n \ov{n}}(z) \big)^{2}} \frac {\pa^{2} g_{n \ov{n}}} {\pa
z_{n} \pa \ov{z}_{n}}(z)
= -\frac {1} {\Big(g_{n \ov{n}} (z) \big( \psi(z) \big)^{2} \Big)^{2}}
\frac{\pa^{2} g_{n \ov{n}}} {\pa z_{n} \pa
\ov{z}_{n}}(z) \big( \psi(z) \big)^{4}
\]
By lemma \ref{asympt-g_nu}
\[
-\frac {1} {\big( g_{\nu n \ov{n}}(p_{\nu}) \big)^{2}} \frac {\pa^{2} g_{\nu n
\ov{n}}} {\pa z_{n} \pa \ov{z}_{n}}(p_{\nu})
 \ra -\frac {1} {\big\{ (2n-2) \psi_{n}(0) \psi_{\ov{n}}(0) \big\}^{2} } \big\{
6(2n-2) \psi_{n}(0) \psi_{\ov{n}}(0) \psi_{n}(0) \psi_{\ov{n}}(0) \big\} =
-\frac{3}{n-1}.
\]

\medskip

\no To compute the limit of the second term note that $g^{\be \ov{\al}} =
\Delta_{\al \ov{\be}} / \det(g_{\al
\ov{\be}})$. There are various cases to be considered depending on $\al$ and
$\be$.

\medskip

\no \textit{Case 1}:  $\al \neq n$, $\be \neq n$. Here
\[
\frac{1}{g_{n \ov{\al}}^{2}} g^{\be \ov{\al}} \frac{\pa g_{n \ov{\al}}}{\pa
z_{n}} \frac{\pa g_{\be \ov{n}}}{\pa
\ov{z}_{n}}
= \frac{1}{(g_{n\ov{n}}\psi^{2})^{2} \big( \det(g_{i \ov{j}})\psi^{n+1}
\big)}(\Delta_{\al \ov{\be}} \psi^{n}) \bigg(
\frac{\pa g_{n \ov{\al}}}{\pa z_{n}}\psi^{2} \bigg) \left(\frac{\pa g_{\be
\ov{n}}}{\pa\ov{z}_{n}}\psi^{3}\right)
\]
By lemma \ref{asympt-g_nu},
\[
g_{\nu n \ov{n}}(p_{\nu}) \big(\psi_{\nu}(p_{\nu}) \big)^{2} \ra (2n-2)
\]
By lemma \ref{finer-det}, $\det \big( g_{i \ov{j}}(p_{\nu}) \big) \big(
\psi_{\nu}p_{\nu}) \big)^{n+1}$ converges to a nonzero finite. Also
\[
\Delta_{\al \ov{\be}} =
\sum_{\sigma}(-1)^{sgn(\sigma)}g_{1\ov{\sigma(1)}}g_{2\ov{\sigma(2)}}\cdots
g_{n\ov{\sigma(n)}}
\]
where the summation runs over all permutations
\[
\sigma:\{1,\ldots,\al-1,\al+1,\ldots,n\}\ra \{1,\ldots,\be-1,\be+1,\ldots,n\}
\]
Hence
\[
\Delta_{\al \ov{\be}}\psi^{n} =
\sum_{\sigma}(-1)^{sgn(\sigma)}(g_{1\ov{\sigma(1)}}\psi)(g_{2\ov{\sigma(2)}}
\psi)\cdots
(g_{n\ov{\sigma(n)}}\psi^{2}).
\]
By lemma \ref{finer-2nd-der-g_nu}, for $1 \leq i \leq n-1$, $g_{\nu
i\ov{\sigma(i)}}(p_{\nu}) \big( \psi_{\nu}(p_{\nu}) \big)$ converges to a finite
quantity. Also
\[
g_{\nu n\ov{\sigma(n)}}(p_{\nu}) \big( \psi(p_{\nu}) \big)^{2} \ra (2n-2)
\psi_{n}(0)\psi_{\ov{\sigma(n)}}(0)
\]
by lemma \ref{asympt-g_nu}. Thus $\Delta_{\nu \al \ov{\be}}(p_{\nu}) \big(
\psi_{\nu}(p_{\nu}) \big)^{n}$ converges to a finite quantity.

\medskip

\no By lemma \ref{finer-3rd-der-g_nu}, $\frac{\pa g_{\nu n \ov{\al}} } {\pa
z_{n}}(p_{\nu}) \big( \psi_{\nu}(p_{\nu}) \big)^{2}$ converges to a finite
quantity and by lemma \ref{asympt-g_nu}
\[
\frac{\pa g_{\nu \be \ov{n}}} {\pa \ov{z}_{n}}(p_{\nu}) \big(
\psi_{\nu}(p_{\nu}) \big)^{3} = \ov{\bigg(\frac{\pa g_{\nu n \ov{\be}}} {\pa
z_{n}}(p_{\nu}) \big( \psi_{\nu}(p_{\nu}) \big)^{3} \bigg)}
\ra -2(2n-2) \ov{\big( \psi_{\nu n}(0) \big)} \, \ov{\big( \psi_{\ov{\be}}(0)
\big)}
\,  \ov{\big( \psi_{n}(0) \big)} = 0.
\]
Hence
\[
\lim_{\nu \ra \infty} \frac {1} {\big( g_{\nu n\ov{n}}(p_{\nu}) \big)^{2}}
g_{\nu}^{\be \ov{\al}}(p_{\nu}) \frac {\pa g_{\nu n\ov{\al}}} {\pa
z_{n}}(p_{\nu}) \frac {\pa g_{\nu \be \ov{n}}} {\pa\ov{z}_{n}}(p_{\nu}) = 0.
\]

\medskip

\no \textit{Case 2}: $\al = n$, $\be \neq n$. Here
\[
\frac{1}{g_{n \ov{n}}^{2}} g^{\be \ov{n}} \frac{\pa g_{n \ov{n}}}{\pa z_{n}}
\frac{\pa g_{\be \ov{n}}}{\pa \ov{z}_{n}}
= \frac{1}{(g_{n \ov{n}} \psi^{2})^{2} (\det(g_{i \ov{j}}) \psi^{n+1})}
(\Delta_{n \ov{\be}} \psi^{n-1}) \bigg(
\frac{\pa g_{n \ov{n}}}{\pa z_{n}} \psi^{3} \bigg) \bigg( \frac{\pa g_{\be
\ov{n}}}{\pa \ov{z}_{n}} \psi^{3} \bigg)
\]
By lemma \ref{asympt-g_nu},
\[
g_{\nu n \ov{n}}(p_{\nu}) \big( \psi_{\nu}(p_{\nu}) \big)^{2} \ra (2n-2).
\]
By lemma \ref{finer-det}, $\det \big( g_{\nu \al \ov{\be}}(p_{\nu}) \big) \big(
\psi(p_{\nu}) \big)^{n+1}$ has a nonzero limit and $\Delta_{\nu n
\ov{\be}}(p_{\nu}) \big( \psi_{\nu}(p_{\nu}) \big)^{n-1}$ converges to a finite
quantity.
By lemma \ref{asympt-g_nu},
\[
\frac{\pa g_{\nu n \ov{n}}} {\pa z_{n}}(p_{\nu}) \big( \psi_{\nu}(p_{\nu})
\big)^{3} \ra -2(2n-2) \psi_{n}(0) \psi_{\ov{n}}(0) \psi_{n}(0) = -2(2n-2)
\]
and
\[
\frac{\pa g_{\nu \be \ov{n}}} {\pa \ov{z}_{n}}(p_{\nu}) \big(
\psi_{\nu}(p_{\nu}) \big)^{3} = \ov{\bigg(\frac{\pa g_{\nu n \ov{\be}}} {\pa
z_{n}}(p_{\nu}) \big( \psi_{\nu}(p_{\nu}) \big)^{3} \bigg)}
\ra -2(2n-2) \ov{\big( \psi_{n}(0) \big)} \, \ov{\big( \psi_{\ov{\be}}(0) \big)}
\,  \ov{\big( \psi_{n}(0) \big)} = 0.
\]
Hence
\[
\lim_{\nu \ra \infty} \frac {1} {\big( g_{\nu n \ov{n}}(p_{\nu}) \big)^{2}}
g_{\nu}^{\be \ov{n}}(p_{\nu}) \frac {\pa g_{\nu n\ov{n}} } {\pa z_{n}}(p_{\nu})
\frac{\pa g_{\nu \be \ov{n}}} {\pa\ov{z}_{n}}(p_{\nu})  = 0
\]

\medskip

\no \textit{Case 3}: $\al \neq n$ and $\be = n$. This case is similar to Case 2
and we
have
\[
\lim_{\nu \ra \infty }\frac{1}{\big(g_{\nu n \ov{n}}(p_{\nu}) \big)^{2}}
g_{\nu}^{n\ov{\al}}(p_{\nu}) \frac{\pa g_{\nu n\ov{\al}}}{\pa z_{n}}(p_{\nu})
\frac{\pa g_{\nu n\ov{n}}}{\pa\ov{z}_{n}}(p_{\nu}) = 0.
\]

\medskip

\no \textit{Case 4}: $\al = n$, $\be = n$. In this case we have
\[
\frac{1}{g_{n\ov{n}}^{2}}g^{n\ov{n}}\frac{\pa g_{n\ov{n}}}{\pa z_{n}}\frac{\pa
g_{n\ov{n}}}{\pa\ov{z}_{n}}
= \frac{1}{(g_{n\ov{n}}\psi^{2})^{2} \big( \det(g_{i \ov{j}})\psi^{n+1}
\big)}(\Delta_{n\ov{n}}\psi^{n-1})
\bigg( \frac{\pa g_{n\ov{n}}}{\pa z_{n}}\psi^{3} \bigg) \bigg( \frac{\pa
g_{n\ov{n}}}{\pa\ov{z}_{n}}\psi^{3} \bigg).
\]
From lemma \ref{asympt-g_nu},
\[
g_{\nu n\ov{n}}(p_{\nu}) \big(\psi_{\nu}(p_{\nu}) \big)^{2} \ra (2n-2)
\]
and both
\[
\frac {\pa g_{\nu n \ov{n}} } {\pa z_{n}}(p_{\nu}) \big( \psi_{\nu}(p_{\nu})
\big)^{3}, \frac{\pa g_{n \ov{n}} } {\pa \ov{z}_{n}}(p_{\nu}) \big(
\psi_{\nu}(p_{\nu}) \big)^{3} \ra -2(2n-2).
\]
From lemma \ref{finer-det}
\[
\Delta_{\nu n \ov{n}}(p_{\nu}) \big( \psi_{\nu}(p_{\nu}) \big)^{n-1} \ra
(-1)^{n} (2n-2)^{n} \det \big(\psi_{i \ov{j}}(0) \big)_{1 \leq i, j
\leq n-1}
\]
and
\[
\det \big( g_{\nu i \ov{j}}(p_{\nu}) \big) \big( \psi_{\nu}(p_{\nu}) \big)^{n+1}
\ra  (-1)^{n} (2n-2)^{n+1} \det \big( \psi_{i \ov{j}}( 0) \big)_{1 \leq i, j
\leq n-1}.
\]
Hence
\[
\lim_{\nu \ra \infty} \frac {1} {\big(g_{\nu n \ov{n}}(p_{\nu}) \big)^{2}}
g_{\nu}^{n \ov{n}}(p_{\nu}) \frac{\pa g_{\nu n \ov{n}}} {\pa z_{n}}(p_{\nu})
\frac{\pa g_{\nu n \ov{n}}} {\pa\ov{z}_{n}}(p_{\nu})
= \frac{2}{n-1}.
\]

\medskip

From the various cases we finally obtain
\[
\lim_{\nu \ra \infty} R \big(z_{\nu} ,v_{N}(z_{\nu}) \big)=-3/(n-1) + 2/(n-1) =
-1/(n-1).
\]

\section{Existence of closed geodesics}
\no In this section we prove theorem 1.4. The main tool that we will use is the
following theorem of Herbort \cite{Her}:

\begin{thm}\label{herbort}
Let $G$ be a bounded domain in $\mathbf{R}^{k}$, such
that $\pi_1(G)$ is nontrivial. Assume that the following conditions are
satisfied:
\begin{enumerate}
\item [(i)] For each $p \in \ov G$ there is an open neighbourhood $U \subset
\mathbf{R}^{k}$, such that the set $G \cap U$ is simply connected.
\item [(ii)] $G$ is equipped with a complete Riemannian metric $g$ which
possesses the following property:
\begin{itemize}
\item [(P)] For each $S>0$ there is a $\de >0$, such that for every point $p \in
G$ with $d(p, \pa D) < \de$ and every $X \in \mathbf{R}^{k}$, $g(p, X) \geq S
\vert X \vert^2$.
\end{itemize}
\end{enumerate}
Then every nontrivial homotopy class in $\pi_1(G)$ contains a closed geodesic
for $g$.
\end{thm}

In \cite{BV} we proved the following boundary behaviour of the $\La$-metric: Let
$D$ be a $C^{\infty}$-smoothly bounded strongly peudoconvex domain in
$\mathbf{C}^{n}$ and $ds^2$ be the $\La$-metric on $D$. Suppose that $\psi$ is
any $C^{\infty}$-smooth defining function for $D$. Then
\begin{equation*}
ds^2_z(v, v) \approx \delta^{-2}(z) \vert v_N(z)\vert^2 +
\delta^{-1}(z) \mathscr{L}_{\psi}\big( \pi(z), v_{H}(z) \big)
\end{equation*}
uniformly for all $z$ sufficiently close to $\pa D$ and all $v \in
\mathbf{C}^{n}$. Here, $v = v_H(z) + v_N(z)$ is as usual the decomposition of
$v$ at the point $\pi(z) \in \pa D$, $\de(z) = d(z, \pa D)$ is the Euclidean
distance of $z$ to the boundary of $D$ and
$\mathscr{L}_{\psi}(z, v)$ denotes the
Levi form of $\psi$ at $z$ along $v$, i.e.,
\[
\mathscr{L}_{\psi}(z, v) = \sum_{\al, \be =1}^{n} \frac{\pa ^2 \psi}{\pa z_{\al}
\pa \ov z_{\be}}(z) v^{\al} \ov v^{\be}.
\]
Also, it is known that
the Bergman metric $ds^2_B$ on $D$ has the same boundary behaviour. It follows
that
\[
ds^2_z(v, v) \approx ds^2_{Bz}(v, v)
\]
uniformly for all $z$ sufficiently close to $\pa D$ and all $v \in
\mathbf{C}^{n}$. Also, on compact subsets of $D$, these two metrics are
uniformly comparable to the Euclidean metric. Thus we have the follwoing:
\begin{prop}\label{La-Ber}
Let $D$ be a $C^{\infty}$-smoothly bounded strongly pseudoconvex domain in
$\mathbf{C}^{n}$. Let $ds^2$ denotes the $\La$-metric on $D$ and $ds^2_B$
denotes the Bergman metric on $D$. Then there exists a constant $C>1$ such that
\[
C^{-1} ds^2_B \leq ds^2 \leq C ds^2_B
\]
uniformly on $D$.
\end{prop}

\medskip

\no \textit{Proof of theorem 1.4}. We will show that the $\La$-metric on $D$
satisfies the hypothesis of theorem \ref{herbort}.  Indeed, since $\pa D$ is
smooth, condition (i) is evidently satisfied. Also, note that the
Bergman metric is complete on $D$ (\cite{Ohs81}) and satisfies property (P)
\cite{Died}. It follows from proposition \ref{La-Ber} that condition
(ii) is satisfied. Thus the theorem is proved.

\section{$L^2$-cohomology of the $\La$-metric}

\no Let $M$ be a complete K\"{a}hler manifold of complex dimension $n$. Let
$\Om^{i}_{2}$ be the space of square integrable $i$-forms on $M$. Then the
(reduced) $L^2$-cohomology of the complex
\[
\Om^{0}_{2}(M) \xrightarrow{d_0} \Om^{1}_{2}(M) \xrightarrow{d_1} \cdots
\xrightarrow{d_{2n-1}} \Om^{2n}_{2}(M) \xrightarrow{d_{2n}} 0
\]
is defined by
\[
H^{i}_{2}(M) = \frac{\text{ker } {d_i}}{\, \ov{\text{Im}} \, d_{i-1}}
\]
where the closure is taken in $L^2$. Now, let $\mathcal{H}^{i}_{2}(M)$ be the
space of square integrable harmonic $i$-forms on $M$. Then the completeness of
the metric implies that $H^{i}_{2}(M) \cong \mathcal{H}^{i}_{2}(M)$. We have the
following result (\cite{Don}) on the vanishing of the
$L^2$-cohomology outside the middle dimension:
\begin{prop}\label{van-coh-kah}
Let $M$ be a complete K\"{a}hler manifold of complex dimension $n$. Suppose that
the K\"{a}hler form $\om$ of $M$ can be written as $\om
= d \eta$, where $\eta$ is bounded in supremum norm. Then
$\mathcal{H}_{2}^{i}(M) = 0$ for $i \neq n$.
\end{prop}
\no Also, we have the following result (\cite{Ohs89}) on the infinite
dimensionality of the $L^2$-cohomology of the middle dimension:

\begin{thm}\label{inf-coh}
Let $D$ be a domain in a connected complex manifold of dimension $n$ and $ds^2$
be a Hermitian metric on $D$. Suppose that there exists a non-degenerate regular
boundary point $z_0 \in \pa D$. Also, suppose that there exist a neighbourhodd
$U$ of $z_0$, a local defining function $\phi$ for $D$ defined on $U$ and a
Hermitian metric $ds^2_U$ defined on $U$ such that
\[
C^{-1} ds^2 < (-\phi)^{-a} ds^2_U + (-\phi)^{-b} \pa \phi \ov \pa \phi < C ds^2
\]
on $U \cap D$, where $a$, $b$ and $C$ are positive numbers with $1 \leq a \leq b
< a+3$. Then, for any positive integer $p$ and $q$ with $p + q =n$,
\[
\dim H_2^{p,q}(D) = \infty
\]
where $H_2^{p,q}(D)$ denotes the $L^2$ $\ov \pa$-cohomology group relative to
$ds^2$.
\end{thm}

\begin{rem}\label{inf-coh2}
The above theorem in particular implies that if $ds^2$ is complete and
K\"{a}hler, then for any positive integer $p$ and $q$ with $p + q =n$,
\[
\dim \mathcal{H}^{p, q}_{2}(D) = \infty
\]
where $\mathcal{H}^{p, q}_{2}(D)$ is the space of square integrable harmonic
$(p, q)$-forms on $D$ relative to $ds^2$.
\end{rem}

To apply these results to the $\La$ metric, let $D$ be a $C^{\infty}$-smoothly
bounded pseudoconvex domain in $\mathbf{C}^{n}$ and
$ds^2$ be the $\La$-metric on $D$. Then the K\"{a}hler form $\om$ of $ds^2$ is
given by
\[
\om = i \sum_{\al=1}^{n} \frac{\pa^2 \log(-\La)}{\pa z_{\al} \pa \ov z_{\be}}
\, dz_{\al} \wedge d\ov z_{\be} = d \eta
\]
where
\[
\eta = - i \sum_{\al=1}^{n} \frac{\pa \log(-\La)}{\pa z_{\al}} \, dz_{\al}.
\]
Now let $\psi$ be a $C^{\infty}$-smooth defining function for $D$. Then,
differentiating the relation
\[
\la = \La \psi^{2n-2}
\]
with respect to $z_{\al}$ we obtain
\begin{equation}\label{der-log-La}
\frac{\pa \log(-\La)}{\pa z_{\al}} = \la^{-1} \la_{\al} - 2(n-1) \psi^{-1}
\psi_{\al}.
\end{equation}
Therefore,
\begin{equation}\label{eta-v}
\eta (v) = -i \sum_{\al=1}^{n}\frac{\pa \log(-\La)}{\pa z_{\al}} v^{\al} = -i
\big( \la^{-1} \langle v, \ov \pa \la \rangle - 2(n -1) \psi^{-1} \langle v, \ov
\pa \psi \rangle \big)
\end{equation}
and
\begin{equation}\label{eta-v^2}
\vert \eta(v) \vert^2 = \la^{-2} \vert \langle v, \ov \pa \la \rangle \vert^2 -
4 (n-1) \la^{-1} \psi^{-1} \Re \Big( \langle v, \ov \pa \la \rangle \ov{\langle
v, \ov \pa \psi \rangle} \Big) + 4(n-1)^2 \psi^{-2} \vert \langle v, \ov \pa
\psi \rangle \vert^2
\end{equation}

Also, differentiating (\ref{der-log-La}) with respect to $\ov z_{\be}$ we obtain
\begin{equation}\label{metric-comp}
\frac{\pa^2 \log(-\La)}{\pa z_{\al} \pa \ov z_{\be}} = \la^{-1} \la_{\al \ov
\be} - \la^{-2} \la_{\al} \la_{\ov \be} + 2(n-1) \psi^{-2} \psi_{\al} \psi_{\ov
\be} - 2(n-1) \psi^{-1} \psi_{\al \ov \be}.
\end{equation}
Therefore,
\begin{multline}\label{g(v,v)}
ds^2(v, v) = \sum_{\al, \be =1}^{n}\frac{\pa^2 \log(-\La)}{\pa z_{\al} \pa \ov
z_{\be}} v^{\al} \ov v ^{\be}
\\= \la^{-1} \mathscr{L}_{\la}(z, v) - \la^{-2} \vert \langle v, \ov \pa \la
\rangle \vert^2 + 2(n-1) \psi^{-2} \vert \langle v, \ov \pa \psi \rangle \vert^2
- 2(n-1) \psi^{-1} \mathscr{L}_{\psi}(z, v)
\end{multline}

\begin{lem}\label{limit-ratio}
Let $D$ be a $C^{\infty}$-smoothly bounded pseudoconvex domain in $\mathbf{C}^n$
and $\psi$ be a $C^{\infty}$ smooth defining function for $D$. One has the
following:
\begin{enumerate}
\item If $z_0 \in \pa D$ and $v \in \mathbf{C}^{n}$ is a unit vector satisfying
$\big\langle v, \pa \ov \psi (z_0) \big\rangle \neq 0$, then
\[
\lim_{z \ra z_0} \frac{\vert \eta_{z}(v) \vert^2}{ds^2_z(v, v)} = 2(n-1),
\]
\item If $z_0 \in \pa D$ is strongly pseudoconvex and $v \in \mathbf{C}^{n}$ is
a unit vector satisfying $\big\langle v, \pa \ov \psi (z_0) \big\rangle = 0$,
then
\[
\lim_{z \ra z_0} \frac{\vert \eta_{z}(v)\vert^2}{ds^2_z(v, v)} = 0.
\]
\end{enumerate}
Moreover, the limits are apporached uniformly for $z_0 \in \pa D$ and unit
vectors $v$.
\end{lem}

\begin{proof}
Since $\la$ is $C^2$-smooth up to $\ov D$ and $\psi$ is $C^{\infty}$-smooth, the
terms
\[
\big\langle v, \pa \ov \la(z) \big\rangle, \quad  \big\langle v, \pa \ov \psi(z)
\big\rangle, \quad \mathscr{L}_{\la}(z, v), \quad \text{and} \quad
\mathscr{L}_{\psi}(z, v)
\]
are uniformly bounded for all $z \in \ov D$ and all $v \in \mathbf{C}^{n}$ with
$\vert v \vert = 1$. Also, since $\la = - \vert \pa \psi \vert^{2n-2}$ on $\pa
D$, it is evident that $\la^{-1}$ is bounded near $\pa D$.

\medskip

By the above observation it is evident from (\ref{eta-v^2}) that
\[
\lim_{z \ra z_0} \big(\psi(z) \big)^2 \vert \eta_z (v) \vert^2 = 4(n-1)^2
\big\vert \big\langle v, \ov \pa \psi(z_0) \big\rangle \big\vert^2
\]
and from (\ref{g(v,v)}) that
\[
\lim_{z \ra z_0} \big(\psi(z) \big)^2 ds^2_z(v, v) = 2(n-1) \big\vert
\big\langle v, \ov \pa \psi(z_0) \big\rangle \big\vert^2
\]
uniformly for $z_0 \in \pa D$ and unit vector $v$. Therefore,
\[
\lim_{z \ra z_0} \frac{\vert \eta_{z}(v) \vert^2}{ds^2_z(v, v)} = 2(n-1)
\]
uniformly for $z_0 \in \pa D$ and unit vector $v$ satisfying $\big \langle v,
\ov \pa \psi (z_0) \big \rangle \neq 0$, which proves (1).

\medskip

To prove (2), observe that if  $\big \langle v, \ov \pa \psi (z_0) \big \rangle
= 0$ then
\[
\big\langle v, \ov \pa \psi (z) \big\rangle = \big\langle v, \ov \pa \psi (z)
\big\rangle  - \big \langle v, \ov \pa \psi (z_0) \big \rangle = \big\langle v,
\ov \pa \psi (z) - \ov \pa \psi (z_0) \big\rangle.
\]
Since
\[
\big\vert \ov \pa \psi(z) - \ov \pa \psi (z_0) \big\vert \lesssim \big(-\psi(z)
\big)
\]
uniformly for $z$ near $z_0$, it follows that
\[
\big \vert \big\langle v, \ov \pa \psi (z) \big\rangle \big \vert \lesssim
\big(-\psi(z) \big)
\]
uniformly for $z$ near $z_0$ and unit vectors $v$ satisfying $\big \langle v,
\ov \pa \psi (z_0) \big \rangle = 0$. Combining this with our previous
observation, it now follows from (\ref{eta-v^2}) that
\[
\lim_{z \ra z_0} \big(-\psi(z) \big) \vert \eta_z (v) \vert^2 = 0
\]
and from (\ref{g(v,v)}) that
\[
\lim_{z \ra z_0} \big(-\psi(z) \big) ds^2_z(v, v) = 2(n-1)
\mathscr{L}_{\psi}(z_0, v)
\]
uniformly for $z_0 \in \pa D$ and unit vectors $v$ satisfying $\big \langle v,
\ov \pa \psi (z_0) \big \rangle = 0$. Since $z_0$ is a strongly pseudoconvex
boundary point, $\mathscr{L}_{\psi}(z_0, v) > 0$ and hence
\[
\lim_{z \ra z_0} \frac{\vert \eta_{z}(v) \vert^2}{ds^2_z(v, v)} = 0
\]
uniformly for all strongly pseudoconvex boundary points $z_0 \in \pa D$ and all
unit vectors $v$ satisfying $\big \langle v, \ov \pa \psi (z_0) \big \rangle =
0$, which proves (2).
\end{proof}

\begin{prop}\label{ratio-bdd}
Let $D$ be a $C^{\infty}$-smoothly bounded strongly pseudoconvex domain in
$\mathbf{C}^{n}$. Then the ratio
\begin{equation}\label{ratio}
\frac{\vert \eta_z (v) \vert^2}{ds^2_z(v, v)}
\end{equation}
is uniformly bounded for $z \in D$ and vectors $v \in \mathbf{C}^{n}$ with $v
\neq 0$.
\end{prop}
\begin{proof}
By lemma \ref{limit-ratio}, the ratio
\[
\frac{\vert \eta_z (v) \vert^2}{ds^2_z(v, v)}
\]
is uniformly bounded for all $z$ near $\pa D$ and all unit vectors $v$. It is
evident that this ratio is uniformly bounded for all $z$ on a compact subset of
$D$ and all unit vectors $v$. Now, by homogenity of $\eta_z(v)$ and
$ds^2_z(v,v)$ in the vector variable $v$ it follows that the ratio is uniformly
bounded above for all $z \in D$ and vectors $v \neq 0$.
\end{proof}
\no We also note the following:
\begin{prop}\label{La-E}
Let $D$ be a $C^{\infty}$-smoothly bounded strongly pseudoconvex domain in
$\mathbf{C}^{n}$ and $ds^2$ be the $\La$-metric on $D$. Suppose that $\psi$ is a
$C^{\infty}$-smooth defining function for $D$. Then
\[
ds^2 \approx (-\psi)^{-1} ds^2_E+(-\psi)^{-2} \pa \psi \ov \pa \psi
\]
uniformly near $\pa D$, where $ds^2_E$ is the Euclidean metric on
$\mathbf{C}^{n}$.
\end{prop}
\begin{proof}
It is known that the Bergman metric on $D$ satisfies the same estimate.
Therefore, the proof follows from proposition \ref{La-Ber}.
\end{proof}

\medskip

\no \textit{Proof of theorem 1.5}. Let $ds^2$ be the $\La$-metric on $D$. By
proposition \ref{La-Ber} and the completeness of the Bergman metric on $D$,
$ds^2$ is complete. Therefore, by propositions \ref{van-coh-kah} and
\ref{ratio-bdd}, we have
\[
\mathcal{H}_2^i(D) = 0
\]
for $i \neq n$ and hence
\[
\mathcal{H}_2^{p,q}(D) = 0
\]
for $p+q \neq n$. Also, by remark \ref{inf-coh2} and proposition \ref{La-E},
\[
\dim \mathcal{H}_2^{p, q}(D) = \infty
\]
for any positive integers $p$ and $q$ with $p+q =n$. Moreover, a harmonic
$(n, 0)$ form on $D$ is precisely of the form
\[
f(z) \, dz_1 \wedge \ldots \wedge dz_n
\]
where $f(z)$ is a harmonic function (with respect to the standard Laplacian) on
$D$. Therefore, $\mathcal{H}^{n,0}_2(D)$ (and thus $\mathcal{H}^{0,n}_2(D)$) is
isomorphic to the space of square integrable harmonic functions on $D$ which is
evidently infinite dimensional. This completes the proof. \hfill $\square$


\begin{bibdiv}
\begin{biblist}

\bib{BV}{article}{
AUTHOR = {Borah, D.},
author = {Verma, K.},
TITLE = {Remarks on the metric induced by the {R}obin function},
Note = {http://www.iumj.indiana.edu/IUMJ/forthcoming.php},
NOTE = {To appear in Indiana Univ. Math. J.},
eprint = {http://www.iumj.indiana.edu/IUMJ/forthcoming.php}
%{arXiv:1001.5101},
}

\bib{Died}{article}{
   author={Diederich, K.},
   title={Das Randverhalten der Bergmanschen Kernfunktion und Metrik in
   streng pseudo-konvexen Gebieten},
   language={German},
   journal={Math. Ann.},
   volume={187},
   date={1970},
   pages={9--36},
   issn={0025-5831},
   review={\MR{0262543 (41 \#7149)}},
}

\bib{Don}{article}{
   author={Donnelly, H.},
   title={$L_2$ cohomology of pseudoconvex domains with complete K\"ahler
   metric},
   journal={Michigan Math. J.},
   volume={41},
   date={1994},
   number={3},
   pages={433--442},
   issn={0026-2285},
   review={\MR{1297700 (95h:32007)}},
   doi={10.1307/mmj/1029005071},
}

\bib{DonFef}{article}{
   author={Donnelly, H.},
   author={Fefferman, C.},
   title={$L^{2}$-cohomology and index theorem for the Bergman metric},
   journal={Ann. of Math. (2)},
   volume={118},
   date={1983},
   number={3},
   pages={593--618},
   issn={0003-486X},
   review={\MR{727705 (85f:32029)}},
   doi={10.2307/2006983},
}

\bib{Gro}{article}{
   author={Gromov, M.},
   title={K\"ahler hyperbolicity and $L_2$-Hodge theory},
   journal={J. Differential Geom.},
   volume={33},
   date={1991},
   number={1},
   pages={263--292},
   issn={0022-040X},
   review={\MR{1085144 (92a:58133)}},
}

\bib{Her}{article}{
   author={Herbort, G.},
   title={On the geodesics of the Bergman metric},
   journal={Math. Ann.},
   volume={264},
   date={1983},
   number={1},
   pages={39--51},
   issn={0025-5831},
   review={\MR{709860 (85f:32030)}},
   doi={10.1007/BF01458049},
}

\bib{LY}{article}{
   author={Levenberg, N.},
   author={Yamaguchi, H.},
   title={The metric induced by the Robin function},
   journal={Mem. Amer. Math. Soc.},
   volume={92},
   date={1991},
   number={448},
   pages={viii+156},
   issn={0065-9266},
   review={\MR{1061928 (91m:32017)}},
}

\bib{Ohs81}{article}{
   author={Ohsawa, T.},
   title={A remark on the completeness of the Bergman metric},
   journal={Proc. Japan Acad. Ser. A Math. Sci.},
   volume={57},
   date={1981},
   number={4},
   pages={238--240},
   issn={0386-2194},
   review={\MR{618233 (82j:32053)}},
}

\bib{Ohs89}{article}{
   author={Ohsawa, T.},
   title={On the infinite dimensionality of the middle $L^2$ cohomology
   of complex domains},
   journal={Publ. Res. Inst. Math. Sci.},
   volume={25},
   date={1989},
   number={3},
   pages={499--502},
   issn={0034-5318},
   review={\MR{1018512 (90i:32016)}},
   doi={10.2977/prims/1195173354},
}

\bib{Y}{article}{
   author={Yamaguchi, H.},
   title={Variations of pseudoconvex domains over ${\mathbf C}^n$},
   journal={Michigan Math. J.},
   volume={36},
   date={1989},
   number={3},
   pages={415--457},
   issn={0026-2285},
   review={\MR{1027077 (90k:32059)}},
   doi={10.1307/mmj/1029004011},
}

\end{biblist}
\end{bibdiv}
\end{document}